\documentclass[12pt,a4paper]{amsart}
\usepackage{fullpage,verbatim,amssymb,mathtools,xcolor,booktabs}
\usepackage[breaklinks]{hyperref}
\hypersetup{colorlinks=true,urlcolor={blue!80!black},citecolor={blue!80!black},linkcolor={blue!80!black}}
\usepackage[active]{srcltx}

\makeatletter
\let\@@pmod\pmod
\DeclareRobustCommand{\pmod}{\@ifstar\@pmods\@@pmod}
\def\@pmods#1{\mkern4mu({\operator@font mod}\mkern 6mu#1)}
\makeatother

\newcommand{\ccg}{/ \! /}
\newcommand{\bsl}{\backslash}
\newcommand{\ochi}{\overline{\chi}}
\newcommand{\opsi}{\overline{\psi}}
\newcommand{\ve}{\varepsilon}
\newcommand{\eB}{{^\ve\! B}}
\newcommand{\eE}{{^\ve\! E}}
\newcommand{\vecB}{\mathbf{B}}
\newcommand{\vecD}{\mathbf{D}}
\newcommand{\vecE}{\mathbf{E}}
\newcommand{\lc}{\left\lceil}
\newcommand{\rc}{\right\rceil}
\newcommand{\lf}{\left\lfloor}
\newcommand{\rf}{\right\rfloor}
\newcommand{\Z}{\mathbb{Z}}
\newcommand{\Q}{\mathbb{Q}}
\newcommand{\R}{\mathbb{R}}
\newcommand{\C}{\mathbb{C}}
\newcommand{\HH}{\mathbb{H}}
\newcommand{\DD}{\mathcal{D}}
\newcommand{\A}{\mathcal{A}}
\newcommand{\new}{\mathrm{new}}
\newcommand{\Anew}{\mathcal{A}^{\mathrm{new}}}
\newcommand{\Amin}{\mathcal{A}^{\mathrm{min}}}
\newcommand{\rr}{\mathfrak{r}}
\newcommand{\N}{\mathsf{N}}
\newcommand{\tPsi}{\tilde{\Psi}}
\newcommand{\tOmega}{\tilde{\Omega}}
\newcommand{\mbc}{\mathbf{c}}
\newcommand{\boldS}{\bold{S}}
\newcommand{\boldh}{\bold{h}}
\newcommand{\sm}{\left(\begin{smallmatrix}}
\newcommand{\esm}{\end{smallmatrix}\right)}
\newcommand{\bpm}{\begin{pmatrix}}
\newcommand{\ebpm}{\end{pmatrix}}

\DeclareMathOperator{\Gal}{Gal}
\DeclareMathOperator{\lcm}{lcm}
\DeclareMathOperator{\SL}{SL}
\DeclareMathOperator{\PSL}{PSL}
\DeclareMathOperator{\GL}{GL}
\DeclareMathOperator{\PGL}{PGL}
\DeclareMathOperator{\G}{G}
\DeclareMathOperator{\M}{M}
\DeclareMathOperator{\sinc}{sinc}
\DeclareMathOperator{\ord}{ord}
\DeclareMathOperator{\Tr}{Tr}
\DeclareMathOperator{\tr}{tr}
\DeclareMathOperator{\sgn}{sgn}
\DeclareMathOperator{\cond}{cond}
\DeclareMathOperator{\I}{I}
\DeclareMathOperator{\NEl}{NEll}
\DeclareMathOperator{\Eis}{Eis}
\DeclareMathOperator{\El}{Ell}
\DeclareMathOperator{\Cu}{C}
\DeclareMathOperator{\Area}{Area}

\newtheorem{theorem}{Theorem}
\newtheorem{lemma}[theorem]{Lemma}
\newtheorem{proposition}[theorem]{Proposition}

\newtheorem{definition}[theorem]{Definition}

\theoremstyle{remark}
\newtheorem{remarks}[theorem]{Remarks}
\newtheorem{remark}[theorem]{Remark}

\numberwithin{theorem}{section}
\numberwithin{equation}{section}

\begin{document}
\title{Twist-minimal trace formulas and the Selberg eigenvalue conjecture}
\author{Andrew R.~Booker, Min Lee and Andreas Str\"ombergsson}
\address{Howard House, University of Bristol, Queens Ave.,
Bristol BS8 1SN, United Kingdom}
\email{\tt andrew.booker@bristol.ac.uk\\min.lee@bristol.ac.uk}
\address{Department of Mathematics, Box 480, Uppsala University,
751 06 Uppsala, Sweden}
\email{\tt astrombe@math.uu.se}
\thanks{A.~R.~B.\ and M.~L.\ were partially supported by EPSRC Grant
\texttt{EP/K034383/1}. M.~L.\ was partially supported by
Royal Society University Research Fellowship ``Automorphic forms,
L-functions and trace formulas''.  A.~S.\ was supported by 
the Swedish Research Council Grant 2016-03360.}
\begin{abstract}
We derive a fully explicit version of the Selberg trace formula
for twist-minimal Maass forms of weight $0$ and arbitrary conductor
and nebentypus character, and apply it to prove two theorems. First,
conditional on Artin's conjecture, we classify the even $2$-dimensional
Artin representations of small conductor; in particular, we show that
the even icosahedral representation of smallest conductor is the one
found by Doud and Moore \cite{DM06}, of conductor $1951$.  Second,
we verify the Selberg eigenvalue conjecture for groups of small level,
improving on a result of Huxley \cite{Hux85} from 1985.
\end{abstract}
\maketitle

\section{Introduction}
In \cite{BS07}, the first and third authors derived a fully explicit
version of the Selberg trace formula for cuspidal Maass newforms of
squarefree conductor, and applied it to obtain partial results toward the
Selberg eigenvalue conjecture and the classification of $2$-dimensional
Artin representations of small conductor. In this paper we remove
the restriction to squarefree conductor, with the following
applications:
\begin{theorem}\label{thm:SEC}
The Selberg eigenvalue conjecture is true for $\Gamma_1(N)$ for $N\leq 880$,
and for $\Gamma(N)$ for $N\leq 226$.
\end{theorem}
\begin{theorem}\label{thm:Artin}
Assuming Artin's conjecture, Table~\ref{tab:Artin} is the complete list,
up to twist, of even, nondihedral, irreducible, $2$-dimensional
Artin representations of conductor $\le2862$.
\end{theorem}

\begin{table}[h!]\begin{scriptsize}
\begin{tabular}{|rrrrrrrrrrrrrrrr|}\hline
\multicolumn{16}{|l|}{\textbf{tetrahedral}} \\
\href{http://www.lmfdb.org/ArtinRepresentation/2.163.8t12.1c1}{$163$} &
\href{http://www.lmfdb.org/ArtinRepresentation/2.277.8t12.1c1}{$277$} &
\href{http://www.lmfdb.org/ArtinRepresentation/2.349.8t12.1c1}{$349$} &
\href{http://www.lmfdb.org/ArtinRepresentation/2.397.8t12.1c1}{$397$} &
\href{http://www.lmfdb.org/ArtinRepresentation/2.547.8t12.1c1}{$547$} &
\href{http://www.lmfdb.org/ArtinRepresentation/2.3e2\_61.8t12.1c1}{$549$} &
\href{http://www.lmfdb.org/ArtinRepresentation/2.607.8t12.1c1}{$607$} &
\href{http://www.lmfdb.org/ArtinRepresentation/2.7\_97.8t12.1c1}{$679$} &
\href{http://www.lmfdb.org/ArtinRepresentation/2.19\_37.8t12.1c1}{$703$} &
\href{http://www.lmfdb.org/ArtinRepresentation/2.709.8t12.1c1}{$709$} &
\href{http://www.lmfdb.org/ArtinRepresentation/2.3e2\_79.8t12.1c1}{$711$} &
\href{http://www.lmfdb.org/ArtinRepresentation/2.7\_109.8t12.1c1}{$763$} &
\href{http://www.lmfdb.org/ArtinRepresentation/2.853.8t12.1c1}{$853$} &
\href{http://www.lmfdb.org/ArtinRepresentation/2.937.8t12.1c1}{$937$} &
\href{http://www.lmfdb.org/ArtinRepresentation/2.13\_73.8t12.1c1}{$949$} &
\href{http://www.lmfdb.org/ArtinRepresentation/2.5\_199.16t60.1c1}{$995$} \\

\href{http://www.lmfdb.org/ArtinRepresentation/2.1009.8t12.1c1}{$1009$} &
\href{http://www.lmfdb.org/ArtinRepresentation/2.29\_37.16t60.1c1}{$1073$} &
\href{http://www.lmfdb.org/ArtinRepresentation/2.3e2\_127.8t12.1c1}{$1143$} &
\href{http://www.lmfdb.org/ArtinRepresentation/2.31\_37.8t12.1c1}{$1147$} &
\href{http://www.lmfdb.org/ArtinRepresentation/2.3e2\_7\_19.8t12.1c1}{$1197$} &
\href{http://www.lmfdb.org/ArtinRepresentation/2.7\_181.8t12.1c1}{$1267$} &
\href{http://www.lmfdb.org/ArtinRepresentation/2.7\_181.16t60.1c1}{$1267$} &
\href{http://www.lmfdb.org/ArtinRepresentation/2.31\_43.8t12.1c1}{$1333$} &
\href{http://www.lmfdb.org/ArtinRepresentation/2.17\_79.16t60.1c1}{$1343$} &
\href{http://www.lmfdb.org/ArtinRepresentation/2.2e3\_3e2\_19.16t60.1c1}{$1368$} &
\href{http://www.lmfdb.org/ArtinRepresentation/2.1399.8t12.1c1}{$1399$} &
\href{http://www.lmfdb.org/ArtinRepresentation/2.3e2\_157.8t12.1c1}{$1413$} &
\href{http://www.lmfdb.org/ArtinRepresentation/2.1699.8t12.1c1}{$1699$} &
\href{http://www.lmfdb.org/ArtinRepresentation/2.3e2\_197.16t60.1c1}{$1773$} &
\href{http://www.lmfdb.org/ArtinRepresentation/2.1777.8t12.1c1}{$1777$} &
\href{http://www.lmfdb.org/ArtinRepresentation/2.1789.8t12.1c1}{$1789$} \\

\href{http://www.lmfdb.org/ArtinRepresentation/2.1879.8t12.1c1}{$1879$} &
\href{http://www.lmfdb.org/ArtinRepresentation/2.3e2\_211.8t12.1c1}{$1899$} &
\href{http://www.lmfdb.org/ArtinRepresentation/2.3e2\_211.8t12.2c1}{$1899$} &
\href{http://www.lmfdb.org/ArtinRepresentation/2.3e2\_5\_43.16t60.1c1}{$1935$} &
\href{http://www.lmfdb.org/ArtinRepresentation/2.1951.8t12.1c1}{$1951$} &
\href{http://www.lmfdb.org/ArtinRepresentation/2.3e2\_7\_31.8t12.1c1}{$1953$} &
\href{http://www.lmfdb.org/ArtinRepresentation/2.19\_103.8t12.1c1}{$1957$} &
\href{http://www.lmfdb.org/ArtinRepresentation/2.2e6\_31.16t60.1c1}{$1984$} &
\href{http://www.lmfdb.org/ArtinRepresentation/2.7\_293.16t60.1c1}{$2051$} &
\href{http://www.lmfdb.org/ArtinRepresentation/2.31\_67.8t12.1c1}{$2077$} &
\href{http://www.lmfdb.org/ArtinRepresentation/2.3e2\_233.16t60.1c1}{$2097$} &
\href{http://www.lmfdb.org/ArtinRepresentation/2.2131.8t12.1c1}{$2131$} &
\href{http://www.lmfdb.org/ArtinRepresentation/2.5\_7\_61.16t60.1c1}{$2135$} &
\href{http://www.lmfdb.org/ArtinRepresentation/2.3e2\_241.8t12.1c1}{$2169$} &
\href{http://www.lmfdb.org/ArtinRepresentation/2.3e2\_241.8t12.2c1}{$2169$} &
\href{http://www.lmfdb.org/ArtinRepresentation/2.3e2\_13\_19.8t12.1c1}{$2223$}
\\

\href{http://www.lmfdb.org/ArtinRepresentation/2.2311.8t12.1c1}{$2311$} &
\href{http://www.lmfdb.org/ArtinRepresentation/2.13\_181.8t12.1c1}{$2353$} &
\href{http://www.lmfdb.org/ArtinRepresentation/2.3e2\_271.8t12.1c1}{$2439$} &
\href{http://www.lmfdb.org/ArtinRepresentation/2.2e3\_307.16t60.1c1}{$2456$} &
\href{http://www.lmfdb.org/ArtinRepresentation/2.13\_199.8t12.1c1}{$2587$} &
\href{http://www.lmfdb.org/ArtinRepresentation/2.7\_13\_29.16t60.1c1}{$2639$} &
\href{http://www.lmfdb.org/ArtinRepresentation/2.2689.8t12.1c1}{$2689$} &
\href{http://www.lmfdb.org/ArtinRepresentation/2.3e2\_7\_43.8t12.1c1}{$2709$} &
\href{http://www.lmfdb.org/ArtinRepresentation/2.13\_211.8t12.1c1}{$2743$} &
\href{http://www.lmfdb.org/ArtinRepresentation/2.3e2\_307.8t12.1c1}{$2763$} &
\href{http://www.lmfdb.org/ArtinRepresentation/2.2797.8t12.1c1}{$2797$} &
\href{http://www.lmfdb.org/ArtinRepresentation/2.2803.8t12.1c1}{$2803$} &
\href{http://www.lmfdb.org/ArtinRepresentation/2.3e2\_313.8t12.1c1}{$2817$} &
& & \\ \hline

\multicolumn{16}{|l|}{\textbf{octahedral}} \\
\href{http://www.lmfdb.org/ArtinRepresentation/2.5\_157.24t138.2c1}{$785$} &
\href{http://www.lmfdb.org/ArtinRepresentation/2.5\_269.24t138.2c1}{$1345$} &
\href{http://www.lmfdb.org/ArtinRepresentation/2.2e2\_5\_97.48.1c1}{$1940$} &
\href{http://www.lmfdb.org/ArtinRepresentation/2.17\_127e2.24t138.2c1}{$2159^\ast$} &
\href{http://www.lmfdb.org/ArtinRepresentation/2.43\_53.48.1c1}{$2279$} &
\href{http://www.lmfdb.org/ArtinRepresentation/2.3e2\_257.24t138.2c1}{$2313$} &
\href{http://www.lmfdb.org/ArtinRepresentation/2.2e2\_3\_197.24t138.2c1}{$2364$} &
\href{http://www.lmfdb.org/ArtinRepresentation/2.2e3\_3\_101.24t138.2c1}{$2424$} &
\href{http://www.lmfdb.org/ArtinRepresentation/2.2e3\_5\_61.48.2c1}{$2440$} &
\href{http://www.lmfdb.org/ArtinRepresentation/2.2713.24t138.2c1}{$2713$} &
\href{http://www.lmfdb.org/ArtinRepresentation/2.2777.24t22.1c1}{$2777$} &
\href{http://www.lmfdb.org/ArtinRepresentation/2.2777.24t138.1c1}{$2777$} &
\href{http://www.lmfdb.org/ArtinRepresentation/2.2777.24t138.3c1}{$2777$} &
\href{http://www.lmfdb.org/ArtinRepresentation/2.2857.24t138.1c1}{$2857$} &
& \\ \hline

\multicolumn{16}{|l|}{\textbf{icosahedral}} \\
\href{http://www.lmfdb.org/ArtinRepresentation/2.1951e2.120.1c1}{$1951^\ast$} &
\href{http://www.lmfdb.org/ArtinRepresentation/2.1951e2.120.1c2}{$1951^\ast$} &
\href{http://www.lmfdb.org/ArtinRepresentation/2.2141e2.120.1c1}{$2141^\ast$} &
\href{http://www.lmfdb.org/ArtinRepresentation/2.2141e2.120.1c2}{$2141^\ast$} &
\href{http://www.lmfdb.org/ArtinRepresentation/2.2e2\_701e2.120.1c1}{$2804^\ast$} &
\href{http://www.lmfdb.org/ArtinRepresentation/2.2e2\_701e2.120.1c2}{$2804^\ast$}
&
& & & & & & & & & \\ \hline
\end{tabular}
\end{scriptsize}
\caption{Even, nondihedral Artin representations of conductor
$\le2862$, up to twist.  For each twist equivalence class we indicate
the minimal Artin conductor and link to the
\href{http://www.lmfdb.org/ArtinRepresentation/}{\texttt{LMFDB}}
page of a representation in the class. It is twist
minimal in all cases except those marked with $\ast$.}
\label{tab:Artin}
\end{table}

\begin{remarks}\
\begin{itemize}
\item As we pointed out in \cite{BS07}, in the case of squarefree
level the Selberg trace formula becomes substantially cleaner if
one sieves down to \emph{newforms}, and that also helps in numerical
applications by thinning out the spectrum. For nonsquarefree level
this is no longer the case, as the newform sieve results in
more complicated formulas in many cases.

Our main innovation in this paper is to introduce a further sieve down to
\emph{twist-minimal} forms, i.e.\ those newforms whose conductor cannot
be reduced by twisting. Although there are many technical complications
to overcome in the intermediate stages, in the end we find that the
twist-minimal trace formula is again significantly cleaner and helps to
improve the numerics.

A natural question to explore in further investigations is whether one
can skip the intermediate stages and derive the twist-minimal trace
formula directly. 
A direct proof might shed light on why several complicated terms 
of the trace formula are annihilated by the twist-minimal sieve, 
and avoid messy calculations.

\item Theorem~\ref{thm:SEC} for $\Gamma(N)$ improves a 30-year-old
result of Huxley \cite{Hux85}, who proved the Selberg eigenvalue
conjecture for groups of level $N\le18$. Treating nonsquarefree
conductors is essential for this application, since a form of level $N$
can have conductor\footnote{The level of a form $f$ is the smallest $N$
such that $f$ is modular for $\Gamma(N)$. Its conductor is the conductor
of the associated automorphic representation, or equivalently the
smallest $N$ such that $f$ is modular for some conjugate of
$\Gamma_1(N)$.}
as large as $N^2$. Moreover, the reduction to
twist-minimal spaces yields a substantial improvement in our numerical
results by essentially halving the spectrum in the critical case of forms of
prime level $N$ and conductor $N^2$ (see \eqref{eq:Mchi} below). This partially
explains why our result for $\Gamma(N)$ is within a factor of $4$ of
that for $\Gamma_1(N)$, despite the conductors being much larger.

\item By the Langlands--Tunnell theorem, the Artin conjecture is true
for tetrahedral and octahedral representations, so the conclusion
of Theorem~\ref{thm:Artin} holds unconditionally for those types.
In the icosahedral case, by \cite{Boo03} it is enough to assume the
Artin conjecture for all representations in a given Galois conjugacy
class; i.e., if there is a twist-minimal, even icosahedral
representation of conductor $\le2862$ that does not appear in
Table~\ref{tab:Artin}, then Artin's conjecture is false for
at least one of its Galois conjugates.

The entries of Table~\ref{tab:Artin} were computed by Jones and Roberts
\cite{JR17} by a thorough search of number fields with prescribed
ramification behavior, and we verified the completeness of the list via
the trace formula. In principle the number field search by Jones and
Roberts \cite{JR14} is exhaustive, so it should be possible to prove
Theorem~\ref{thm:Artin} unconditionally with a further computation,
but that has not yet been carried out to our knowledge.

For comparison, we note that Buzzard and Lauder \cite{BL17b} have
characterized the \emph{odd} $2$-dimensional representations of conductor
$\le1500$ by computing bases of the associated spaces of weight $1$
holomorphic modular forms.

\item
Theorem~\ref{thm:SEC} for $\Gamma_1(N)$ improves on the result
from \cite{BS07} by extending to nonsquarefree $N$ and increasing the
upper bound from $854$ to $880$. To accomplish the latter, we computed a
longer list of class numbers of the quadratic fields $\Q(\sqrt{t^2\pm4})$
using the algorithm from \cite{BBJ}.  Nothing (other than limited patience
of the user) prevents computing an even longer list and increasing the
bounds in Theorem~\ref{thm:SEC} a bit more. However, as explained in
\cite[\S6]{BS07}, our method suffers from an exponential barrier to
increasing the conductor, so that by itself is likely to yield only
marginal improvements. Some ideas for surmounting this barrier are
described in \cite[\S6]{BS07}; in any case, as Theorem~\ref{thm:Artin}
shows, the first even icosahedral representation occurs at conductor
$N=1951$,\footnote{The existence of this representation
was first shown by Doud and Moore \cite{DM06},
who also proved that $1951$ is minimal among \emph{prime} conductors of
even icosahedral representations.} so the bounds in Theorem~\ref{thm:SEC}
cannot be improved unconditionally beyond $1950$.
\item
All of our computations with real and complex numbers were carried out
using the interval arithmetic package
\href{http://arblib.org/}{\texttt{Arb}} \cite{Joh17}. Thus,
modulo bugs in the software and computer hardware, our results are
rigorous. The reader is invited to inspect our source code at \cite{code}.
\end{itemize}
\end{remarks}

We conclude the introduction with a brief outline of the paper. In
\S\ref{sec:prelim}--\ref{ss:twist-minimal_TF} we define the space of
twist-minimal Maass forms and state our version of the Selberg
trace formula for it. In \S\ref{sec:fullTF} we state and prove the
full trace formula in general terms, and then specialize it to
$\Gamma_0(N)$ with nebentypus character. In
\S\ref{sec:sieve} we apply the sieving process to pass from the full
space to newforms, and then to twist-minimal forms. In
\S\ref{sec:galois}, we describe some details of the application of the
trace formula to $\Gamma(N)$ and to Artin representations.
Finally, in \S\ref{sec:numerics}
we make a few remarks on numerical aspects of the proofs of
Theorems~\ref{thm:SEC} and \ref{thm:Artin}.

\subsection*{Acknowledgements}
We are grateful to Abhishek Saha for teaching us about twist-minimal
representations, in particular Lemma~\ref{lem:twistconductor}.
We thank John Jones and David Roberts for their efforts to
find Artin representations for all of the twist-equivalence classes
in Table~\ref{tab:Artin} and sharing their data in the
\href{http://www.lmfdb.org/ArtinRepresentation/}{\texttt{LMFDB}}.
Finally, we thank Andrew Knightly for helpful comments and corrections.

\subsection{Preliminaries on twist-minimal spaces of Maass forms}\label{sec:prelim}
Let $\HH=\{z=x+iy\in\C:y>0\}$ denote the hyperbolic plane.
Given a real number $\lambda>0$ and an even Dirichlet character $\chi\pmod*{N}$,
let $\A_\lambda(\chi)$ denote the vector space of Maass cusp forms of
eigenvalue $\lambda$, level $N$ and nebentypus character $\chi$,
i.e.\ the set of smooth functions $f:\HH\to\C$ satisfying
\begin{enumerate}
\item $f(\frac{az+b}{cz+d})=\chi(d)f(z)$ for all
$\begin{psmallmatrix}a&b\\c&d\end{psmallmatrix}\in\Gamma_0(N)$;
\item $\int_{\Gamma_0(N)\backslash\HH}|f|^2\frac{dx\,dy}{y^2}<\infty$;
\item $-y^2\left(\frac{\partial^2}{\partial x^2}
+\frac{\partial^2}{\partial y^2}\right)f=\lambda f$.
\end{enumerate}
We omit the level $N$ from the notation $\A_\lambda(\chi)$ since it is
determined implicitly as the modulus of $\chi$ (which might differ
from its conductor, i.e.\ $\chi$ need not be primitive).
Any $f\in\A_\lambda(\chi)$ has a Fourier expansion of the form
$$
f(x+iy)=\sqrt{y}\sum_{n\in\Z\setminus\{0\}}
a_f(n)K_{\sqrt{\frac14-\lambda}}(2\pi|n|y)e^{2\pi inx},
$$
for certain coefficients $a_f(n)\in\C$, where
$K_s(y)=\frac12\int_\R e^{st-y\cosh{t}}\,dt$
is the $K$-Bessel function.

For any $n\in\Z\setminus\{0\}$ coprime to $N$, let
$T_n:\A_\lambda(\chi)\to\A_\lambda(\chi)$ denote the Hecke
operator defined by
$$
(T_nf)(z)=\frac1{\sqrt{|n|}}
\sum_{\substack{a,d\in\Z\\d>0,ad=n}}
\chi(a)\sum_{b\pmod*{d}}
\begin{cases}
f\!\left(\frac{az+b}{d}\right)&\text{if }n>0,\\
f\!\left(\frac{a\bar{z}+b}{d}\right)&\text{if }n<0.
\end{cases}
$$
We say that $f\in\A_\lambda(\chi)$ is a
\emph{normalized Hecke eigenform} if $a_f(1)=1$ and
$f$ is a simultaneous eigenfunction of $T_n$ for every $n$ coprime to
$N$.  In this case, one has $T_nf=a_f(n)f$.

Let $\cond(\chi)$ denote the conductor of $\chi$. For any $M\in\Z_{>0}$
with $\cond(\chi)\mid M$, let $\chi|_M$ denote the unique character mod $M$
such that $\chi|_M(n)=\chi(n)$ for all $n$ coprime to $MN$.  Then for
any $M$ with $\cond(\chi)\mid M\mid N$ and any $d\mid\frac{N}{M}$, we have
a linear map $\ell_{M,d}:\A_\lambda(\chi|_M)\to\A_\lambda(\chi)$
defined by $(\ell_{M,d}f)(z)=f(dz)$. Let
$$
\A^{\rm old}_\lambda(\chi)=\sum_{\substack{M,d\in\Z_{>0}\\
\cond(\chi)\mid M\mid N,\,M<N\\d\mid\frac{N}{M}}}
\ell_{M,d}\A_\lambda(\chi|_M)
$$
denote the span of the images of all lower level forms under these maps,
and let $\Anew_\lambda(\chi)\subseteq\A_\lambda(\chi)$ denote the orthogonal
complement of $\A^{\rm old}_\lambda(\chi)$ with respect to the Petersson inner
product
$$
\langle f,g\rangle=\int_{\Gamma_0(N)\backslash\HH}
f\bar{g}\frac{dx\,dy}{y^2}.
$$
We call this the space of \emph{newforms} of eigenvalue $\lambda$,
conductor $N$ and character $\chi$.

By strong multiplicity one, we have
$\Anew_\lambda(\chi)\cap\Anew_\lambda(\chi')=\{0\}$
unless $\chi$ and $\chi'$ have the same modulus and satisfy
$\chi(n)=\chi'(n)$ for all $n$.
In particular, any nonzero
$f\in\Anew_\lambda(\chi)$ uniquely determines its conductor, which
we denote by $\cond(f)$.
Moreover, the Hecke operators $T_n$ map $\Anew_\lambda(\chi)$ to itself,
and $\Anew_\lambda(\chi)$ has a unique basis consisting of
normalized Hecke eigenforms.

Suppose that $f\in\Anew_\lambda(\chi)$ is a normalized Hecke eigenform.
Then for any Dirichlet
character $\psi\pmod*{q}$, there is a unique $M\in\Z_{>0}$
and a unique
$g\in\Anew_\lambda(\chi\psi^2|_M)$ such that
$a_g(n)=a_f(n)\psi(n)$ for all $n\in\Z\setminus\{0\}$ coprime to $q$.
We write $f\otimes\psi$ to denote the corresponding $g$.
We say that $f$ is
\emph{twist minimal} if $\cond(f\otimes\psi)\ge\cond(f)$ for every
Dirichlet character $\psi$.

Let $\Amin_\lambda(\chi)\subseteq\Anew_\lambda(\chi)$ denote the
subspace spanned by twist-minimal normalized Hecke eigenforms.
Clearly any normalized Hecke eigenform can be expressed as
$f\otimes\psi$ for some twist-minimal form $f$. In turn, for any
normalized Hecke eigenform $f\in\Amin_\lambda(\chi)$, we have
$$
\cond(f\otimes\psi)=\lcm(\cond(f),\cond(\psi)\cond(\chi\psi)),
$$
as implied by the following lemma, strengthening
\cite[Lemma~2.1]{Hum19} and \cite[Lemma~2.7]{CS18}.
\begin{lemma}\label{lem:twistconductor}
Let $F$ be a nonarchimedean local field, $\pi$ an irreducible,
admissible, generic
representation of $\GL_2(F)$ with central character
$\chi$, and $\psi$ a character of $F^\times$.
Let $a(\pi)$ and $a(\psi)$ denote the respective conductor exponents of
$\pi$ and $\psi$ (written additively). Then we have
\begin{equation}\label{eq:twistconductor}
a(\pi\otimes\psi)\le\max\{a(\pi),a(\psi)+a(\chi\psi)\},
\end{equation}
with equality if $\pi$ is twist minimal or $a(\pi)\ne a(\psi)+a(\chi\psi)$.
\end{lemma}
\begin{proof}
There are three cases to consider:
\begin{enumerate}
\item $\pi\cong\chi_1\boxplus\chi_2$ is a principal series
representation, in which case
$$
a(\pi\otimes\psi)=a(\chi_1\psi)+a(\chi_2\psi),\quad
a(\pi)=a(\chi_1)+a(\chi_2),\quad\text{and}\quad\chi=\chi_1\chi_2.
$$
\item $\pi\cong\operatorname{St}\otimes\chi_1$ is a twist of the Steinberg
representation, in which case
$$
a(\pi\otimes\psi)=\max\{1,2a(\chi_1\psi)\},\quad
a(\pi)=\max\{1,2a(\chi_1)\},\quad\text{and}\quad\chi=\chi_1^2.
$$
\item $\pi$ is supercuspidal.
\end{enumerate}
In the first two cases, we verify \eqref{eq:twistconductor} by a
laborious case-by-case analysis based on the observation that, for any
characters $\xi_1$ and $\xi_2$,
$a(\xi_1\xi_2)\le\max\{a(\xi_1),a(\xi_2)\}$, with equality if
$a(\xi_1)\ne a(\xi_2)$. This also shows that equality holds in
\eqref{eq:twistconductor} when $a(\pi)\ne a(\psi)+a(\chi\psi)$.

In the third case, Tunnell \cite[Prop.~3.4]{Tun78} showed that
$a(\pi\otimes\psi)\le\max\{a(\pi),2a(\psi)\}$, with equality if
$a(\psi)\ne\frac12a(\pi)$, and that $a(\chi)\le\frac12a(\pi)$.
It follows that
$\max\{a(\pi),2a(\psi)\}=\max\{a(\pi),a(\psi)+a(\chi\psi)\}$, so
Tunnell's proposition implies \eqref{eq:twistconductor}.

Let us prove that equality holds in \eqref{eq:twistconductor} if
$a(\pi)\neq a(\psi)+a(\chi\psi)$. By what we have already pointed out,
we may assume $a(\psi)=\frac12a(\pi)$, and then
$a(\pi)\neq a(\psi)+a(\chi\psi)$ implies $a(\chi\psi)<\frac12a(\pi)$.
Applying Tunnell's proposition to the representation
$\pi^\vee\cong\pi\otimes\chi^{-1}$ (which has $a(\pi^\vee)=a(\pi)$)
and the character $\chi\psi$, we then obtain
$a(\pi\otimes\psi)=a(\pi^\vee\otimes\chi\psi)=a(\pi)$, i.e.\
equality holds in \eqref{eq:twistconductor}.

Finally, if $\pi$ is twist minimal then $a(\pi\otimes\psi)\ge a(\pi)$,
so equality holds in \eqref{eq:twistconductor} even if
$a(\pi)=a(\psi)+a(\chi\psi)$.
\end{proof}
\noindent
Thus, $f\mapsto f\otimes\psi$ extends to an injective linear map from
$\Amin_\lambda(\chi)$ to $\Anew_\lambda(\chi\psi^2|_M)$, where
$M=\lcm(N,\cond(\psi)\cond(\chi\psi))$.

In light of this, it is enough to consider the trace formula for
twist-minimal spaces of forms. In fact, since some twist-minimal spaces
are trivial, and for others a given form can have more
than one representation as the twist of a twist-minimal form (i.e.\ it is
possible to have $\cond(f\otimes\psi)=\cond(f)$ and $f\otimes\psi\ne f$,
cf.\ \cite{Hum19}),
a further reduction of the nebentypus character is possible, as follows.
\begin{definition}\label{def:minimal}
Let $\chi=\prod_{p\mid N}\chi_p$ be a Dirichlet character modulo $N$,
and put $e_p=\ord_p{N}$, $s_p=\ord_p\cond(\chi)$.
We say that $\chi$ is \emph{minimal}
if the following statement holds for every prime $p\mid N$:
\begin{align*}
&p>2\text{ and }\bigl[s_p\in\{0,e_p\}\text{ or }
\chi_p\text{ has order }2^{\ord_2(p-1)}\bigr]\\
\text{or }&p=2\text{ and }s_p\in\{\lfloor{e_p/2}\rfloor,e_p\}\cup\begin{cases}
\{0\}&\text{if }e_p\le 3,\\
\{0,2\}&\text{if }e_p>3\text{ and }2\nmid e_p,\\
\emptyset&\text{if }e_p>3\text{ and }2\mid e_p.
\end{cases}
\end{align*}
\end{definition}
Note that for an odd prime $p$,
there are $2^{\ord_2(p-1)-1}$ choices of $\chi_p\pmod*{p^{e_p}}$
of order $2^{\ord_2(p-1)}$, and for any such $\chi_p$ we
have $s_p=1$ and $\chi_p(-1)=-1$; in particular, when
$p\equiv3\pmod*{4}$, the Legendre symbol $\left(\frac{\cdot}{p}\right)$
is the unique such character. (For $p\equiv1\pmod*{4}$ and $e_p>1$, the spaces
resulting from different choices of $\chi_p$ of order $2^{\ord_2(p-1)}$
are twist equivalent, but there is no canonical choice. Similarly, for
$p=2$ the characters of conductor $2^{\lfloor{e_2/2}\rfloor}$ and fixed
parity yield twist-equivalent spaces.)

\begin{lemma}\label{lem:twistminimal_isom}
Let $\chi\pmod*{N}$ be a Dirichlet character, and
suppose that $\Amin_\lambda(\chi)\ne\{0\}$. Then there exists
$\psi\pmod*{N}$ such that $\chi\psi^2$ is minimal and
$\Amin_\lambda(\chi)\xrightarrow{\otimes\psi}
\Amin_\lambda(\chi\psi^2)$ is an isomorphism.
\end{lemma}
\begin{proof}
Since the map $f\mapsto f\otimes\psi$ is injective, by
Lemma~\ref{lem:twistconductor} it suffices to show
that there is a $\psi\pmod*{N}$ such that $\chi\psi^2$ is minimal and
$\cond(\psi)\cond(\chi\psi)\mid N$. Writing $\psi=\prod_{p\mid N}\psi_p$,
this is equivalent to finding $\psi_p\pmod*{p^{\ord_p{N}}}$ such that
$\chi_p\psi_p^2$ is minimal and
\begin{equation}\label{eq:ordpsum}
\ord_p\cond(\psi_p)+\ord_p\cond(\chi_p\psi_p)\le\ord_p{N}
\end{equation}
for each prime $p\mid N$.

Fix $p\mid N$ and set $e=\ord_p{N}$, $s=\ord_p\cond(\chi_p)$. If
$e\in\{s,1\}$ then $\chi_p$ is minimal, so we
can take $\psi_p$ equal to the trivial character mod $p^e$.
Hence we may assume that $e>\max\{s,1\}$.
In this case, since $\Amin_\lambda(\chi)$ is nonzero, it follows from
\cite[Theorem~$4.3'$]{AL78} that $s\le\frac12e$.

Suppose that $p$ is odd, and let $g$ be a primitive root mod $p^e$. Then
$\chi_p(g)=\exp(2\pi ia/\varphi(p^e))$ for a unique $a\in\Z\cap[1,\varphi(p^e)]$.
Set
$$
b=\begin{cases}
-a/2&\text{if }2\mid{a},\\
[\varphi(p^e)2^{-\ord_2(p-1)}-a]/2&\text{if }2\nmid a,
\end{cases}
$$
and let $\psi_p$ be the character defined by
$\psi_p(g)=\exp(2\pi ib/\varphi(p^e))$. Then $\chi_p\psi_p^2$ is minimal, and
since $p^{\ord_p{a}}\mid b$ and $e>1$, it
follows that $\ord_p\cond(\psi_p)\le\frac12e$, which implies the desired
inequality \eqref{eq:ordpsum}.

Suppose now that $p=2$. Since $s\le\frac12e$ and there is no character
of conductor $2$, $\chi_2$ is already minimal if $e\le 3$. Also,
by the analogue of \cite[Theorem~4.4(iii)]{AL78} for Maass forms, if $e\ge4$ is
even then we must have $s=e/2$, so $\chi_2$ is again minimal. Hence, we
may assume that $e$ is an odd number exceeding $3$.

If $s\in\{0,2,\lfloor{e/2}\rfloor\}$ then $\chi_2$ is minimal, so we may
assume that $3\le s\le(e-3)/2$.
Recalling that $(\Z/2^e\Z)^\times$ is generated by $-1$ and $5$,
we have $\chi_2(5)=\exp(2\pi ia/2^{s-2})$ for a unique odd number
$a\in[1,2^{s-2})$. Let $\psi_2$ be the character defined by
$\psi_2(-1)=1$ and $\psi_2(5)=\exp(-2\pi ia/2^{s-1})$. Then
$\chi_2(5)\psi_2(5)^2=1$, so $\chi_2\psi_2^2$ has conductor at most $4$,
and is therefore minimal.
Moreover, $\cond(\psi_2)=\cond(\chi_2\psi_2)=2^{s+1}$, so that
$$
\ord_2\cond(\psi_2)+\ord_2\cond(\chi_2\psi_2)=2s+2<e,
$$
as desired.
\end{proof}
\noindent
Thus, we may restrict our attention to the spaces $\Amin_\lambda(\chi)$
for minimal characters $\chi$.

\subsection{The twist-minimal trace formula}\label{ss:twist-minimal_TF}
Let $g:\R\to\C$ be even, continuous and absolutely integrable, with Fourier
transform $h(r)=\int_\R g(u)e^{iru}\,du$. We say that the pair $(g,h)$
is of \emph{trace class} if there exists $\delta>0$ such that $h$ is
analytic on the strip $\Omega=\{r\in\C:|\Im(r)|<\frac12+\delta\}$
and satisfies $h(r)\ll(1+|r|)^{-2-\delta}$ for all $r\in\Omega$.
%

Let $\chi\pmod*{N}$ be a Dirichlet character, and write
$\chi=\prod_{p\mid N}\chi_p$, where
each $\chi_p$ is a character modulo $p^{\ord_p{N}}$.
The trace formula is an expression for
$\sum_{\lambda>0}\tr T_n|_{\Amin_\lambda(\chi)}
h\Bigl(\sqrt{\lambda-\tfrac14}\Bigr)$
in terms of $g$ and $\chi$.
Its terms are linear functionals of $g$
with coefficients that are multiplicative functions of $\chi$,
i.e.\ functions $F$ satisfying
$F(\chi)=\prod_{p\mid N}F(\chi_p)$.

In what follows we fix a prime $p$ and a character $\chi\pmod*{p^e}$ of
conductor $p^s$, and define the local factor at $p$ for various terms
appearing in the trace formula. As a notational convenience,
for any proposition $P$ we write $\delta_P$ to denote the
characteristic function of $P$, i.e.\ $\delta_P=1$ if $P$ is true and
$\delta_P=0$ if $P$ is false.

\subsubsection*{Identity term}
Define
\begin{equation}\label{eq:Mchi}
M(\chi)=\begin{cases}
p^{e-1}(p+1) & \text{if }s=e, \\
\frac{p-\chi(-1)}{\gcd(2,p-1,e)}\varphi(p^{e-1})
& \text{if }s<e\le2,\\
\frac{p^2-1}{\gcd(2,p-1,e)}\varphi(p^{e-2})
& \text{if }e>\max\{s,2\}.
\end{cases}
\end{equation}

\subsubsection*{Constant eigenfunction}
Define
$$
\mu(\chi)=\begin{cases}
-1&\text{if }s<e=1,\\
0&\text{otherwise.}
\end{cases}
$$

\subsubsection*{Parabolic terms}
For $m\ge1$ and $n\in\{\pm1\}$, define
\begin{align}\label{Phimndef}
\Phi_{m,n}(\chi)=\begin{cases}
\overline{\chi(m)}+\chi(n)\chi(m)&\text{if }s=e,\\
-1&\text{if }s<e=1\text{ and }m=p^k\text{ for some }k>0,\\
0&\text{otherwise}
\end{cases}
\end{align}
and
$$
\Psi_n(\chi)=\begin{cases}
1+\chi(n) & \text{if }s=e\text{ or }n=e=1,\\
\frac{p-\chi(-1)}{\gcd(2,p-1,e)}  & \text{if } n=1, s<e=2, \\
\frac{p^{\lfloor\frac{e-3}2\rfloor}+p^{\lceil\frac{e-3}2\rceil}}
{\gcd(2,p-1,e)}(p-1)
& \text{if }n=1, e>\max\{s,2\},\\
\frac32 & \text{if }n=-1, p=2, s<e=1, \\
\frac12 & \text{if }n=-1, p=2, s<e\in\{2, 3\},\\
1 & \text{if }n=-1, p>2, s<e=1,\\
0 & \text{otherwise.}
\end{cases}
$$

\subsubsection*{Elliptic and hyperbolic terms}
Following the notation in \cite[\S1.1]{BL17a}, let
$\DD$ denote the set of discriminants, that is
$$
\DD=\{D\in\Z:D\equiv 0\mbox{ or }1\pmod*{4}\}.
$$
Any nonzero $D\in\DD$ may be expressed uniquely in the form $d\ell^2$,
where $d$ is a fundamental discriminant and $\ell>0$. We define
$\psi_D(n)=\left(\frac{d}{n/\gcd(n,\ell)}\right)$, where
$\left(\frac{\;\;}{\;\;}\right)$ denotes the Kronecker symbol.  Note
that
$\psi_D$ is periodic modulo $D$, and if $D$ is fundamental then $\psi_D$
is the usual quadratic character mod $D$.  Set
$$
L(z,\psi_D)=\sum_{n=1}^{\infty}\frac{\psi_D(n)}{n^z}
\quad\mbox{for }\Re(z)>1.
$$
Then it is not hard to see that
$$
L(z,\psi_D)=L(z,\psi_d)\prod_{p\mid\ell}
\Biggl[1+\bigl(1-\psi_d(p)\bigr)\sum_{j=1}^{\ord_p\ell}p^{-jz}\Biggr],
$$
so that $L(z,\psi_D)$ has analytic continuation to $\C$, apart from a
simple pole at $z=1$ when $D$ is a square. In particular, if $D$ is not
a square then we have
$$
L(1,\psi_D)=L(1,\psi_d)\cdot\frac1{\ell}\prod_{p\mid\ell}
\left[1+\bigl(p-\psi_d(p)\bigr)\frac{p^{\ord_p\ell}-1}{p-1}\right].
$$

Let $t\in\Z$ and $n\in\{\pm1\}$ with $D=t^2-4n$ not a square.
Then $D\in\DD$, so we may write $D=d\ell^2$ as above.
Define $r=\ord_p(D/2)+1$,
$$
\alpha=\frac{\psi_d(p)-1}
{1+(p-\psi_d(p))\frac{p^{\ord_p\ell}-1}{p-1}}
\quad\text{and}\quad
\omega=\begin{cases}
1&\text{if }2\nmid t\text{ or }p=2\text{ and }r=2s,\\
0&\text{otherwise.}
\end{cases}
$$
For $s=e$ we set
\begin{equation}\label{e:Htn_s=e}
H_{t,n}(\chi)=\begin{cases}
\chi\bigl(\frac{t+p^s\omega}{2}\bigr)
\bigl(2+\alpha\frac{p^e+p^{e-1}-2}{p-1}\bigr)
& \text{if }r\geq 2e,\\
\chi\bigl(\frac{t+\sqrt{d}\ell}{2}\bigr) +
\chi\bigl(\frac{t-\sqrt{d}\ell}{2}\bigr)
& \text{if }r<2e\text{ and }\psi_d(p)=1, \\
0 & \text{otherwise,}
\end{cases}
\end{equation}
where $\sqrt{d}$ denotes a square root of $d$ modulo $4p^e$.
For $s<e$ we set
\begin{multline}\label{e:Htn_s<e}
H_{t,n}(\chi)=\delta_{e=1\text{ or } \left(\frac{n}{p}\right)=1}
\cdot\delta_{r\geq e-1}\cdot\frac{\alpha}{\gcd(2,e)}
\chi\!\left(\frac{t+p^s\omega}{2}\right) p^{e-3}\\
\cdot\bigg[p\bigl(p-\delta_{e=2}\chi(-1)
-\delta_{r=e-1}(p+\delta_{2\mid e})\bigr)-\delta_{e>2}\bigg]
\end{multline}
when $p>2$, and
\begin{equation}\label{e:boldS2_min_e>s}	
H_{t,n}(\chi)=\alpha\chi\!\left(\frac{t+2^{r/2}\omega}{2}\right)2^{e-3}
\begin{cases}
3 & \text{if }e\ge3, r>e, \\
1+2(-1)^e & \text{if }e\ge 3, r=e, \\
1-2(-1)^d & \text{if }e\ge3, r=e-1, 2\nmid e, \\
2 & \text{if }e=2, r\geq e, \\
-1 & \text{if }e=2, r<e, \\
4 & \text{if } e=1,\\
0 & \text{otherwise}
\end{cases}
\end{equation}
when $p=2$.

With the notation in place, we can now state the trace formula for
$T_n$ acting on $\Amin_\lambda(\chi)$.
\begin{theorem}\label{THMmintf}
Let $\chi\pmod*{N}$ be a minimal character of conductor $q$,
$n\in\{\pm1\}$, and $(g,h)$ a pair of test functions of trace class.
Then
\begin{align*}
\sum_{\lambda>0}&\tr T_n|_{\Amin_\lambda(\chi)}
h\Bigl(\sqrt{\lambda-\tfrac14}\Bigr)\\
&=-\delta_{n=1}\biggl[
\frac{M(\chi)}{12}\int_\R\frac{g'(u)}{\sinh(u/2)}\,du
+\mu(\chi)\int_\R g(u)\cosh(u/2)\,du\biggr]\\
&+\sum_{\substack{t\in \Z\\D=t^2-4n\\\sqrt{D}\notin\Q}}
H_{t,n}(\chi)L(1,\psi_{D})
\begin{cases}
g\!\left(2\log\frac{|t|+\sqrt{D}}2\right)
&\text{if }D>0,\\
\frac{\sqrt{|D|}}{\pi}\int_\R
\frac{g(u)\cosh(u/2)}{4\sinh^2(u/2)+|D|}\,du
&\text{if }D<0
\end{cases}
\\
&+\Phi_{1,n}(\chi)\bigg\{
\biggl(\gamma+\log\frac{2\pi}{N}+\frac12\log\frac{2}{(2,N)}\biggr)g(0)
-\int_0^\infty\log\Bigl(\sinh\frac{u}{2}\Bigr)g'(u)\,du\\
&+\delta_{n=1}\biggl[
\biggl(\gamma+\log\frac{2}{N}+\frac12\log(2,N)
+\frac12\sum_{p\mid{N}}\log{p}\biggr)g(0)
-\int_0^\infty\log\Bigl(\tanh\frac{u}{4}\Bigr)g'(u)\,du
\biggr]\bigg\}\\
&+2\sum_{m=2}^\infty \frac{\Lambda(m)}{m}\Phi_{m,n}(\chi)g(2\log m)
-\Lambda(N/q)\Psi_n(\chi)g(0)-\delta_{N=1}\frac14\int_\R g(u)\,du.
\end{align*}

\end{theorem}
\section{Statement and proof of the full Selberg trace formula}\label{sec:fullTF}
\subsection{Selberg trace formula in general form from \cite{BS07}}\label{ss:STF_general}
The group of all isometries (orientation preserving or not) of $\HH$ can be identified with $\G = \PGL_2(\R)$,
where the action is defined by
$$
	T(z)
	=
	\begin{cases}
	\frac{az+b}{cz+d} & \text{if } ad-bc>0, \\
	\frac{a\bar{z}+b}{c\bar{z}+d} & \text{if } ad-bc<0,
	\end{cases}
	\quad
	\text{for } T=\bpm a & b\\ c& d\ebpm \in \G.
$$
The group of orientation preserving isometries, $\G^+ = \PSL_2(\R)$, is a subgroup of index $2$ in $\G$.
We write $\G^-=\G\setminus \G^+$ for the other coset in $\G$.

Let $\Gamma$ be a discrete subgroup of $\G$ such that the surface $\Gamma\bsl\HH$ is noncompact but of finite area,
and let $\chi$ be a (unitary) character on $\Gamma$.
We set $\Gamma^+:= \Gamma\cap \G^+$ and assume $\Gamma^+\neq \Gamma$.
We let $L^2(\Gamma\bsl \HH, \chi)$ be the Hilbert space of functions $f: \HH \to \C$ satisfying
the automorphy relation
$$
	f(\gamma z) = \chi(\gamma) f(z), \quad\forall \gamma \in \Gamma,
$$
and
$$
	\int_{\Gamma\bsl \HH} \left|f\right|^2 \; d\mu < \infty.
$$
We let $\{\phi_j\}_{j\geq 1}$ be any orthonormal basis of the discrete spectrum
of the Laplace operator $\Delta = -y^2 \left(\frac{\partial^2}{\partial x^2}+\frac{\partial^2}{\partial y^2}\right)$
on $L^2(\Gamma\bsl \HH, \chi)$, i.e.\
$\phi_j \in C^\infty(\HH)\cap L^2(\Gamma\bsl \HH, \chi)$
and $\Delta \phi_j = \lambda_j \phi_j$,
say with increasing eigenvalues $0\leq \lambda_1\leq\lambda_2\leq\cdots$.
We also let
$$
	r_j = \sqrt{\lambda_j -\tfrac{1}{4}} \in \R^+\cup i\left[-\tfrac{1}{2}, 0\right].
$$

The trace formula from \cite[Theorem~2 and (2.37)]{BS07} reads as follows.
The even analytic function $h$ and its Fourier transform $g$ are given as in Section~\ref{ss:twist-minimal_TF}.
The trace formula for $(\Gamma, \chi)$ is
$$
	\sum_{j\geq 1} h(r_j)
	=
	\I(\Gamma, \chi)+ \NEl(\Gamma, \chi) + \El(\Gamma, \chi) + \Cu(\Gamma, \chi) + \Eis(\Gamma, \chi), 	
$$
where
\begin{equation}\label{e:I}
	\I(\Gamma, \chi) = \frac{\Area(\Gamma\bsl \HH)}{4\pi} \int_\R h(r) r\tanh(\pi r) \; dr,
\end{equation}
\begin{equation}\label{e:NEl}
	\NEl(\Gamma, \chi)
	=
	\sum_{\{T\}\subset M, \text{ non-ell}}
	\frac{\log N(T_0)}{[Z_{\Gamma}(T): [T_0]]}
	\frac{\chi(T)}{N(T)^{\frac{1}{2}}- \sgn(\det T) N(T)^{-\frac{1}{2}}}
	 g(\log N(T)),
\end{equation}
\begin{equation}\label{e:El}
	\El(\Gamma, \chi)
	=
	\sum_{\{T\}\subset M, \text{ elliptic}} \frac{\chi(T)}{2|Z_\Gamma(T)| \sin(\theta(T))}
	\int_\R \frac{e^{-2\theta(T) r}}{1+e^{-2\pi r}} h(r) \; dr,
\end{equation}
\begin{multline}\label{e:Cu}
	\Cu(\Gamma, \chi)
	=
	\bigg(\sum_{j\in C_{\Gamma, \chi}, k(j) = j} \chi(T_{j, 0})\bigg)
	\bigg\{ \frac{1}{8} h(0)
	+ \frac{1}{4\pi} \int_\R h(r) \bigg(\frac{\Gamma'}{\Gamma}\!\left(\frac{1}{2}+ir\right)
	-\frac{\Gamma'}{\Gamma}(1+ir) \bigg) \; dr
	\bigg\}
	\\
	+
	\frac{g(0)}{4}
	\bigg( \sum_{1\leq j\leq \kappa, k(j) = j} \sum_{v\in\{0, 1\}} \chi(T_{j, v}) \log \mbc_{j, v}
	- 2\sum_{1\leq j\leq \kappa, j\notin C_{\Gamma, \chi}} \log |1-\chi(T_j)|
	\bigg)
	\\
	+
	\frac{|C_{\Gamma, \chi}|}{2}
	\bigg\{
	\frac{1}{4} h(0)
	-
	g(0)\log 2
	-
	\frac{1}{2\pi}\int_\R h(r) \frac{\Gamma'}{\Gamma}(1+ir) \; dr
	\bigg\}
\end{multline}
and
\begin{equation}\label{e:Eis}
	\Eis(\Gamma, \chi)
	=
	\frac{1}{4\pi} \int_\R h(r)
	\frac{(\varphi^{\Gamma})'}{\varphi^{\Gamma}}\!\left(\frac{1}{2}+ir\right) \; dr
	- \frac{1}{4} h(0) \tr \Phi^{\Gamma}\!\left(\frac{1}{2}\right).
\end{equation}
Note that all the sums and integrals are absolutely convergent (see \cite[Proposition~2.2]{BS07}).
In the remainder of this section we explain the notation appearing in \eqref{e:I}, \eqref{e:NEl}, \eqref{e:El}, \eqref{e:Cu} and \eqref{e:Eis}.
%

The set $M$ in $\NEl(\Gamma, \chi)$ and $\El(\Gamma, \chi)$ is given as
$$
	M = \{T\in \Gamma\;:\; T\text{ has no cusp of $\Gamma^+$ as a fixpoint}\},
$$
i.e.\ $M$ is the set of all $T\in \Gamma$ that do not fix any cusp of $\Gamma^+$ and
``$\{T\}\subset M$" denotes that we add over a set of representatives for the $\Gamma$-conjugacy classes in $M$.
We write $Z_\Gamma(T)$ for the centralizer of $T$ in $\Gamma$.
The sum in $\NEl(\Gamma, \chi)$ in \eqref{e:NEl} is over all non-elliptic conjugacy classes in $M$.
Thus $T$ is hyperbolic, reflection or a glide reflection,
and let $\alpha\in (-\infty, -1] \cup (1, \infty)$ be the unique number
such that
$T$ is conjugate within $\G$ to $\sm \alpha & 0 \\ 0 & 1\esm$.
Then we write
$$
	N(T) = |\alpha|.
$$
We denote by $T_0$ some hyperbolic element or a glide reflection in $Z_\Gamma(T)$.
Then the infinite cyclic group
$$
	[T_0] = \{T_0^n\;:\; n\in \Z\}
$$
has finite index $[Z_\Gamma(T): [T_0]]$ in $Z_\Gamma(T)$
and the ratio
$$
	\frac{\log N(T_0)}{[Z_\Gamma(T):[T_0]]}
$$
depends only on $T$ and not on our choice of $T_0$.
The sum in $\El(\Gamma, \chi)$ in \eqref{e:El} is over all elliptic conjugate classes in $M$,
and we write $\theta(T)$ for the unique number $\theta\in (0, \frac{\pi}{2}]$ such that $T$ is conjugate within $\G$ to $\sm \cos\theta & \sin\theta\\ -\sin\theta & \cos\theta\esm$.
%

Now we explain the notations appearing in $\Cu(\Gamma, \chi)$ in \eqref{e:Cu}.
Let
$$
	\eta_1, \ldots, \eta_\kappa \in \partial\HH = \R\cup\{\infty\}
$$
be a set of representatives of the cusps of $\Gamma^+\bsl \HH$,
one from each $\Gamma^+$-equivalence class.
For each $j\in \{1, \ldots, \kappa\}$ we choose $\N_j\in \G^+$ such that $\N_j(\eta_j) = \infty$
and the stabilizer $\Gamma_{\eta_j}^+$ is $[T_j]$ where
$$
	T_j = \N_j^{-1} \bpm 1 & -1 \\ 0 & 1\ebpm \N_j.
$$
We write $C_{\Gamma, \chi}$ for the set of open cusp representatives, viz.\
$$
	C_{\Gamma, \chi}  = \bigl\{j\in \{1, \ldots, \kappa\}\;:\; \chi(T_j) = 1\bigr\}.
$$

We fix, once and for all, an element $V\in \Gamma-\Gamma^+$.
For each $j\in\{1, \ldots, \kappa\}$ there exists $k(j)\in\{1, \ldots, \kappa\}$ and $U_j\in \Gamma^+$ such that
$$
	V \eta_j = U_j \eta_{k(j)}.
$$
Note that $k(j)$ is uniquely determined
by the condition that $\eta_{k(j)}$ is the representative of the cusp $V\eta_j$
and then $U_j$ is determined up to the right shifts with $T_{k(j)}$.
By \cite[p.~116 and (2.5)]{BS07},
$U_j$ satisfies
\begin{equation}\label{e:Uj}
	U_j = V\N_j^{-1} \bpm -1 & x_j \\ 0 & 1\ebpm \N_{k(j)} \in \Gamma^+, \text{ for some }x_j\in \R.
\end{equation}
We note in particular that $k(k(j)) = j$ and that $k(j)\in C_{\Gamma, \chi}$ if and only if $j\in C_{\Gamma, \chi}$.
For each $j\in\{1, \ldots, \kappa\}$ with $k(j) = j$, and each $v\in \Z$, we set
\begin{equation}\label{e:Tjv}
	T_{j, v} = \N_j^{-1} \bpm -1 & x_j+v\\ 0 & 1\ebpm \N_j = U_j^{-1} V T_j^v\in \Gamma.
\end{equation}
The last identity follows from \eqref{e:Uj}.
Then $T_{j, v}$ is a reflection fixing the point $\eta_j$.
Also the other fixpoint of $T_{j, v}$ in $\partial\HH$ must be a cusp;
we write it as $V_2\eta_k$ for some $V_2\in \Gamma^+$
and $k\in \{1, \ldots, \kappa\}$,
and then define the number $\mbc_{j, v}>0$ by the relation (cf.~\cite[(2.26)]{BS07})
\begin{equation}\label{e:mbc}
	\N_j V_2 \N_k^{-1} = \bpm * & * \\ \mbc_{j, v} & * \ebpm,
	\text{ where } \det(\N_j V_2 \N_k^{-1}) = 1.
\end{equation}

It now remains to explain the notations $\varphi^\Gamma$ and $\Phi^\Gamma$
appearing in $\Eis(\Gamma, \chi)$ in \eqref{e:Eis}.
For $j\in C_{\Gamma, \chi}$, let $E_j(z, s, \chi)$ be the Eisenstein
series for $\langle\Gamma^+, \chi\rangle$ associated to the cusp $\eta_j$:
\begin{equation}\label{e:Eis_j}
	E_j(z, s, \chi)
	=
	\sum_{W\in [T_j]\bsl \Gamma^+}
	\chi(W^{-1}) (\Im(\N_j W z))^s,
	\text{ for } z\in \HH \text{ and } \Re(s)>1,
\end{equation}
continued meromorphically to all $s\in \C$.
There is a natural $\Gamma$-analog of $E_j(z, s, \chi)$,
which for every $j\in C_{\Gamma, \chi}$ is given by
(cf.~\cite[(2.8)--(2.10)]{BS07}):
\begin{equation}\label{e:Eis_j_Gamma}
	E_j^{\Gamma}(z, s, \chi)
	=
	E_j(z, s, \chi)
	+
	\chi(V^{-1}) E_j (Vz, s, \chi).
\end{equation}

We fix, once and for all, a subset $R_{\Gamma, \chi} \subset C_{\Gamma, \chi}$ such that
\begin{equation}\label{e:R_Gammachi}
	\forall j\in C_{\Gamma, \chi}:
	\begin{cases}
	\text{ exactly one of } j, k(j)\in R_{\Gamma, \chi} & \text{if } k(j)\neq j, \\
	j\in R_{\Gamma, \chi} \iff \chi(V^{-1}U_j)=1 & \text{if } k(j)=j.
	\end{cases}
\end{equation}
For each $j\in C_{\Gamma, \chi}$ with $k(j) = j$,
it holds that $\chi(V^{-1} U_j) = \pm1$.
Hence if $j\notin R_{\Gamma, \chi}$ then $E_j^{\Gamma}(z, s, \chi) \equiv 0$.
In \cite[(2.32)]{BS07}, meromorphic functions $\varphi_{j, \ell}^{\Gamma}(s)$ for $j, \ell\in R_{\Gamma, \chi}$
are constructed so that for each $j\in R_{\Gamma, \chi}$,
\begin{equation}\label{e:Eis_avr_FE}
	E_j^{\Gamma}(z, 1-s, \chi)
	=
	\sum_{\ell\in R_{\Gamma, \chi}} \varphi_{j, \ell}^{\Gamma}(1-s) E_\ell^{\Gamma}(z, s, \chi).
\end{equation}
The relations \eqref{e:Eis_avr_FE} together determine the functions
$\varphi_{j, \ell}^{\Gamma}(s)$ uniquely (cf.~\cite[p.~125(bottom)]{BS07}).
We also note that all $\varphi_{j, \ell}^{\Gamma}(s)$ are holomorphic along the line $\Re(s) = \frac{1}{2}$.
The $|R_{\Gamma, \chi}|\times |R_{\Gamma, \chi}|$ matrix
$$
	\Phi^{\Gamma}(s) = (\varphi_{j, \ell}^{\Gamma}(s))_{j, \ell\in R_{\Gamma, \chi}}
$$
is called the scattering matrix for $\Gamma$ and we write
$$
	\varphi^{\Gamma}(s) = \det\Phi^\Gamma(s).
$$
We point out that this determinant $\varphi^{\Gamma}(s)$, as well as the trace of $\Phi^{\Gamma}(s)$
which also appears in \eqref{e:Eis}
are independent of which ordering of $R_{\Gamma, \chi}$ we chose when defining the matrix $\Phi^{\Gamma}(s)$.

\subsection{Trace formula for $\Gamma_0(N)$, $\chi$}
Throughout this section we will use the convention that all matrix representatives for elements in $\G = \PGL_2(\R)$ are taken to have determinant $1$ or $-1$.
Let $N$ be an arbitrary positive integer and set $\Gamma^+= \Gamma_0(N)\subset \G^+$.
Fix
$$
	V = \bpm -1 &0 \\ 0 & 1\ebpm \in \G
$$
and note that $V^2 =1_2$, the identity matrix, and $V\Gamma^+V^{-1} = \Gamma^+$.
Hence $\Gamma = \Gamma_0^\pm (N) = \langle\Gamma^+, V\rangle\subset \G$
is a supergroup of $\Gamma^+$ of index $2$.
We will use the standard notation $\omega(N)$ for the number of primes dividing $N$.

Let $\chi$ be a Dirichlet character modulo $N$.
For any divisor $\alpha\mid N$ satisfying $\gcd(\alpha, N/\alpha)=1$,
we define
$$
	\chi_\alpha(x) = \chi(y)
$$
for $y\equiv_{\alpha} x$ and $y\equiv_{N/\alpha}1$.
By the Chinese remainder theorem,
such $y$ is uniquely determined modulo $N$ and $\chi_\alpha$ is a Dirichlet character modulo $\alpha$.
It follows that $\chi = \chi_\alpha \chi_{N/\alpha}$ for all such divisors $\alpha$.
For any prime $p$, we set
$$
	\chi_p = \chi_{\gcd(N, p^{\infty})}.
$$
It follows that $\chi = \prod_{p\mid N} \chi_p$
(identity of functions on $\Z$); also $\chi_p \equiv 1$ for all $p\nmid N$.

Let us agree to call $\chi$ pure if $\chi_p(-1) = 1$ for every odd prime $p$.
Note that if $\chi$ is even then $\chi$ is pure if and only if $\chi_p(-1)=1$ for all primes $p$.
From now on we fix $\chi$ to be an even Dirichlet character modulo $N$,
and we set $q=\cond(\chi)$.
We introduce two indicator functions $I_\chi$ and $I_{q, 4}$ via
\begin{align}
	& I_{\chi} = \delta_{\chi\text{ is pure}};\\
	& I_{q, 4} = \delta_{p\equiv_41\;\forall p\mid q}.
\end{align}

We view $\chi$ as a character on $\Gamma^+$ via
$$
	\chi\!\left(\sm a & b\\ c & d\esm\right) = \chi(d) \text{ for } \sm a & b\\ c& d\esm\in\Gamma^+.
$$
This character can be extended in two ways to a character of $\Gamma$: either by $\chi(V)=1$ or $\chi(V)=-1$.
These extensions are explicitly given by
\begin{equation}\label{e:chi-epsilon}
	\chi^{(\ve)}(T) = \det(T)^{\ve} \chi(d), \text{ for } T=\sm a & b\\ c& d\esm\in \Gamma \text{ where } \ve\in\{0, 1\}.
\end{equation}

Let $\Tr(\G, \chi^{(\ve)})$ be the trace formula for $(\Gamma_0^\pm (N), \chi^{(\ve)})$.
For $n\in \{\pm1\}$, we have
\begin{equation}\label{e:Tn_epsilon}
	\sum_{\lambda\geq0} \tr T_n|_{\A_\lambda(\chi)} h\left(\sqrt{\lambda-\tfrac{1}{4}}\right)
	=
	\Tr(\G, \chi^{(0)})+n \Tr (\G, \chi^{(1)}).
\end{equation}


For each prime $p\mid N$, define
\begin{equation}\label{e:Psi1}
	\Psi_1(p^{e_p}, p^{s_p})
	=
	\begin{cases}
	2p^{e_p-s_p} & \text{if } e_p < 2s_p, \\
	p^{\lf\frac{e_p}{2}\rf}+ p^{\lf\frac{e_p-1}{2}\rf} & \text{if } e_p \geq 2s_p, 	
	\end{cases}
\end{equation}
\begin{equation}\label{e:Psi2}
	\Psi_2(p^{e_p}, p^{s_p})
	=
	\begin{cases}
	1 & \text{if } e_p=0, \\
	\begin{cases}
	2e_p & \text{if } e_p>0 \text{ and } s_p=0, \\
	2(e_p-2s_p+1) & \text{if } e_p\geq 2s_p \text{ and } s_p>0, \\
	\end{cases}
	& \text{if } p \text{ is odd}, \\
	\begin{cases}
	e_p+1 & \text{if } e_p\in\{1, 2\} \text{ and } s_p=0, \\
	2(e_p-1) & \text{if } e_p\in\{4, 3\} \text{ and } s_p=0, \\
	4(e_p-3) & \text{if } e_p\geq 5 \text{ and } s_p=0, \\
	4(e_p-2s_p-1) & \text{if } e_p\geq 2s_p+1 \text{ and } s_p\geq 3, \\
	\end{cases}
	& \text{if } p=2, \\
	0 & \text{otherwise,}
	\end{cases}
\end{equation}
\begin{equation}\label{e:tPsi2}
	\tPsi_2(p^{e_p}, \chi_p)
	=
	\begin{cases}
	1 & \text{if } e_p=0, \\
	\min\{e_2+1, 4\} & \text{if } p=2 \text{ and } s_2=0, \\
	2 & \text{if } p\equiv_4-1 \text{ and } s_p=0, \\
	\chi_p(\sqrt{-1}) \Psi_2(p^{e_p}, p^{s_p}) & \text{if } p\equiv_41 \text{ and } \chi_p \text{ even}, \\
	0 & \text{otherwise}
	\end{cases}
\end{equation}
(where $\sqrt{-1}$ is a square root of $-1$ in $\Z/p\Z$),
and
\begin{equation}\label{e:Psi3}
	\Psi_3(p^{e_p}, p^{s_p})
	=
	\begin{cases}
	0 & \text{if } e_p=0, \\
	p^{e_p-s_p}(4s_p-1) & \text{if } e_p < 2s_p, \\
	\left(p^{\lf\frac{e_p}{2}\rf}+p^{\lf\frac{e_p-1}{2}\rf}\right)
	\left(e_p-\frac{1}{p-1}\right)
	+
	\frac{2}{p-1} + \delta_{s_p>0} \big(p^{s_p-1} +2\frac{p^{s_p-1}-1}{p-1}\big)
	& \text{if } e_p\geq 2s_p.
	\end{cases}
\end{equation}
For $j\in \{1, 2\}$ we set $\Psi_j(N, q) = \prod_{p\mid N} \Psi_j(p^{e_p}, p^{s_p})$.
Similarly, $\tPsi_2(N,\chi)=\prod_{p\mid N}\tPsi_2(p^{e_p}, \chi_p)$.

Let
\begin{equation}\label{e:Omega1}
	\Omega_1(N, q)
	=
	\begin{cases}
	1 & \text{if } 0\leq e_2 \leq 1 \text{ or } s_2=e_2, \\
	\frac{3}{2} & \text{if } e_2=2\text{ and } s_2 < e_2, \\
	2 & \text{if } e_2\geq 3 \text{ and } s_2 < e_2,
	\end{cases}
\end{equation}
\begin{equation}\label{e:Omega2}
	\Omega_2(N, q)
	=
	\begin{cases}
	\frac{1}{2}-\frac{3}{4}e_2
	& \text{if } 0\leq e_2 \leq 2, \\
	-2 & \text{if } e_2\geq 3, s_2=0, \\
	-2s_2 & \text{if } e_2> s_2 \geq 3, \\
	-e_2 & \text{if } e_2=s_2\geq 3.
	\end{cases}
\end{equation}

Finally we define
\begin{equation}\label{e:Phi}
	\Phi_\pm(p^{e_p}, m)
	=
	\Psi_1(p^{e_p}, p^{w_\pm(m)})
\qquad\text{with }\:
w_\pm(m) = \max\{s_p, e_p-\ord_p(\pm m^2-1)\}.
\end{equation}
%

\begin{theorem}\label{thm:STF_N}
For $n\in\{\pm1\}$ we have
\begin{equation}\label{e:STF_N}
\sum_{\lambda\geq0} \tr T_n|_{\A_\lambda(\chi)} h\left(\sqrt{\lambda-\tfrac{1}{4}}\right)
	=
	\I(\Gamma, \chi; n)
	+
	(\NEl+\El)(\Gamma, \chi; n)
	+
	(\Cu+\Eis)(\Gamma, \chi; n),
\end{equation}
where
$$
	\I(\Gamma, \chi; n)
	=
	\frac{1+n}{2} \frac{N\prod_{p\mid N} (1+p^{-1})}{12}
	\int_\R rh(r) \tanh(\pi r) \; dr,
$$
$$
	(\NEl+\El)(\Gamma, \chi; n)
	=
	\sum_{\substack{t\in \Z, t^2-4n=d\ell^2, \\ \sqrt{t^2-4n}\notin \Q}}
	A_{t, n}[h] \boldh(d)
	\prod_{p} \boldS_p(p^{e_p}, \chi_p; t, n),
$$
where $A_{t, n}[h]$
is defined in \eqref{e:Atnh},
$\boldh(d)$ is the narrow class number of $\Q(\sqrt d)$,\linebreak
$\boldS_p(p^{e_p}, \chi_p; t, n)$ is given in \eqref{e:boldSp_e=0_0} and Lemma~\ref{lem:boldSp},
\begin{multline}\label{e:Cu+Eis_N_1}
	(\Cu+\Eis)(\Gamma, \chi; 1)
	\\
	=
	\Psi_1(N, q) \bigg\{ \frac{1}{4} h(0)
	-\frac{1}{2\pi} \int_\R h(r) \bigg(\frac{\Gamma'}{\Gamma}\!\left(\frac{1}{2}+ir\right) + \frac{\Gamma'}{\Gamma}(1+ir)\bigg) \; dr
	\bigg\}
	- \frac{1}{4} I_\chi \Psi_2(N, q) h(0)
	\\
	+
	\bigg\{ \Psi_1(N, q) \log\!\left(\frac{\pi}{2}\right)
	-
	\sum_{p\mid N} \Psi_3(p^{e_p}, p^{s_p}) \prod_{p'\mid N, p'\neq p} \Psi_1({p'}^{e_{p'}}, {p'}^{s_{p'}}) \log p\bigg\}
	g(0)
	\\
	+
	2\sum_{m\geq1, \gcd(m,q)=1}\frac{\Lambda(m)}{m} \prod_{p\mid N,p\nmid m} \Big[(\Re\chi_p(m))\Phi_+(p^{e_p}, m)\Big]\,
	g(2\log m)
\end{multline}
and
\begin{multline}\label{e:Cu+Eis_N_-1}
	(\Cu + \Eis) (\Gamma, \chi; -1)
	=
	\frac{1}{4}I_\chi \big\{ 2^{\omega(N)} \Omega_1(N, q) - I_{q, 4} \tPsi_2(N, \chi)\big\} h(0)
	\\
	-
	I_\chi 2^{\omega(N)}\Omega_1(N, q)
	\frac{1}{2\pi} \int_\R h(r) \frac{\Gamma'}{\Gamma}(1+ir) \; dr\\
	+
	I_\chi 2^{\omega(N)}
	\bigg\{
	\Omega_1(N, q) \bigg(\log \pi - \sum_{p\mid N, p>2} \max\{\tfrac{1}{2}, s_p\} \log p\bigg)
	+
	\Omega_2(N, q) \log 2
	\bigg\} g(0)
	\\
	+
	2\sum_{m\geq1, \gcd(m,q)=1} \frac{\Lambda(m)}{m} \prod_{p\mid N,p\nmid m}
\Big[(\{\Re/i\Im\}\chi_p(m)) \Phi_-(p^{e_p}, m)\Big]\,
	g(2\log m),
\end{multline}
where $\{\Re/i\Im\}\chi_p(m) = \frac{1}{2}\big(\overline{\chi_p(m)}+\chi_p(-m)\big)$.
\end{theorem}


\subsection{Identity $I(\Gamma, \chi)$}
\begin{lemma}\label{lem:I_N}
For $n\in \{\pm1\}$, we have
$$
	I(\Gamma, \chi; n)
	=
	\frac{1+n}{2} \frac{N \prod_{p\mid N} (1+p^{-1})}{12}
	\int_\R rh(r) \tanh(\pi r) \; dr.
$$
\end{lemma}
\begin{proof}
It is well known that $\Area(\Gamma^+\bsl \HH) = \frac{\pi}{3} N\prod_{p\mid N} (1+p^{-1})$ (cf.~\cite[Theorem~4.2.5(2)]{Miy89}),
and the area of $\Gamma\bsl \HH$ is half as large.
By \eqref{e:I}, we have
$$
	\I(\Gamma, \chi^{(\ve)}) = \frac{N \prod_{p\mid N}(1+p^{-1})}{24} \int_\R h(r) r\tanh(\pi r) \; dr
\qquad\text{for }\:\ve=\pm1.
$$
\end{proof}

\subsection{Non-cuspidal contributions $\NEl+\El(\Gamma, \chi)$}

The aim of this section is to prove the following proposition.
\begin{proposition}\label{prop:NEll+Ell_N}
For $n\in \{\pm1\}$, we have
\begin{equation}\label{e:NEll+Ell_N}
	\NEl(\Gamma, \chi; n)+\El(\Gamma, \chi; n)
	=
	\sum_{\substack{ t\in \Z, t^2-4n=d\ell^2, \\\sqrt{t^2-4n}\notin \Q}}
	A_{t, n}[h]\, \boldh(d) \prod_p \boldS_p(p^{e_p}, \chi_p; t, n),
\end{equation}
where $d$ is a fundamental discriminant and $\ell\in \Z_{>0}$.
\end{proposition}

The method we use to enumerate the conjugacy classes appearing
in these sums is well known, cf., e.g., \cite{Eic56}, \cite{Vig80} and \cite{Miy89}.
Here we will follow the setup in \cite{Miy89} fairly closely
(as was done in \cite{BS07} in the case of $N$ squarefree).

We start by introducing some notations.
Set
$$
	R = \left\{
	\bpm a & b\\ c& d\ebpm \in \M_2(\Z)\;:\; c\equiv_N0
	\right\}
$$
and
$$
	R^1=\{T\in R\;:\; \det(T)=1\}.
$$
Let
$$
	P = \left\{\langle n, t\rangle\in \{\pm1\} \times \Z\;:\;
	\sqrt{t^2-4n}\notin \Q,\;
	\exists T_{n, t}\in R, \; \det(T_{n, t})=n \text{ and } \tr(T_{n, t}) = t\right\}.
$$
For each $\langle n, t\rangle\in P$, we fix one such element $T_{n, t}\in R$.

Given $\langle n, t\rangle\in P$, we let $d$ and $\ell$ be the unique integers such that
$t^2-4n=d\ell^2$, $\ell\in \Z_{\geq 1}$ and $d$ is a fundamental discriminant
(that is $d\equiv_41$, is squarefree and $d\neq 1$ or else $d\equiv_40$, $d/4$ is squarefree and $d/4\not\equiv_41$).
Then the subalgebra $\Q[T_{n, t}]\subset \M_2(\Q)$ is isomorphic to the quadratic field $\Q(\sqrt{d})$.
For any positive integer $f$, we write $\rr[f]$ for the order
$\Z+f\frac{d+\sqrt{d}}{2}\Z$ in $\Q(\sqrt{d})$.
Using $\Q[T_{n, t}]\cong \Q(\sqrt{d})$ (either of the two possible isomorphisms),
we let $\rr[f]$ denote also the corresponding order in $\Q[T_{n, t}]$.
For any order $\rr$ in $\Q[T_{n, t}]$ we set
$$
	C(T_{n, t}, \rr)
	=
	\{ \delta T_{n, t} \delta^{-1}\;:\;
	\delta\in \GL_2(\Z), \; \Q[T_{n, t}]\cap \delta^{-1} R\delta=\rr
	\}.
$$
Note that $C(T_{n, t}, \rr)$ is closed under conjugation by elements in $R^1$.
We denote by $C(T_{n, t}, \rr) \ccg R^1$
a set of representatives for the inequivalent $R^1$-conjugacy classes in $C(T_{n, t}, \rr)$.
Also for $U=\sm a& b\\ c& d\esm \in C(T_{n, t}, \rr)$, we write
$$
	\chi^{(\ve)}(U) = \det(U)^{\ve} \chi(d).
$$
For every $f\mid \ell$, it holds that $C(T_{n, t}, \rr[f])\subset R$
and using this one checks that $\chi^{(\ve)}(U)$ for $U\in C(T_{n, t}, \rr[f])$
only depends on the $R^1$-conjugacy class of $U$.

Let $\boldh(\rr[f])$ be the narrow class number for $\rr[f]$.
We also define
\begin{equation}\label{e:Atnh}
	A_{t, n}[h] =
	\begin{cases}
	\frac{\log \epsilon_1}{\sqrt{t^2-4n}} g\left(\log \frac{(|t|+\sqrt{t^2-4n})^2}{4}\right),
	& \text{ if } t^2-4n>0, \\
	\frac{2}{|\rr[1]^1| \sqrt{4n-t^2}} \int_\R \frac{e^{-2r\cdot \arccos(|t|/2)}}{1+e^{-2\pi r}} h(r) \; dr,
	& \text{ if } t^2-4n < 0,
	\end{cases}
\end{equation}
where in the first case $\epsilon_1>1$ is the proper fundamental unit in $\Q(\sqrt{d})$.

Following the discussion in \cite[pp.\ 136-137]{BS07},
we see that $\NEl(\Gamma, \chi^{(\ve)}) + \El(\Gamma, \chi^{(\ve)})$ can be collected as
\begin{multline}\label{e:NEl+El_N_1}
	\NEl(\Gamma, \chi^{(\ve)}) + \El(\Gamma, \chi^{(\ve)})
	\\
	=\frac12
	\sum_{\langle n, t\rangle\in P}
	\bigg\{\sum_{f\mid \ell}
	\left[\rr[1]^1:\rr[f]^1\right]
	\bigg( \sum_{U\in C(T_{n, t}, \rr[f]) \ccg  R^1} \chi^{(\ve)}(U)\bigg)
	\bigg\}
	A_{t, n}[h]
	\begin{cases}
	1 & \text{if } t^2-4n>0, \\
	\frac{1}{2} & \text{if } t^2-4n<0.
	\end{cases}
\end{multline}
%

\begin{remark}
The expression \eqref{e:NEl+El_N_1} would of course look slightly nicer
if we did not include the factor ``$2$" in the definition of $A_{t, n}[h]$ \eqref{e:Atnh}
in the case $t^2-4n<0$.
However our definition makes the final expression slightly simpler.
\end{remark}

The following lemma generalizes \cite[Lemma~3.4]{Str00}.
\begin{lemma}\label{lem:ntP}
Let $n, t$ be any integers satisfying $\sqrt{t^2-4n}\notin \Q$,
and let $d, \ell$ be the unique integers such that $t^2-4n=d\ell^2$, $\ell>0$ and $d$ is a fundamental discriminant.
Then there exists $T\in R$ satisfying $\det(T)=n$ and $\tr(T)=t$
if and only if, for each prime $p\mid N$, we have
\begin{equation}\label{e:ntP}
	\ord_p(\ell) \geq \lc\frac{e_p}{2}\rc \text{ or }
	\left(\frac{d}{p}\right)=1
	\text{ or }
	\left[ 2\nmid e_p, \; \ord_p(\ell) \geq \lf\frac{e_p}{2}\rf \text{ and } p\mid d\right].
\end{equation}
Here $e_p = \ord_p(N)$.
\end{lemma}
\begin{proof}
Such an element $T\in R$ exists if and only if there is some $a\in \Z$
satisfying $a(t-a)-n\equiv_N0$, and this holds if and only if the same congruence equation is solvable modulo $p^{e_p}$ for each prime $p\mid N$.
The lemma now follows by completing the square in the expression $a(t-a)-n$,
and splitting into the classes $2\mid e_p$ and $2\nmid e_p$,
as well as $p=2$ and $p>2$.
\end{proof}

The $R^1$-conjugacy classes in each $C(T_{n, t}, \rr)$ can be enumerated using a local-to-global principle.
Let $T=T_{n, t}$ and $\rr = \rr[f]$ for some $f\mid \ell$.
For each prime $p$, we set
$$
	C_p(T, \rr)
	=
	\left\{x T x^{-1}\;:\; x\in \GL_2(\Q_p), \; \Q_p[T]\cap x^{-1}R_p x=\rr_p\right\},
$$
where
$$
	R_p = \left\{\bpm a& b\\ c& d\ebpm \in \M_2(\Z_p)\;:\;
	\gamma\equiv_N0\right\}
$$
is the closure of $R$ in $\M_2(\Q_p)$ and $\rr_p$ is the closure of $\rr$ in $\Q_p(\sqrt{d})\cong\Q_p[T]$.
One checks that $C_p(T, \rr)$ is closed under $R_p^\times$-conjugation.
We also set
$$
	C_\infty(T, \rr) = \left\{xT x^{-1}\;:\; x\in \GL_2(\R)\right\},
$$
which is closed under $R_\infty^\times$-conjugation, where we have set $R_\infty^\times = \GL_2^+(\R)$.
Let us now mildly alter our previous definition of $C(T, \rr) \ccg R^1$ (a set of representatives), so as to instead
let $C(T, \rr) \ccg R^1$ denote the set of $R^1$-conjugacy classes in $C(T, \rr)$.
Similarly for each place $v$ of $\Q$, we let $C_v(T, \rr) \ccg R_v^\times$ be the set of $R_v^\times$-conjugacy classes in $C_v(T, \rr)$.
Clearly we have a natural map
\begin{equation}\label{e:theta_map}
	\theta: C(T, \rr)  \ccg  R^1 \to \prod_{v} C_v(T, \rr) \ccg R_v^\times.
\end{equation}
By \cite[Lemma~6.5.2]{Miy89} (trivially generalized as noted in \cite{BS07}),
the map $\theta$ is surjective, and in fact $\theta$ is exactly $\boldh(\rr)$-to-$1$.

It remains to understand each factor in the right hand side of \eqref{e:theta_map}.
For $v=\infty$, we have, as in \cite[(6.6.1)]{Miy89},
$$
	\left|C_\infty(T, \rr[f]) \ccg R_\infty^\times\right|
	=
	\begin{cases}
	1 & \text{if } d>0, \\
	2 & \text{if } d<0.
	\end{cases}
$$
Also for primes $p\nmid N$, we have $R_p \subset \M_2(\Z_p)\subset \M_2(\Q_p)$ and thus
$$
	\left|C_p(T, \rr) \ccg  R_p^\times\right| = 1
$$
(cf., e.g., \cite[Theorem~6.6.7]{Miy89}).

Finally, if $p\mid N$, let $e = \ord_p(N)$ and $\rho = \ord_p(\ell/f)$. 
For $\alpha\in \Z_{\geq 1}$ let 
$$
	\Omega(p^{\alpha}; n, t)
	=
	\left\{ \xi\in \Z_p\;:\; \xi^2-t\xi+n\equiv_{p^{\alpha}} 0\right\}.
$$
Then by \cite[Theorem~6.6.6]{Miy89}
\footnote{We correct for a misprint in \cite[Theorem~6.6.6]{Miy89}:
$\Omega'/p^{\nu+\rho+1}$ should be replaced by $\Omega'/p^{\nu+\rho}$},
a complete set of representatives for $C_p(T, \rr) \ccg R_p^\times$
(where $T=T_{n, t}$, $\rr = \rr[f]$ and $f\mid \ell$)
is given by
\begin{multline*}
	\left\{\bpm \xi & p^{\rho}\\-\frac{\xi^2-t\xi+n}{p^{\rho}} & t-\xi\ebpm\;:\;
	\xi\in \Omega(p^{e+2\rho}; n, t)/p^{e+\rho}\right\}
	\\
	\bigcup
	\begin{cases}
	\left\{\bpm t-\xi & -\frac{t^2-t\xi+n}{p^{e+\rho}}\\ p^{e+\rho} & \xi\ebpm\;:\;
	\xi\in \Omega(p^{e+2\rho+1}; n, t)/p^{e+\rho}\right\}
	& \text{if } p^{2\rho+1} \mid \ell^2 d, \\
	\emptyset & \text{otherwise.}
	\end{cases}
\end{multline*}
The sets $\Omega(p^{e+2\rho}; n, t)$ and $\Omega(p^{e+2\rho+1}; n, t)$ are both closed under addition with any element from $p^{e+\rho}\Z_p$,
and $\Omega(p^{e+2\rho}; n, t)/p^{e+\rho}$ and $\Omega(p^{e+2\rho+1}; n, t)/p^{e+\rho}$
denote complete sets of representatives for
$\Omega(p^{e+2\rho}; n, t)\mod p^{e+\rho} \Z_p$
and $\Omega(p^{e+2\rho+1}; n, t) \mod p^{e+\rho}\Z_p$,
respectively.

Using the above facts together with $\chi = \prod_{p\mid N} \chi_p$
and the fact that $\xi^2-t\xi+n = \xi(\xi-t)+n$
is invariant under $\xi\mapsto t-\xi$,
it follows that, for any $f\mid \ell$,
\begin{multline}\label{e:sum_chi(U)}
	\sum_{U\in C(T_{n, t}, \rr[f]) \ccg  R^1} \chi^{(\ve)}(U)
	=
	\prod_{p\mid N}
	\bigg(
	\sum_{\xi\in \Omega(e^{e+2\rho}; n, t)/p^{e+\rho}} \chi_p(\xi)
	+
	\delta_{p^{2\rho+1}\mid \ell^2d} \sum_{\xi\in \Omega(p^{e+2\rho+1}; n, t) / p^{e+\rho}} \chi_p(\xi)
	\bigg)
	\\
	\times
	\boldh(\rr[f]) n^{\ve}
	\begin{cases}
	1 & \text{if } d>0, \\
	2 & \text{if } d<0.
	\end{cases}
\end{multline}

For $p\mid N$ with $e=\ord_p(N)$ and $f=\ord_p(\ell)$, set
\begin{multline}\label{e:boldSp_0}
	\boldS_p(p^{e}, \chi; t, n)
	=
	\sum_{\xi\in \Omega(p^{e+2f}; n, t)/p^{e+f}} \chi(\xi)
	+
	\delta_{p\mid d} \sum_{\xi\in \Omega(p^{e+2f+1}; n, t)/p^{e+f}} \chi(\xi)
	\\
	+
	\left(p-\left(\frac{d}{p}\right)\right) \sum_{\beta\in \{0, 1\}}
	\sum_{k=1}^f p^{k-1}
	\sum_{\xi\in \Omega(p^{e+2f-2k+\beta}; n, t)/p^{e+f-k}} \chi(\xi).
\end{multline}
Furthermore for $p\nmid N$, set
\begin{equation}\label{e:boldSp_e=0_0}
	\boldS_p(1, 1; t, n) = 1+\left(p-\left(\frac{d}{p}\right)\right)
	\frac{p^{\ord_p(\ell)}-1}{p-1}.
\end{equation}
Since
$$
	\boldh(\rr[f])\left[\rr[1]^1:\rr[f]^1\right]
	=
	\boldh(d) f\prod_{p\mid f} p^{-1}\left(p-\left(\frac{d}{p}\right)\right),
$$
for each $\langle n, t\rangle\in P$, we have
\begin{equation}\label{e:sum_f|ell_prod}
	\sum_{f\mid \ell}
	\left[\rr[1]^1:\rr[f]^1\right]
	\left( \sum_{U\in C(T_{n, t}, \rr[f]) \ccg  R^1} \chi^{(\ve)}(U)\right)
	=
	\bigg(\prod_p \boldS_p(p^{e_p}, \chi_p; t, n)\bigg)
	n^{\ve} \boldh(d)
	\begin{cases}
	1 & \text{if } d>0, \\
	2 & \text{if } d<0.
	\end{cases}
\end{equation}
Applying \eqref{e:sum_f|ell_prod} to \eqref{e:NEl+El_N_1},
then by \eqref{e:Tn_epsilon},
\begin{equation}\label{e:NEl+El_N_2}
	\NEl(\Gamma, \chi; n) + \El(\Gamma, \chi; n)
	=
	\sum_{\substack{t\in\Z\\(\langle n, t\rangle\in P)}}
	A_{t, n}[h] \boldh(d) \prod_p \boldS_p(p^{e_p}, \chi_p; t, n).
\end{equation}
In the remainder of this section, we prove Lemma~\ref{lem:boldSp} and get an explicit formula for $\boldS_p(p^e, \chi; t, n)$.
Let us first record the general solution to the congruence equation $\xi^2-t\xi+n\equiv0$ modulo a prime power in the following lemma.
\begin{lemma}\label{lem:Omega_p}
Let $p$ be a prime and $\alpha$ a positive integer.
Then
\begin{multline}\label{e:Omega_p}
	\Omega(p^{\alpha}; n, t)
	\\
	=
	\begin{cases}
	\left\{ \frac{t+p^{\lc\frac{\alpha}{2}\rc} \omega}{2} + p^{\lc\frac{\alpha}{2}\rc}u\;:\; u\in \Z_p\right\}
	& \text{if } \ord_p(t^2-4n) \geq \alpha+\max\{\ord_p(d)-1, 0\}, \\
	\left\{ \frac{t}{2}\pm \frac{\sqrt{d}\ell}{2} + p^{\alpha-\ord_p(\ell)}u\;:\;
	u\in \Z_p\right\}
	& \text{if } \ord_p(t^2-4n) < \alpha \text{ and } \left(\frac{d}{p}\right)=1, \\
	\emptyset & \text{otherwise,}
	\end{cases}
\end{multline}
where
$$
	\omega =
	\begin{cases}
	1 & \text{if } p=2, \alpha=\ord_p(t^2-4n)-1 \text{ and } 8\nmid d, \\
	& \text{or } p=2, \alpha=\ord_p(t^2-4n) \text{ and } 2\nmid d, \\
	& \text{or } p\text{ and } t \text{ odd}, \\
	0 & \text{otherwise.}
	\end{cases}
$$
and $t^2-4n = d\ell^2$ with $d$ a fundamental discriminant and $\ell\in \Z$.
\end{lemma}
\begin{proof}
Suppose first that $p$ is odd.
For $\alpha\geq 1$, since $\xi\in \Omega(p^{\alpha}; n, t)$,
we have
$$
\xi^2-t\xi+n \equiv \left(\xi-\frac{t}{2}\right)^2 -\frac{t^2}{4}+n
\equiv 0\pmod*{p^{\alpha}},
$$
so that
\begin{equation}\label{e:eq_p_comp}
\left(\xi-\frac{t}{2}\right)^2 \equiv \frac{d\ell^2}{4} \pmod*{p^{\alpha}}.
\end{equation}
When $\ord_p(d\ell^2) \geq \alpha$ this yields
$\left(\xi-\frac{t}{2}\right)^2\equiv 0\pmod*{p^{\alpha}}$. 
For some $u\in \Z_{p}$, we have
$$
\xi = \frac{t}{2}+p^{\lc\frac{\alpha}{2}\rc} u. 
$$
Next consider the case when $\ord_p(d\ell^2) < \alpha$.
We must have
$p\nmid d$ and $\ord_p(d\ell^2)=2\ord_p(\ell) < \alpha$ 
for $\xi$ to satisfy \eqref{e:eq_p_comp}. 
In this case \eqref{e:eq_p_comp} has a solution only if $\left(\frac{d}{p}\right)=1$,
and then $\ord_p\left(\xi-\frac{t}{2}\right)=\ord_p(\ell) = \frac{\ord_p(d\ell^2)}{2}$.
Therefore, for some $u\in \Z_p$, 
$$
\xi = \frac{t}{2}\pm \frac{\sqrt{d}\ell}{2} + p^{\alpha-\ord_p(\ell)} u.
$$

Suppose now that $p=2$.
Since $n\in \{\pm1\}$ and
$$
	\xi^2-t\xi+n = \xi(\xi-t) +n\equiv 0\pmod*{2^{\alpha}},
$$
both $\xi$ and $\xi-t$ must be odd,
so $t$ is even.
Note that when $\alpha=1$ and $t$ is even,
$$
	\Omega(2; n, t)
	=
	\left\{1+2u\;:\; u\in \Z_2\right\}.
$$
From now on, we assume that $\alpha\geq 2$ and $t$ is even.
Then $d\ell^2 = t^2-4n = 4\left(\frac{t^2}{4}-n\right)$, so $4\mid d$ or $2\mid \ell$.
Therefore
\begin{equation}\label{e:eq_2_comp}
	\left(\xi-\frac{t}{2}\right)^2 \equiv \frac{t^2-4n}{4} \equiv
	\frac{d\ell^2}{4} \pmod*{2^{\alpha}}.
\end{equation}
If $\ord_2(d\ell^2) -2 \geq \alpha$, we have
$\left(\xi-\frac{t}{2}\right)^2\equiv 0\pmod*{2^{\alpha}}$,
so
$$
	\Omega(2^{\alpha}; n, t)
	=
	\left\{
	\frac{t}{2}+2^{\lc\frac{\alpha}{2}\rc} u\;:\; u\in \Z_2
	\right\}.
$$

Next consider the case when $\ord_2(d\ell^2)-2 < \alpha$. 
When $4\mid d$, we have $(\ord_2(d)-2) + 2\ord_2(\ell) < \alpha$.
Dividing both sides of \eqref{e:eq_2_comp} by $2^{2\ord_2(\ell)}$, we get
$$
	\left(\frac{\xi-\frac{t}{2}}{2^{\ord_2(\ell)}}\right)^2
	\equiv
	\frac{d}{4}\left(\frac{\ell}{2^{\ord_2(\ell)}}\right)^2
	\pmod*{2^{\alpha-2\ord_2(\ell)}}.
$$
Note that $d/4\equiv_4 2$ or $3$, i.e.\ $d/4$ is not a square modulo $4$,
which implies that $\Omega(2^{\alpha}; n, t)=\emptyset$ if $\alpha-2\ord_2(\ell)\geq 2$.
If $\alpha-2\ord_2(\ell)=1$ then
$$
	\frac{\xi-\frac{t}{2}}{2^{\ord_2(\ell)}}
	\equiv
	\frac{d}{4} \pmod*{2}.
$$
Thus, when $\alpha=2\ord_2(\ell)+1$ and $4\mid d$, we get
$$
	\Omega(2^{\alpha}; n, t)
	=
	\left\{\frac{t}{2}+2^{\ord_2(\ell)}\left(\frac{d}{4} + 2u\right)\;:\;
	u\in \Z_2\right\}.
$$
When $d\equiv 1\pmod*{4}$, we have $2\ord_2(\ell) -2< \alpha$.
Dividing both sides of \eqref{e:eq_2_comp} by $2^{2\ord_2(\ell) -2}$,
we get
$$
	\left(\frac{\xi-\frac{t}{2}}{2^{\ord_2(\ell)-1}}\right)^2
	\equiv
	d\left(\frac{\ell}{2^{\ord_2(\ell)}}\right)^2
	\pmod*{2^{\alpha-2\ord_2(\ell)+2}}.
$$
When $\alpha=2\ord_2(\ell)-1$ or $\alpha=2\ord_2(\ell)$, we have
$$
	\Omega(2^{\alpha}; n, t)
	=
	\left\{\frac{t}{2}+2^{\ord_2(\ell)-1} (1+2u)\;:\; u\in \Z_2\right\}.
$$
For $\alpha>2\ord_2(\ell)$, following a similar procedure to the case when $p$ is odd, 
we get 
$$
\Omega(2^{\alpha}; n, t)
=
\left\{ \frac{t}{2}\pm \frac{\sqrt{d}\ell}{2}+2^{\alpha-\ord_2(\ell)} u\;:\; u\in \Z_2 \right\}.
$$
\end{proof}

Next, by Lemma~\ref{lem:ntP}, if $\langle n, t\rangle\in \{\pm1\}\times \Z$
satisfies $\sqrt{t^2-4n}\notin \Q$ and $\langle n, t\rangle\notin P$,
then there is a prime $p\mid N$ such that either
\begin{itemize}
	\item $\left(\frac{d}{p}\right)=-1$ and $e_p>2\ord_p(\ell)$;
	or
	\item $p\mid d$ and $e_p > 2\ord_p(\ell) + 1$.
\end{itemize}
By inspection in \eqref{e:Omega_p}, we see that the summation condition
``$\langle n, t\rangle\in P$" in \eqref{e:NEl+El_N_2} may be replaced
by summation over all $\langle n, t\rangle\in\{\pm1\}\times \Z$
satisfying $\sqrt{t^2-4n}\notin \Q$, as in \eqref{e:NEll+Ell_N}.
%

\begin{lemma}\label{lem:avg_char}
Let $p$ be a prime, $s\in \Z_{\geq 0}$ and $\psi$ a primitive Dirichlet character of conductor $p^s$.
For $a, b\in \Z_{\geq 0}$, $a+b\geq s$, we have
\begin{equation}\label{e:avg_char}
	\sum_{u\pmod*{p^{b}}} \psi(x+p^{a} u)
	=
	\begin{cases}
	p^b \psi(x) & \text{if } s\leq a, \\
	0 & \text{otherwise.}
	\end{cases}
\end{equation}
\end{lemma}
\begin{proof}
Since $\psi$ is primitive, $\tau(\overline{\psi})\neq 0$ and we have
$$
	\psi(x+p^a u)
	=
	\frac{1}{\tau(\overline{\psi})}
	\sum_{\alpha\pmod*{p^s}} \overline{\psi(\alpha)} e^{2\pi i (x+p^a u) \frac{\alpha}{p^s}}.
$$
Then
$$
	\sum_{u\pmod*{p^b}} \psi(x+p^a u)
	=
	\frac{1}{\tau(\overline{\psi})}
	\sum_{\alpha\pmod*{p^s}} \overline{\psi(\alpha)} e^{2\pi ix\frac{\alpha}{p^s}}
	\begin{cases}
	\sum_{u\pmod*{p^b}} e^{2\pi i\alpha\frac{u}{p^{s-a}}} & \text{if } s> a, \\
	p^b & \text{if } s\leq a.
	\end{cases}
$$
Since $a+b\geq s$, if $s>a$,
$$
	\sum_{u\pmod*{p^b}} e^{2\pi i \alpha \frac{u}{p^{s-a}}}
	=
	\sum_{u\pmod*{p^b}} e^{2\pi i \alpha p^{a+b-s} \frac{u}{p^b}}
	=
	\begin{cases}
	p^b & \text{if } p^{s-a}\mid \alpha, \\
	0 & \text{otherwise.}
	\end{cases}
$$
Since $\psi(\alpha)=0$ when $s>0$ and $p\mid \alpha$, we get \eqref{e:avg_char}.
\end{proof}

\begin{lemma}\label{lem:boldSp}
For $t\in\Z$ and $n\in \{\pm1\}$, let $t^2-4n = d\ell^2$
where $d$ is a fundamental discriminant and $\ell\in\Z_{>0}$.
For $e=\ord_p(N)\geq 1$ and $s=\ord_p(\cond(\chi))$,
let $h = \max\{2s-1, e\}$, $g=\ord_p(t^2-4n)$ and $f=\frac{g-\ord_p(d)}{2}=\ord_p(\ell)$.

When $p$ is odd, we have
\begin{multline}\label{e:boldSp_odd}
	\boldS_p(p^e, \chi; t, n)
	\\
	=
	\delta_{g\geq h>0} \chi\!\left(\frac{t+p^s \delta_{2\nmid t}}{2}\right)
	\bigg\{
	p^{e-1} \big(p^{f-\lf\frac{h-1}{2}\rf}+p^{f-\lf\frac{h}{2}\rf}\big)
	+
	\left(1-\left(\frac{d}{p}\right)\right) \frac{p^{e-1}}{p-1}
	\big(p^{f-\lf\frac{h-1}{2}\rf} + p^{f-\lf\frac{h}{2}\rf} - p-1\big)
	\bigg\}
	\\
	+
	\delta_{\substack{g\leq h-1, \\ \left(\frac{d}{p}\right)=1}}
	\left(\chi\!\left(\frac{t+\sqrt{d}\ell}{2}\right)+ \chi\!\left(\frac{t-\sqrt{d}\ell}{2}\right)\right)
	p^{f+\min\{e-s, f\}}.
\end{multline}
When $p=2$ we have $\boldS_2(2^e, \chi; t, n)=0$ if $t$ is odd, while if $t$ is even then
\begin{multline}\label{e:boldSp_even}
	\boldS_2(2^e, \chi; t, n)
	\\
	=
	\chi\!\left(\frac{t}{2}\right)
	\begin{cases}
	\delta_{2\nmid d} 2^{e-1}
	\big(2^{\frac{g}{2}-\lf\frac{h-1}{2}\rf} +2^{\frac{g}{2}-\lf\frac{h}{2}\rf}\big)
	\\
	+
	\left(1-\left(\frac{d}{2}\right)\right) 2^{e-1}
	\big(2^{\lf\frac{g}{2}\rf-\lf\frac{h-1}{2}\rf} + 2^{\lf\frac{g}{2}\rf-\lf\frac{h}{2}\rf}-3\big)
	& \text{if } g\geq h+1 \text{ and } 2\nmid g, \\
	& \text{or } g\geq \max\{h, 2s+2\} \text{ and } 2\mid g, \\
	-\delta_{2\nmid d} 2^{e+1} - \left(1-\left(\frac{d}{2}\right)\right) 2^{e-1}
	& \text{if } g=2s\geq e+1, \\
	-3\cdot 2^{e-1} & \text{if } g=2s=e \text{ and } 2\nmid d, \\
	0 & \text{otherwise}
	\end{cases}
	\\
	+
	\delta_{\substack{g\leq h-1, \\ \left(\frac{d}{2}\right)=1}}
	\left(\chi\!\left(\frac{t+\sqrt{d}\ell}{2}\right) + \chi\!\left(\frac{t-\sqrt{d}\ell}{2}\right)\right)
	2^{f+\min\{e-s, f\}}.
\end{multline}
\end{lemma}
\begin{proof}
By Lemma~\ref{lem:Omega_p}, setting $\alpha=e+2(f-k)+\beta$ for $\beta\in\{0, 1\}$ and $0\leq k \leq f$ in \eqref{e:Omega_p},
we get
\begin{multline*}
	\Omega(p^{e+2(f-k)+\beta}; n, t)/p^{e+f-k}
	\\
	=
	\begin{cases}
	\left\{\frac{t+p^{\lc\frac{e+\beta}{2}\rc+f-k} \omega}{2} +
	p^{\lc\frac{e+\beta}{2}\rc +f-k} u\;:\; u\pmod*{p^{\lc\frac{e+\beta}{2}\rc-\beta}}\right\}
	& \text{if } 2k+\ord_p(d) \geq e+\beta+\delta, \\
	\left\{\frac{t\pm \sqrt{d}\ell}{2}+p^{e+\beta+f-2k} u\;:\; u\pmod*{p^{k-\beta}}\right\}
	& \text{if } 2k \leq e+\beta-1 \text{ and } \left(\frac{d}{p}\right)=1, \\
	\emptyset & \text{otherwise,}
	\end{cases}
\end{multline*}
where $\delta=\max\{\ord_p(d)-1, 0\}$ and
$$
	\omega = \begin{cases}
	1
	& \text{if } p=2, 2\nmid e+\beta, k=\frac{e+\beta+1}{2}-\frac{\ord_2(d)}{2}\leq f
	\text{ and } 8\nmid d, \\
	& \text{or } p=2, 2\mid e+\beta, k=\frac{e+\beta}{2}\leq f
	\text{ and } 2\nmid d, \\
	& \text{or } p\neq 2 \text{ and } t\text{ odd, }\\
	0 & \text{otherwise.}
	\end{cases}
$$

When $p$ is odd, applying Lemma~\ref{lem:avg_char},
\begin{multline*}
	\sum_{\xi\in \Omega(p^{e+2(f-k)+\beta}; n, t)/p^{e+f-k}} \chi(\xi)
	\\
	=
	\begin{cases}
	p^{\lf\frac{e+\beta}{2}\rf-\beta}
	\chi\!\left(\frac{t+\delta_{2\nmid t} p^s}{2}\right)
	& \text{if } \lc\frac{e+\beta-\ord_p(d)}{2}\rc\leq k\leq f+\lc\frac{e+\beta}{2}\rc-\lc\frac{h+\beta}{2}\rc, \\
	p^{k-\beta}
	\left(\chi\!\left(\frac{t+\sqrt{d}\ell}{2}\right) + \chi\!\left(\frac{t-\sqrt{d}\ell}{2}\right)\right)
	& \text{if } k\leq\min\{\lf\frac{e+\beta+f-s}{2}\rf, \lf\frac{e-1+\beta}{2}\rf, f\}
	\text{ and } \left(\frac{d}{p}\right)=1, \\
	0 & \text{otherwise.}
	\end{cases}
\end{multline*}
Recalling \eqref{e:boldSp_0}, we have
\begin{multline*}
	\boldS_p(p^e, \chi; t, n)
	\\
	=
	\delta_{\left(\frac{d}{p}\right)=1}
	\left(\chi\!\left(\frac{t+\sqrt{d}\ell}{2}\right) + \chi\!\left(\frac{t-\sqrt{d}\ell}{2}\right)\right)
	\bigg\{
	1+\sum_{\beta\in \{0, 1\}} p^{-1-\beta}(p-1)
	\sum_{k-1}^{\lf\frac{\min\{e+\beta+f-s, e-1+\beta, 2f\}}{2}\rf}
	p^{2k}
	\bigg\}
	\\
	+
	\chi\!\left(\frac{t+\delta_{2\nmid t} p^s}{2}\right)
	\left(p-\left(\frac{d}{p}\right)\right)
	\sum_{\beta\in \{0, 1\}} p^{\lf\frac{e+\beta}{2}\rf-\beta-1}
	\sum_{k=\lc\frac{e+\beta-\ord_p(d)}{2}\rc}^{f+\lc\frac{e+\beta}{2}\rc-\lc\frac{h+\beta}{2}\rc}
	p^k.
\end{multline*}
Since
\begin{multline}\label{e:sum_d/p=1}
	1+\sum_{\beta\in \{0, 1\}} p^{-1-\beta}(p-1)
	\sum_{k-1}^{\lf\frac{\min\{e+\beta+f-s, e-1+\beta, 2f\}}{2}\rf}
	p^{2k}
	\\
	=
	p^{\min\{e+f-s, e-1, 2f\}}
	=
	\begin{cases}
	p^{e-1} & \text{if } g\geq h, \\
	p^{\min\{e+f-s, 2f\}} & \text{if } g\leq h-1
	\end{cases}
\end{multline}
and
\begin{multline*}
	\sum_{\beta\in \{0, 1\}} p^{\lf\frac{e+\beta}{2}\rf-\beta-1}
	\sum_{k=\lc\frac{e+\beta-\ord_p(d)}{2}\rc}^{f+\lc\frac{e+\beta}{2}\rc-\lc\frac{h+\beta}{2}\rc}
	p^k
	\\
	=
	p^{e-1} \frac{p^{\max\{0, \lf\frac{g}{2}\rf-\lc\frac{h}{2}\rc+1\}}-1}{p-1}
	+
	p^{e-1-\ord_p(d)} \frac{p^{\max\{0, \lc\frac{g}{2}\rc -\lf\frac{h}{2}\rf\}}-1}{p-1}
	\\
	=
	\begin{cases}
	p^{e-1} \frac{p^{\lf\frac{g}{2}\rf-\lc\frac{h}{2}\rc+1}-1}{p-1}
	+
	p^{e-1-\ord_p(d)} \frac{p^{\lc\frac{g}{2}\rc-\lf\frac{h}{2}\rf}-1}{p-1}
	& \text{if } g\geq h, \\
	0 & \text{if } g\leq h-1,
	\end{cases}
\end{multline*}
we get \eqref{e:boldSp_odd}.

Now assume $p=2$.
If $t$ is odd then $d\equiv5$ mod $8$; thus $\bigl(\frac d2\bigr)=-1$,
and $\Omega(p^\alpha;n,t)=\emptyset$ for all $\alpha\geq1$
by Lemma \ref{lem:Omega_p};
therefore all sums in \eqref{e:boldSp_0} are empty, 
and so $\boldS_2(2^e, \chi; t, n)=0$.
From now on we assume $t$ is even.
Applying Lemma~\ref{lem:avg_char}, we get
\begin{multline*}
	\sum_{\xi\in \Omega(2^{e+2(f-k)+\beta}; n, t) /2^{e+f-k}} \chi(\xi)
	\\
	=
	\begin{cases}
	2^{\lf\frac{e+\beta}{2}\rf-\beta} \chi\!\left(\frac{t}{2}\right)
	& \text{if } \lc\frac{e+\beta-\delta_{4\mid d}}{2}\rc \leq k \leq f+\lc\frac{e+\beta}{2}\rc -\lc\frac{h+\beta}{2}\rc, \\
	& \text{and } \left[2\nmid g \text{ or } 2\mid g \text{ and } \frac{g}{2}-1\right], \\
	2^{\lf\frac{e+\beta}{2}\rf-\beta} \chi\!\left(\frac{t}{2}+2^{s-1}\right)
	& \text{if } g=2s, k=\lc\frac{e+\beta}{2}\rc \leq f\text{ and } 2\nmid d, \\
	& \text{or } g=2s, 2\nmid e+\beta-1, k=\frac{e+\beta-1}{2}\leq f\text{ and } 4\mid d, \\
	2^{k-\beta} \left(\chi\!\left(\frac{t+\sqrt{d}\ell}{2}\right) + \chi\!\left(\frac{t-\sqrt{d}\ell}{2}\right)\right)
	& \text{if } 2k \leq \min\{e+\beta+f-s, e+\beta-1, 2f\}
	\text{ and } \left(\frac{d}{2}\right)=1, \\
	0 & \text{otherwise.}
	\end{cases}
\end{multline*}
Note that if $2\nmid d$ then $d\equiv_41$.
So for $s=2$ and $\left(\frac{d}{2}\right)=1$, $\sqrt{d}\pmod*{4}$ is well defined.
Furthermore, for $2\nmid d$, since $4\mid d\ell^2$, so $f=\ord_p(\ell) \geq 1$.
Also note that since $(x+2^{s-1})^2\equiv x^2\pmod*{2^{s}}$ for $s\geq 2$,
we have $\chi(x+2^{s-1}) = \pm \chi(x)$.
Because $\chi$ is primitive, we get
$$
	\chi(x+2^{s-1}) = -\chi(x).
$$

By \eqref{e:boldSp_0},
\begin{multline*}
	\boldS_2(2^e, \chi; t, n)
	\\
	=
	\delta_{\left(\frac{d}{2}\right)=1}
	\left(\chi\!\left(\frac{t+\sqrt{d}\ell}{2}\right) + \chi\!\left(\frac{t-\sqrt{d}\ell}{2}\right)\right)
	\bigg\{ 1+ \sum_{\beta\in \{0, 1\}} \sum_{k=1}^{\lf\frac{\min\{e+\beta+f-s, e+\beta-1, 2f\}}{2}\rf}
	2^{2k-1-\beta}\bigg\}\\
	+
	\delta_{\substack{2\nmid g \\ \text{ or } 2\mid g \text{ and } s\leq \frac{g}{2}-1}}
	\left(2-\left(\frac{d}{2}\right)\right)
	\chi\!\left(\frac{t}{2}\right)
	\sum_{\beta\in\{0, 1\}} 2^{\lf\frac{e+\beta}{2}\rf-\beta}
	\sum_{k=\lc\frac{e+\beta-\delta_{4\mid d}}{2}\rc}^{f+\lc\frac{e+\beta}{2}\rc-\lc\frac{h+\beta}{2}\rc} 2^{k-1}
	\\
	-
	\delta_{g=2s} \chi\!\left(\frac{t}{2}\right)
	\left(2-\left(\frac{d}{2}\right)\right)
	\sum_{\beta\in \{0, 1\}} 2^{\lf\frac{e+\beta}{2}\rf-\beta}
	\begin{cases}
	2^{\lc\frac{e+\beta}{2}\rc-1} & \text{if } \lc\frac{e+\beta}{2}\rc\leq f \text{ and } 2\nmid d, \\
	2^{\frac{e+\beta-1}{2}-1} & \text{if } \frac{e+\beta-1}{2}\leq f \text{ and } 4\mid d, \\
	0 & \text{otherwise.}
	\end{cases}
\end{multline*}
Applying \eqref{e:sum_d/p=1} and
\begin{multline*}
	\sum_{\beta\in \{0, 1\}} 2^{\lf\frac{e+\beta}{2}\rf-\beta-1}
	\sum_{k=\lc\frac{e+\beta-\delta_{4\mid d}}{2}\rc}^{f+\lc\frac{e+\beta}{2}\rc-\lc\frac{h+\beta}{2}\rc} 2^k
	\\
	=
	\begin{cases}
	2^{e-1-\delta_{4\mid d} }
	\big(2^{\lf\frac{g}{2}\rf-\lf\frac{h-1}{2}\rf} + 2^{\lf\frac{g}{2}\rf-\lf\frac{h}{2}\rf} -2-\delta_{4\mid d}\big)
	& \text{if } g\geq h+\delta_{4\mid d}, \\
	0 & \text{otherwise,}
	\end{cases}
\end{multline*}
we get \eqref{e:boldSp_even}.
\end{proof}

\subsection{Cusps}\label{ss:cusps}
One easily verifies that every cusp of $\Gamma^+ = \Gamma_0(N)$ is equivalent to a cusp of the form
\begin{equation}\label{e:cusp_N}
	\tfrac{a}{c}, \; a, c\in \Z, \; c\geq 1, \; c\mid N \text{ and } \gcd(c, a)=1.
\end{equation}
Any two such cusps $\frac{a_1}{c_1}$ and $\frac{a_2}{c_2}$ are $\Gamma^+$-equivalent if and only if
$$
	c_1 = c_2 \text{ and } a_1\equiv_{\gcd(c_1, N/c_1)} a_2
$$
and we then write $\frac{a_1}{c_1}\sim_{\Gamma^+} \frac{a_2}{c_2}$.
We fix, once and for all, a set $C_\Gamma$ containing exactly one representative
$\frac{a}{c}$ satisfying \eqref{e:cusp_N} for each cusp class, i.e.\
$$
	C_\Gamma
	=
	\left\{\tfrac{a}{c}\;:\; a, c\in \Z, \; c\geq 1, \; c\mid N, \; \gcd(a, c)=1\right\}
	/\sim_{\Gamma^+}.
$$
We make our choice of $C_\Gamma$ in such a way that
\begin{equation}\label{e:CGamma_gcd}
\begin{cases}
\text{for every $\frac ac\in C_\Gamma$ with $\gcd(c,N/c)>2$, also $\:\frac{-a}c\in C_\Gamma$;}
\\[5pt]
\text{for every $c\mid N$ with $\gcd(c,N/c)\leq2$, $\:\frac 1c\in C_\Gamma$.}
\end{cases}
\end{equation}
This is possible since $\frac{a}{c}\not\sim_{\Gamma^+} \frac{-a}{c}$
whenever $\gcd(c,N/c)>2$.
Note also that it follows from our choice of $C_\Gamma$ that for every $c\mid N$ with $\gcd(c,N/c)\leq2$,
$\frac1c$ is the unique cusp in $C_\Gamma$ with denominator $c$.

For each $\frac{a}{c}\in C_{\Gamma}$, we fix a choice of scaling matrix
$$
	W_{a/c} = \bpm a & b\\ c& d\ebpm\in \SL_2(\Z).
$$
We make these choices so that
$$
	W_{-a/c} = \bpm -a & b\\ c& -d\ebpm
	\text{ for } \gcd(c, N/c)>2
	\text{ and }
	W_{1/c} = \bpm 1 & 0 \\ c & 1\ebpm
	\text{ for } \gcd(c, N/c)\leq 2
$$
holds for any $\frac{a}{c}\in C_{\Gamma}$ (cf.\ \eqref{e:CGamma_gcd}).
Note that $W_{a/c}(\infty) = \frac{a}{c}$,
so $W_{a/c}^{-1}(\frac{a}{c}) = \infty$.
We also set
\begin{equation}\label{e:Na/c}
	\N_{a/c}
	=
	\bpm \sqrt{\frac{\gcd(c^2, N)}{N}} & 0 \\ 0 & \sqrt{\frac{N}{\gcd(c^2, N)}}\ebpm
	W_{a/c}^{-1}
\end{equation}
and
\begin{equation}\label{e:Ta/c}
	T_{a/c} = \N_{a/c}^{-1} \bpm1 & -1\\ 0 & 1\ebpm \N_{a/c}\in \Gamma^+.
\end{equation}
We then verify that the fixator subgroup of $\frac{a}{c}$ in $\Gamma^+$ is the cyclic subgroup generated by $T_{a/c}$.
That is, for the cusp $\eta_j = \frac{a}{c}$, $\N_{a/c}$ and $T_{a/c}$ correspond to $\N_j$ and $T_j$ in \S\ref{ss:STF_general}.
\begin{lemma}\label{lem:chi_Ta/c}
For any $\frac{a}{c}\in C_{\Gamma}$,
$$
	T_{a/c} =
	\bpm 1+ac\frac{N}{\gcd(c^2, N)} & -a^2 \frac{N}{\gcd(c^2, N)}\\
	N\frac{c^2}{\gcd(c^2, N)} & 1-ac\frac{N}{\gcd(c^2, N)}\ebpm
$$
and
\begin{equation}\label{e:chi_Ta/c}
	\chi(T_{a/c})
	=
	\chi\!\left(1-ac\frac{N}{\gcd(c^2, N)}\right).
\end{equation}
\end{lemma}
\begin{proof}
This lemma follows from the definition of $T_{a/c}$ in \eqref{e:Ta/c}.
\end{proof}

For later use, we prove the following lemma and then compute the set of open cusps.
From now on, for each prime $p\mid N$, set $e_p = \ord_p(N)$ and $s_p = \ord_p(q)$, where $q=\cond(\chi)$.
\begin{lemma}\label{lem:varphi_gcd}
For every $x\in\R$,
\begin{multline}\label{e:varphi_gcd}
	\sum_{c\mid N, q\mid \frac{N}{\gcd(c, N/c)}}
	\varphi(\gcd(c, N/c)) (\gcd(c, N/c))^x
	\\
	=
	\prod_{p\mid N}
	\bigg[2+
	(p-1)p^{x}\frac{p^{\min\{\lf\frac{e_p}{2}\rf, e_p-s_p\}(x+1)}+p^{\min\{\lf\frac{e_p-1}{2}\rf, e_p-s_p\}(x+1)}-2}{p^{x+1}-1}\bigg].
\end{multline}
\end{lemma}
\begin{proof}
By the multiplicativity of the Euler $\varphi$-function,
the left hand side of \eqref{e:varphi_gcd} becomes a product over primes dividing $N$:
$$
	\prod_{p\mid N} \left(\sum_{\substack{j\geq0, \\ e_p-s_p\geq \min\{j, e_p-j\}}}
	\varphi(p^{\min\{j, e_p-j\}}) p^{x\min\{j, e_p-j\}}\right).
$$
For each prime $p\mid N$, the inner sum is
\begin{multline}\label{e:varphi_gcd_p}
	\sum_{\substack{j\geq0, \\ e_p-s_p\geq \min\{j, e_p-j\}}}
	\varphi(p^{\min\{j, e_p-j\}}) p^{x\min\{j, e_p-j\}}
	\\
	=
	\sum_{j=0}^{\min\{\lf\frac{e_p}{2}\rf, e_p-s_p\}} \varphi(p^j) p^{jx}
	+
	\sum_{j=0}^{\min\{\lf\frac{e_p-1}{2}\rf, e_p-s_p\}} \varphi(p^j) p^{jx}
	\\
	=
	2+ (p-1)p^{x} \frac{p^{\min\{\lf\frac{e_p}{2}\rf, e_p-s_p\}(x+1)}-1}{p^{x+1}-1}
	+
	(p-1)p^{x} \frac{p^{\min\{\lf\frac{e_p-1}{2}\rf, e_p-s_p\}(x+1)}-1}{p^{x+1}-1}
	.
\end{multline}
\end{proof}
\begin{lemma}[cf.\ \cite{Iwa97}, Lemma~13.5]\label{lem:open_cusp}
The set of open cusps is 
\begin{equation}\label{e:open_cusp_N}
	C_{\Gamma, \chi}
	=
	\left\{\tfrac{a}{c}\in C_{\Gamma}\;:\; q\mid \frac{N}{\gcd(c, N/c)}\right\}.
\end{equation}
Its cardinality is
$$
	|C_{\Gamma, \chi}| = \sum_{c\mid N, q\mid \frac{N}{\gcd(c, N/c)}}
	\varphi(\gcd(c, N/c))
	=
	\Psi_1(N, q).
$$
Here $q=\cond(\chi)$ and $\Psi_1(N, q)$ is given in \eqref{e:Psi1}.
\end{lemma}
\begin{proof}
Using \eqref{e:chi_Ta/c} and $\cond(\chi)=q$ we have
$$
	\chi(T_{a/c}) = \chi\!\left(1-a\tfrac{N}{\gcd(c, N/c)}\right)=1
	\iff
	q\mid \tfrac{N}{\gcd(c, N/c)}.
$$
By Lemma~\ref{lem:varphi_gcd}, taking $x=0$, we get
$$
	|C_{\Gamma, \chi}|
	=
	\sum_{c\mid N, q\mid\frac{N}{\gcd(c, N/c)}}\varphi(\gcd(c, N/c))
	=
	\prod_{p\mid N}\bigg[p^{\min\{\lf\frac{e_p}{s}\rf, e_p-s_p\}}+p^{\min\{\lf\frac{e_p-1}{2}\rf, e_p-s_p\}}\bigg]
	=
	\Psi_1(N, q).
$$
\end{proof}

Adapting the notation from \S\ref{ss:STF_general} to our present explicit setting,
for each $\frac{a}{c}\in C_{\Gamma}$ we write $k(a/c)$ for the representative in $C_\Gamma$ for the cusp $V\frac{a}{c} = \frac{-a}{c}$.
Then by \eqref{e:CGamma_gcd},
\begin{equation}\label{e:ka/c}
	k(a/c) = \begin{cases}
	\frac{-a}{c} & \text{if } \gcd(c, N/c) >2, \\
	\frac{1}{c} & \text{if } \gcd(c, N/c)\leq 2.
	\end{cases}
\end{equation}
Furthermore, as in \eqref{e:Uj} we have fixed a choice of an element $U_{a/c}\in \Gamma^+$ satisfying
\begin{equation}\label{e:Ua/c}
	U_{a/c} = V \N_{a/c}^{-1} \bpm -1 & u \\ 0 & 1\ebpm \N_{k(a/c)}
\end{equation}
for some $u\in \R$ (which may depend on $\frac ac$).
\begin{lemma}\label{lem:Ua/c}
For $\frac{a}{c}\in C_{\Gamma}$, when $\gcd(c, N/c)>2$ we have
$$
	U_{a/c} =
	V\N_{a/c}^{-1} \bpm -1 & u \\ 0 & 1\ebpm \N_{-a/c}
	=
	\bpm -1 + au\frac{N}{\gcd(c, N/c)} & a^2 u \frac{N/c}{\gcd(c, N/c)}\\
	-Nu \frac{c}{\gcd(c, N/c)} & -1-au \frac{N}{\gcd(c, N/c)}\ebpm,
$$
for some $u\in \Z$.
When $\gcd(c, N/c)\leq 2$,
\begin{equation}\label{e:U1/c}
	U_{1/c} =
	V \N_{1/c}^{-1}
	\bpm -1 & u_1\frac{\gcd(c, N/c)}{N/c} \\ 0 & 1\ebpm
	\N_{1/c}
\end{equation}
for some $u_1\in \Z$ satisfying $cu_1\equiv_{N/c} -2$.
Furthermore, for every $\frac{a}{c}\in C_{\Gamma, \chi}$,
\begin{equation}\label{e:chi_U1/c_1}
	\chi(U_{a/c})= \chi_{\alpha(N, c)}(-1),
\end{equation}
where
\begin{equation}\label{e:alpha}
	\alpha(N, c)
	=
	\begin{cases}
	N & \text{if } \gcd(c, N/c)>2, \\
	\frac{N}{c} & \text{if } \gcd(c, N/c)=1, \\
	\frac{N}{2c} & \text{if } \gcd(c, N/c) = 2 \text{ and } \frac{N}{2c} \text{ is odd, }\\
	\frac{2N}{c} & \text{if } \gcd(c, N/c) = 2 \text{ and } \frac{N}{2c} \text{ is even.}
	\end{cases}
\end{equation}
\end{lemma}
\begin{proof}
Given $\frac{a}{c}\in C_{\Gamma}$, recalling \eqref{e:Ua/c},
$$
	U_{a/c} = V \N_{a/c}^{-1} \bpm -1 & u \\ 0 & 1\ebpm \N_{k(a/c)} \in \Gamma^+,
$$
for some $u\in \R$, which is uniquely determined modulo one.
If $\gcd(c, N/c)>2$, then $k(a/c) = -a/c$ and we get
$$
	U_{a/c} = \bpm -1 + acu \frac{N}{\gcd(c^2, N)} & a^2u \frac{N}{\gcd(c^2, N)}\\
	-Nu\frac{c^2}{\gcd(c^2, N)} & -1-acu \frac{N}{\gcd(c^2, N)}\ebpm.
$$
One easily checks that $U_{a/c}\in \Gamma^+=\Gamma_0(N)$ if and only if $u\in \Z$.
If $\frac{a}{c}\in C_{\Gamma, \chi}$, then $q\mid \frac{N}{\gcd(c, N/c)}$.
Because
$$
	-1-acu\frac{N}{\gcd(c^2, N)} = -1-au\frac{N}{\gcd(c, N/c)} \equiv_q -1
$$
and $\chi$ is even, we get $\chi(U_{a/c}) = \chi(-1) = 1$.

If $\gcd(c, N/c) \leq 2$, then
$$
	U_{1/c} = \bpm 1+ cu\frac{N}{\gcd(c^2, N)} & -u \frac{N}{\gcd(c^2, N)}\\
	-2c-c^2u \frac{N}{\gcd(c^2, N)} & 1+ cu\frac{N}{\gcd(c^2, N)}\ebpm.
$$
Let $u_1 = u\frac{N}{\gcd(c^2, N)}$.
Then $U_{1/c}\in \Gamma_0(N)$ if and only if $u_1\in \Z$ and satisfying $cu_1\equiv_{N/c}-2$.
Then $1+cu_1 \equiv_{N/c}-1$.
Moreover,
$$
	\chi(U_{1/c}) = \chi(1+cu_1).
$$
If $\gcd(c, N/c)=1$, then $\chi = \chi_c \chi_{N/c}$, so
$$
	\chi(U_{1/c}) = \chi_{N/c}(1+cu_1) = \chi_{N/c}(-1).
$$
Assume that $\gcd(c,N/c)=2$ and $\frac{1}{c}\in C_{\Gamma, \chi}$,
so that $q\mid \frac{N}{2}$.
If $\frac{N}{2c}$ is odd then $\gcd(2c, \frac{N}{2c})=1$ and $\chi = \chi_{2c} \chi_{\frac{N}{2c}}$.
Since $q = \cond(\chi_{2c}) \cond(\chi_{\frac{N}{2c}})$ divides $\frac{N}{2}$
and $\cond(\chi_{\frac{N}{2c}})\mid \frac{N}{2c}$,
we have $\cond(\chi_{2c})\mid c$.
Thus
$$
	\chi(U_{1/c}) = \chi_{2c}(1+cu_1) \chi_{\frac{N}{2c}}(1+cu_1)
	=
	\chi_{\frac{N}{2c}} (-1).
$$
If $\frac{N}{2c}$ is even, then $\frac{c}{2}$ is odd, so $\gcd\!\left(\frac{2N}{c}, \frac{c}{2}\right)=1$
and $\chi = \chi_{\frac{2N}{c}} \chi_{\frac{c}{2}}$.
Since $q=\cond(\chi_{\frac{2N}{c}}) \cond(\chi_{\frac{c}{2}})$ divides $\frac{N}{2}$, $\cond(\chi_{\frac{2N}{c}})$ must divide $\frac{N}{c}$,
so we get
$$
	\chi(U_{1/c}) =\chi_{\frac{c}{2}}(1+cu_1) \chi_{\frac{2N}{c}} (1+cu_1) = \chi_{\frac{2N}{c}}(-1).
$$
\end{proof}

Note that in our present setting, \eqref{e:Tjv} says that for each $c\mid N$ with $\gcd(c, N/c)\leq 2$,
and each $v\in \Z$, we set
\begin{equation}\label{e:T1/cv}
	T_{1/c, v} = U_{1/c}^{-1} V T_{1/c}^v.
\end{equation}
Also, for any such $c$ and $v$, we have a number $\mbc_{1/c,v}>0$ defined by \eqref{e:mbc}.
The following two lemmas evaluate $\chi^{(\ve)}(T_{1/c, v})$ and $\mbc_{1/c,v}$.
Recall that $\chi^{(\ve)}(V) = (-1)^{\ve}$.
\begin{lemma}\label{lem:chi_T1/cv}
For $c\mid N$ and $\gcd(c, N/c)\leq 2$, we have
\begin{equation}\label{e:chi_T1/cv}
	\chi^{(\ve)}(T_{1/c, v}) =
	\begin{cases}
	(-1)^{\ve} \chi_{\alpha(N, c)}(-1) & \text{if } \frac{1}{c}\in C_{\Gamma, \chi}, \\
	(-1)^{\ve+v+w_c} & \text{if } \frac{1}{c}\notin C_{\Gamma, \chi},
	\end{cases}
\end{equation}
where $w_c\in\{0, 1\}$ is a constant determined by $(-1)^{w_c} = \chi(U_{1/c})$.
Here $\alpha(N, c)$ as in \eqref{e:alpha}.
\end{lemma}
\begin{proof}
It follows from \eqref{e:T1/cv} that 
$\chi^{(\ve)}(T_{1/c, v}) = (-1)^{\ve} \overline{\chi(U_{1/c})} \chi(T_{1/c})^v$.

If $\frac{1}{c}\in C_{\Gamma, \chi}$ then $\chi(T_{1/c})=1$, and using
\eqref{e:chi_U1/c_1} we have
$$
	\chi^{(\ve)}(T_{1/c, v}) = (-1)^{\ve} \chi_{\alpha(N, c)}(-1),
$$
for all $v\in \Z$.

Now assume $\frac{1}{c}\notin C_{\Gamma, \chi}$. Then $\gcd(c, N/c)=2$ and $q\nmid \frac{N}{2}$,
and by Lemma~\ref{lem:chi_Ta/c} we have
$$
	\chi(T_{1/c}) = \chi\!\left(1-\tfrac{N}{2}\right)=-1,
$$
since $\left(1-\frac{N}{2}\right)^2 \equiv_N1$ and $q\nmid \frac{N}{2}$.
By Lemma~\ref{lem:Ua/c}, $\chi(U_{1/c}) = \chi(1+cu_1)$ for some $u_1\in \Z$ satisfying $cu_1\equiv_{N/c}-2$.
Then
$$
	(1+cu_1)^2 = 1+ 2cu_1\left(1+\tfrac{c}{2} u_1\right) \equiv_N1,
$$
so $\chi(U_{1/c})^2 = 1$ and $\chi(U_{1/c})\in \{\pm1\}$.
Choosing $w_c\in\{0, 1\}$ so that $\chi(U_{1/c}) = (-1)^{w_c}$,
we get
$$
\chi^{(\ve)}(T_{1/c, v}) = (-1)^{\ve+v+w_c}.
$$
\end{proof}

\begin{lemma}\label{lem:mbc}
Assume that $c\mid N$ and $\gcd(c, N/c)\leq 2$.
If $4\mid N$ then $\mbc_{1/c, v}=\sqrt N$ for all $v$.
On the other hand, if $4\nmid N$,
then there is some $s\in\{0,1\}$
(which depends on $c$ and $U_{1/c}$)
such that for all $v$,
\begin{equation}\label{e:mbc_1/c_v}
	\mbc_{1/c, v}
	=
	\begin{cases}
	\sqrt{N} & \text{if } v\equiv_2 s,\\
	\sqrt{2N} & \text{if } v\not\equiv_2 s\text{ and }2\mid N\\
	2\sqrt{N} & \text{if } v\not\equiv_2 s\text{ and }2\nmid N.
	\end{cases}
\end{equation}
\end{lemma}

\begin{proof}
For given $c$ and $v$,
following the definition of $\mbc_{1/c,v}$ in \eqref{e:mbc},
we recall that $T_{1/c,v}$ is a reflection fixing the point $\frac1c$,
and that also the other fixpoint of $T_{1/c,v}$ in $\partial\HH$ must be a $\Gamma^+$-cusp,
which we call $\eta$.
We choose $V_2\in\Gamma^+$ and $\frac{a'}{c'}\in C_\Gamma$ so that $\eta=V_2(\frac{a'}{c'})$.
Then $\frac{a'}{c'}$ is a fixpoint of $V_2^{-1}T_{\frac 1c,v}V_2\in\Gamma^+ V$,
and this forces $k(\frac{a'}{c'})=\frac{a'}{c'}$,
i.e.\ $\gcd(c',N/c')\leq2$ and $a'=1$.

Set $M=\gcd(c,N/c)\leq2$.
From \eqref{e:U1/c} and \eqref{e:T1/cv}
we get
\begin{align}\label{lem:mbcPF1}
	\N_{1/c} T_{1/c, v} \N_{1/c}^{-1}=  \bpm -1 & u_1\frac{M}{N/c}+v \\ 0 & 1\ebpm,
\end{align}
where $u_1$ is an integer satisfying $cu_1\equiv_{N/c} -2$;
and the two fixpoints in $\partial\HH$ of the reflection in \eqref{lem:mbcPF1} are
$\infty$ and $\frac12(u_1\frac M{N/c}+v)$.
Hence
$$
\eta
=
	\N_{1/c}^{-1} \biggl(\frac12\Bigl(u_1\frac{M}{N/c} +v\Bigr)\biggr)
	=
	\bpm 1 & 0\\ c& 1\ebpm \frac{u_2}{2} = \frac{u_2}{cu_2+2},
\qquad\text{with }\:u_2:=u_1+v\frac{N/c}{M}.
$$
But we know from above that the $C_\Gamma$-representative for the cusp $\eta$ is $\frac1{c'}$,
and one easily verifies that the $C_\Gamma$-representative of an arbitrary $\Gamma^+$-cusp 
$\frac{\alpha}{\gamma}$ with $\alpha, \gamma\in \Z$ has
denominator $\gcd\bigl(N, \frac{\gamma}{\gcd(\alpha, \gamma)}\bigr)$.
Hence, letting $u_3:=\gcd(u_2,cu_2+2)=\gcd(u_2,2)\in\{1,2\}$, we have
$$
c'=\gcd\Bigl(N,\frac{cu_2+2}{u_3}\Bigr)=\frac1{u_3}\gcd\bigl(Nu_3,cu_2+2\bigr)=\frac N{cu_3}\gcd\Bigl(cu_3,\frac{cu_2+2}{N/c}\Bigr).
$$
In the last equality we used the fact that $\frac Nc\mid cu_2+2$, since $cu_1\equiv_{N/c} -2$.
To make the last expression more explicit, note that 
$\gcd(cu_3,\frac{cu_2+2}{N/c})\in\{1,2\}$, since $\gcd(cu_3,cu_2+2)=\gcd(cu_3,2)\in\{1,2\}$.
Therefore,
\begin{equation}\label{e:c'}
c'=\begin{cases}
N/c&\text{if }\: u_2\equiv\frac{cu_2+2}{N/c}\pmod* 2,
\\
N/(2c)&\text{if }\: 2\mid u_2\text{ and }2\nmid\frac{cu_2+2}{N/c},
\\
2N/c&\text{if }\: 2\nmid u_2\text{ and }2\mid\frac{cu_2+2}{N/c}.
\end{cases}
\end{equation}

Let us write $W_{1/c} V_2 W_{1/c'}=\sm \alpha & \beta\\ \gamma & \delta \esm\in\SL_2(\Z)$ with $\gamma\geq0$.
Since $\frac{\alpha}{\gamma}=W_{1/c} V_2 W_{1/c'} (\infty) = \frac{u_2}{2}$,
we get $\gamma = \frac{2}{u_3}$.
Using this in
the definition of $\mbc_{1/c,v}$, \eqref{e:mbc},
we get
\begin{equation}\label{e:mbc_1/c_v_0A}
\mbc_{1/c,v}=\frac2{u_3}\sqrt{\frac{N/c}{\gcd(c,N/c)}}\sqrt{\frac{N/c'}{\gcd(c',N/c')}}.
\end{equation}
Let us set $f=\ord_2(c)$ and $g=\ord_2(N/c)$;
then $\gcd(c,N/c)\leq2$ implies $\min(f,g)\leq1$.
Using \eqref{e:c'}, the formula \eqref{e:mbc_1/c_v_0A} can be re-expressed as
$\mbc_{1/c,v}=\sqrt N\cdot 2^{\rho/2}$, where
\begin{equation}\label{e:mbc_1/c_v_0B}
\rho=\begin{cases}
-2\min(f,g)& \text{if } 2\mid u_2 \text{ and } 2\mid\frac{cu_2+2}{N/c}, 
\\
1-\min(f,g)-\min(f+1,g-1)& \text{if } 2\mid u_2 \text{ and } 2\nmid\frac{cu_2+2}{N/c},
\\
1-\min(f,g)-\min(f-1,g+1)& \text{if } 2\nmid u_2 \text{ and } 2\mid \frac{cu_2+2}{N/c}, 
\\
2-2\min(f,g)	& \text{if } 2\nmid u_2 \text{ and } 2\nmid \frac{cu_2+2}{N/c}.
\end{cases}
\end{equation}

Let us first assume $\min(f,g)=1$, i.e.\
both $c$ and $\frac Nc$ are even.
In this case, if either $4\mid c$ or $2\mid u_2$, then $cu_2+2\equiv_42$,
which implies $2\nmid \frac{cu_2+2}{N/c}$ and $4\nmid\frac Nc$.
Hence, the first case in \eqref{e:mbc_1/c_v_0B} cannot occur,
and if $f\geq2$ ($\Rightarrow g=1$) then also the third case in \eqref{e:mbc_1/c_v_0B} is excluded,
while if $g\geq2$ ($\Rightarrow f=1$) then the second case in \eqref{e:mbc_1/c_v_0B} is excluded.
By inspection, it then follows that $\rho=0$.
Next assume $\min(f,g)=0$.
Note that if $f\geq1$ ($\Rightarrow g=0$) then $2\mid\frac{cu_2+2}{N/c}$,
and on the other hand if if $g\geq1$ ($\Rightarrow f=0$) then the condition $cu_1\equiv_{N/c}-2$ forces $u_1$ to be even,
and then also $u_2=u_1+v\frac{N/c}M$ is even.
If $f\geq2$ or $g\geq2$ then these observations imply $\rho=0$, by
\eqref{e:mbc_1/c_v_0B}.
We have thus proved the lemma in the case $4\mid N$,
and it only remains to consider the three cases with $f,g\leq1$, not both $1$.
If $f=1$ ($\Rightarrow g=0$) then \eqref{e:mbc_1/c_v_0B} implies $\rho=\delta_{2\nmid u_2}$,
and so, recalling $u_2=u_1+v\frac{N/c}M$, it follows that \eqref{e:mbc_1/c_v} holds with $s\equiv_2 u_1$.
If $g=1$ ($\Rightarrow f=0$) then \eqref{e:mbc_1/c_v_0B} implies $\rho=\delta_{2\nmid \frac{cu_2+2}{N/c}}
=\delta_{4\mid u_2}$,
and so \eqref{e:mbc_1/c_v} holds with $s\equiv_2 1+\frac{u_1}2$.
Finally if $f=g=0$, i.e.\ $2\nmid c$ and $2\nmid\frac Nc$,
then $u_2\equiv_2\frac{cu_2+2}{N/c}$,
and thus \eqref{e:mbc_1/c_v_0B} implies $\rho=2\delta_{2\nmid u_2}$;
therefore \eqref{e:mbc_1/c_v} again holds
with $s\equiv_2u_1$.
\end{proof}

\subsection{Cuspidal contributions $\Cu(\Gamma, \chi)$}
Recall that we write $\Cu(N,\chi;n)$ for the cuspidal contribution
in the trace formula in Theorem \ref{thm:STF_N};
thus $\Cu(N,\chi;n)=\Cu(\Gamma, \chi^{(0)})+n\Cu(\Gamma, \chi^{(1)})$
where $C(\Gamma,\chi^{(\ve)})$ is the cuspidal contribution (cf.\ \eqref{e:Cu})
in the trace formula for $(\Gamma_0^\pm (N), \chi^{(\ve)})$.
Our aim in this section is to prove the following proposition,
giving an explicit formula for $\Cu(N,\chi;n)$.
Let
\begin{equation}\label{e:veN}
	\ve_N
	=
	\begin{cases}
	1 & \text{if } N \text{ is odd}, \\
	\frac{1}{2} & \text{if } 2\parallel N, \\
	0 & \text{otherwise,}
	\end{cases}
\end{equation}
and for each prime $p\mid N$, define
\begin{equation}\label{e:Psi4}
	\Psi_4(p^{e_p}, p^{s_p})
	=
	p^{e_p-s_p}
	\max\{2s_p-e_p-1, 0\}.
\end{equation}

\begin{proposition}\label{prop:Cu}
We have
\begin{multline}\label{e:Cu_N_1}
	\Cu(N, \chi; 1)
	=
	\Psi_1(N, q)
	\bigg\{ \frac{1}{4} h(0) - \frac{1}{2\pi} \int_\R h(r) \frac{\Gamma'}{\Gamma}(1+ir)\; dr
	\bigg\}
	\\
	-
	\bigg\{\Psi_1(N, q) \sum_{p\mid N} \frac{\Psi_4(p^{e_p}, p^{s_p})}{\Psi_1(p^{e_p}, p^{s_p})} \log p
	+ \Psi_1(N, q) \log 2\bigg\}
	g(0)
\end{multline}
and
\begin{multline}\label{e:Cu_N_-1}
	\Cu(N, \chi; -1)
	=
	I_\chi 2^{\omega(N)} \Omega_1(N, q)
	\bigg\{
	\frac{1}{4} h(0)
	+ \frac{1}{2\pi}
	\int_\R h(r) \bigg(\frac{\Gamma'}{\Gamma}\!\left(\frac{1}{2}+ir\right) - \frac{\Gamma'}{\Gamma}(1+ir)\bigg)
	\; dr
	\bigg\}
	\\
	+
	\bigg\{I_\chi 2^{\omega(N)-1} \big(\Omega_1(N, q) \log N + \ve_N \log 2\big)
	\bigg\}
	g(0).
\end{multline}
Here $\Psi_1$, $\Psi_4$, $\Omega_1$ and $\ve_N$ are given in \eqref{e:Psi1}, \eqref{e:Psi4}, \eqref{e:Omega1} and \eqref{e:veN} respectively.
\end{proposition}

The proposition will be proved by evaluating the various sums appearing in \eqref{e:Cu}
in our explicit setting.
By \eqref{e:open_cusp_N}, we can write
\begin{align}
	&
	\sum_{j\in C_{\Gamma, \chi}, k(j)=j} \chi(T_{j, 0})
	=
	\sum_{c\mid N, \gcd(c, N/c)\leq 2, q\mid \frac{N}{\gcd(c, N/c)}}
	\chi^{(\ve)}(T_{1/c, 0});
	\\
	&
	\sum_{1\leq j \leq\kappa, k(j)=j} \sum_{v\in \{0, 1\}}
	\chi(T_{j, v}) \log\mbc_{j, v}
	=
	\sum_{c\mid N, \gcd(c, N/c)\leq 2}
	\sum_{v\in\{0, 1\}}
	\chi^{(\ve)}(T_{1/c, v}) \log \mbc_{1/c, v}; \\
	&
	\sum_{\substack{1\leq j\leq\kappa, \\ j\notin C_{\Gamma, \chi}}} \log\left|1-\chi(T_j)\right|
	=
	\sum_{\substack{\frac{a}{c}\in C_{\Gamma}, \\ q\nmid \frac{N}{\gcd(c, N/c)}}}
	\log \left|1-\chi(T_{a/c})\right|.
\end{align}
We get formulas for each of these sums in the following lemmas.

\begin{lemma}\label{lem:sum_chiT1/c0}
$$
	\sum_{c\mid N, \gcd(c, N/c)\leq 2, q\mid \frac{N}{\gcd(c, N/c)}}
	\chi^{(\ve)}(T_{1/c, 0})
	=
	(-1)^{\ve} I_\chi 2^{\omega(N)} \Omega_1(N, q),
$$
where $\Omega_1(N, q)$ is given in \eqref{e:Omega1}.
\end{lemma}
\begin{proof}
By Lemma \ref{lem:chi_T1/cv} and the definition of $\alpha(N,c)$ in \eqref{e:alpha}, we have
\begin{multline*}
	\sum_{c\mid N, \gcd(c, N/c)\leq 2, q\mid \frac{N}{\gcd(c, N/c)}}
	\chi^{(\ve)}(T_{1/c, 0})
	=
	(-1)^{\ve}
	\sum_{c\mid N, \gcd(c, N/c)\leq 2, q\mid \frac{N}{\gcd(c, N/c)}}
	\chi_{\alpha(N, c)} (-1)
	\\
	=
	(-1)^{\ve} 
\biggl\{
	\sum_{c\mid N, \gcd(c, N/c)=1} \chi_{N/c}(-1)
	+
	\delta_{4\mid N, q\mid \frac{N}{2}}
	\sum_{c\mid N, \gcd(c, N/c)=2}
	\big\{
	\delta_{2\nmid\frac{N}{2c}}\,
	\chi_{\frac{N}{2c}}(-1)
	+
	\delta_{2\mid\frac{N}{2c}}\,
	\chi_{\frac{2N}{c}}(-1)
	\big\}\biggr\}.
\end{multline*}
For $\gcd(c, N/c)=1$, we have
\begin{equation}\label{e:avr_chi_gcd1}
	\sum_{c\mid N, \gcd(c, N/c)=1}
	\chi_{N/c}(-1)
	=
	\prod_{p\mid N} (1+\chi_p(-1))
	=
	I_\chi 2^{\omega(N)}.
\end{equation}
When $4\mid N$, $q\mid\frac{N}{2}$ and $\gcd(c,N/c)=2$, then either
$\frac{N}{2c}$ is odd so $2^{e_2-1}\parallel c$ or $\frac{N}{2c}$ is
even, $8\mid N$ and $2\parallel c$.
For $\frac{N}{2c}$ odd, we have
$$
	\sum_{c\mid N, \gcd(c, N/c)=2, 2\nmid \frac{N}{2c}}
	\chi_{\frac{N}{2c}}(-1)
	=
	\prod_{p\mid N, p>2} (1+\chi_p(-1))
	=
	I_\chi 2^{\omega(N)-1}.
$$
For $8\mid N$, $q\mid \frac{N}{2}$ and $\frac{N}{2c}$ even, we have
$$
	\sum_{c\mid N, \gcd(c, N/c)=2, 2\mid \frac{N}{2c}}
	\chi_{\frac{2N}{c}}(-1)
	=
	\chi_2(-1) \prod_{p\mid N, p>2} (1+\chi_p(-1))
	=
	I_\chi 2^{\omega(N)-1}.
$$
Combining these, we obtain the formula stated in the lemma.
\end{proof}

\begin{lemma}\label{lem:sum_T1/cv_mbc}
$$
	\sum_{c\mid N, \gcd(c, N/c)\leq 2}
	\sum_{v\in\{0, 1\}}
	\chi^{(\ve)}(T_{1/c, v}) \log \mbc_{1/c, v}
	=
	(-1)^{\ve} I_\chi 2^{\omega(N)}
	\big(
	\Omega_1(N, q)\log N + \ve_N \log 2
	\big).
$$
Here $\Omega_1(N, q)$ and $\ve_N$ are given in \eqref{e:Omega1} and \eqref{e:veN} respectively.
\end{lemma}
\begin{proof}
Recalling \eqref{e:chi_T1/cv} and \eqref{e:mbc_1/c_v},
when $4\mid N$, we get
\begin{multline*}
	\sum_{c\mid N, \gcd(c, N/c)\leq 2}
	\sum_{v\in\{0, 1\}}
	\chi^{(\ve)}(T_{1/c, v}) \log \mbc_{1/c, v}
	\\
	=
	(-1)^{\ve}
	\log N
	\left\{
	\sum_{c\mid N, \gcd(c, N/c)=1}
	\chi_{\alpha(N, c)}(-1)
	+
	\frac{1}{2}
	\sum_{c\mid N, \gcd(c, N/c)=2}
	\begin{cases}
	2\chi_{\alpha(N, c)}(-1)
	& \text{if } q\mid \frac{N}{2}, \\
	\sum_{v\in \{0, 1\}} (-1)^{v+w_c}
	& \text{if } q\nmid \frac{N}{2}
	\end{cases}
	\right\}
	\\
	=
	(-1)^{\ve}
	\log N
	\sum_{c\mid N, \gcd(c, N/c)\leq 2, q\mid \frac{N}{\gcd(c, N/c)}}
	\chi_{\alpha(N, c)}(-1)
=(-1)^{\ve} I_\chi 2^{\omega(N)} \Omega_1(N, q) \log N,
\end{multline*}
where the last equality holds by
Lemma \ref{lem:sum_chiT1/c0} (and its proof).

For $4\nmid N$, by \eqref{e:chi_T1/cv} and \eqref{e:mbc_1/c_v}, we have
\begin{multline*}
	\sum_{c\mid N, \gcd(c, N/c)\leq 2}
	\sum_{v\in\{0, 1\}}
	\chi^{(\ve)}(T_{1/c, v}) \log \mbc_{1/c, v}
	\\
	=
	(-1)^{\ve}
	\sum_{c\mid N, \gcd(c, N/c)=1}
	\chi_{\alpha(N, c)}(-1)
	\bigg( \log N + \frac{1}{\gcd(2, N)} \log 2\bigg)
\\
=(-1)^{\ve}I_\chi 2^{\omega(N)} \Omega_1(N, q)\bigg( \log N + \frac{1}{\gcd(2, N)} \log 2\bigg),
\end{multline*}
where we again used Lemma \ref{lem:sum_chiT1/c0}.
The desired formula follows if we also note that  $4\nmid N$ implies $\Omega_1(N, q)=1$.
\end{proof}

\begin{lemma}\label{lem:sum_notinCGamachi}
$$
	\sum_{\substack{\frac{a}{c}\in C_{\Gamma}, \\ q\nmid \frac{N}{\gcd(c, N/c)}}}
	\log \left|1-\chi(T_{a/c})\right|
	=
	\Psi_1(N, q)
	\sum_{p\mid N} \frac{\Psi_4(p^{e_p}, p^{s_p})}{\Psi_1(p^{e_p}, p^{s_p})}
	\log p.
$$
Here $\Psi_1$ and $\Psi_4$ are given in \eqref{e:Psi1} and \eqref{e:Psi4}.
\end{lemma}
\begin{proof}
Recalling \eqref{e:chi_Ta/c},
and writing $M=\gcd(c,N/c)$, we have
\begin{align}\label{lem:sum_notinCGamachipf2}
	\sum_{\substack{\frac{a}{c}\in C_{\Gamma}, \\ q\nmid \frac{N}{M}}} \log \left|1-\chi(T_{a/c})\right|
	=
	\sum_{c\mid N, q\nmid \frac{N}{M}}
	\sum_{a\in (\Z/M\Z)^\times}
	\log \left|1-\chi\!\left(1-a\tfrac{N}{M}\right)\right|.
\end{align}
Now let $c$ be fixed subject to the conditions $c\mid N$ and $q\nmid\frac NM$.
Note that $N\mid (N/M)^2$, and therefore the map
$a\mapsto 1-aN/M$ is 
a homomorphism from the additive group $\Z/M\Z$ 
to the multiplicative group $(\Z/N\Z)^\times$.
In particular there is some $A\in\Z$ (uniquely determined modulo $M$)
such that $\chi(1-aN/M)=e(aA/M)$ for all $a\in\Z$,
and we have $M\nmid A$ since $q\nmid\frac NM$.
It follows that, with $B=M/\gcd(A,M)>1$,
\begin{align*}
	\sum_{a\in (\Z/M\Z)^\times}
	\log \left|1-\chi\!\left(1-a\tfrac{N}{M}\right)\right|
&=\frac{\varphi(M)}{\varphi(B)}\,\log\biggl(\,\prod_{y\in(\Z/B\Z)^\times}\left|1-e\left(\tfrac yB\right)\right|\biggr)
\hspace{100pt}&
\\
&=\frac{\varphi(M)}{\varphi(B)}\begin{cases}
\log p&\text{ if $B=p^r$ for some prime $p$ and $r\geq1$,}
\\
0&\text{ otherwise.}
\end{cases}
\end{align*}
(cf., e.g., \cite[\S46(7)]{Nag51}).    
Note that for $a\in\Z$ we have $\chi(1-aN/M)=e(aA/M)=1$ if and only if $B\mid a$;
hence in fact 
\begin{align}\label{lem:sum_notinCGamachipf1}
B=\frac{q}{\gcd(N/M,q)}.
\end{align}

For any prime $p\mid N$, let $f_p=\ord_p(c).$
It follows from \eqref{lem:sum_notinCGamachipf1} that
$B=p^r$ holds for some prime $p$ and some $r\geq1$ if and only if
$\min\{f_p,e_p-f_p\}>e_p-s_p$ while 
$\min\{f_{p'},e_{p'}-f_{p'}\}\leq e_{p'}-s_{p'}$ for every prime $p'\neq p$.
Furthermore, when this holds,
we have
$$
	\frac{\varphi(M)}{\varphi(B)}
	=
	p^{e_p-s_p}
	\prod_{p'\mid N, p'\neq p} \varphi\!\left({p'}^{\min\{e_{p'}-f_{p'}, f_{p'}\}}\right).
$$
Hence the expression in \eqref{lem:sum_notinCGamachipf2} equals
\begin{align*}
	\sum_{p\mid N}
	\bigg[
	\prod_{p'\mid N, p'\neq p}
	\bigg( \sum_{\substack{ 0\leq f_{p'}\leq e_{p'}, \\
	\min\{f_{p'},e_{p'}-f_{p'}\}\leq  e_{p'}-s_{p'} }}
 \varphi\!\left({p'}^{\min\{e_{p'}-f_{p'}, f_{p'}\}}\right)
	\bigg)
	\bigg]
\,
p^{e_p-s_p}\log p
\sum_{\substack{0\leq f_p\leq e_p\\
\min\{f_p,e_p-f_p\}>e_p-s_p}}1.
\end{align*}
Here the sum over $f_p$ equals $\max\{2s_p-e_p-1,0\}$,
and each sum over $f_{p'}$ can be evaluated
using \eqref{e:varphi_gcd_p} (with $x=0$).
This leads to the statement of the lemma.
\end{proof}

\begin{proof}[Proof of Proposition \ref{prop:Cu}]
The proposition follows from
$$
\Cu(N,\chi;n)=\Cu(\Gamma, \chi^{(0)})+n\Cu(\Gamma, \chi^{(1)})
$$
by evaluating the various sums appearing in \eqref{e:Cu}
using the above three lemmas
and also the formula $|C_{\Gamma, \chi}|=\Psi_1(N, q)$ (cf.\ Lemma~\ref{lem:open_cusp}).
\end{proof}

\subsection{Eisenstein series}\label{EISsec}
Huxley \cite{Hux84} gave explicit expressions (involving Dirichlet $L$-functions)
for the scattering matrix and its determinant for the congruence subgroups
$\Gamma^0(N)$, $\Gamma^1(N)$ and $\Gamma(N)$ of the modular group,
for arbitrary level $N$.
Note that $\Gamma^0(N)$ is conjugate to the group $\Gamma_0(N)$, which we consider here.
The case of squarefree $N$ was previously considered in
\cite[Ch.~11]{Hej83}.
The calculations in this section could likely be shortened by appealing to
work of Young \cite{You19} that was written subsequently to this paper;
we have included our original treatment for completeness.

In our setting, the Eisenstein series for $\langle\Gamma^+, \chi\rangle$ associated to an open cusp $\frac{a}{c}\in C_{\Gamma, \chi}$ is given by (cf.~\eqref{e:Eis_j})
\begin{equation}\label{e:Eis_a/c}
	E_{a/c}(z, s, \chi) = \sum_{U\in [T_{a/c}]\bsl \Gamma^+} \overline{\chi(U)} (\Im(\N_{a/c} Uz))^s,
	\text{ for } z\in \HH \text{ and } \Re(s)>1.
\end{equation}
Substituting $U=W_{a/c} \sm \alpha & \beta\\ \gamma & \delta\esm$,
using \eqref{e:Na/c} and \eqref{e:Ta/c},
and writing $z=x+iy$ and $W_{a/c} = \sm a & b\\ c& d\esm\in \SL_2(\Z)$,
one sees that \eqref{e:Eis_a/c} can be rewritten as
\begin{equation}\label{e:Eis_a/c_2}
	E_{a/c}(z, s, \chi)
	=
	\frac{\gcd(c^2, N)^{s}}{2N^s}
	\sum_{\langle\gamma, \delta\rangle}
	\chi(a\alpha+b\gamma) \frac{y^{s}}{|\gamma z+\delta|^{2s}},
\end{equation}
where now the sum runs over all $\langle\gamma, \delta\rangle\in \Z^2$
satisfying $\gcd(\gamma, \delta)=1$, $\gcd(\gamma, N)=c$
and $\delta\frac{\gamma}{c}\equiv_{\gcd(c, N/c)} -a$.
Also in \eqref{e:Eis_a/c_2},
$\alpha=\alpha(\gamma, \delta)\in \Z$ denotes any integer satisfying
both $\alpha\delta\equiv_\gamma 1$ and $\alpha\equiv_{N/c} -d\frac{\gamma}{c}$.
Such $\alpha$ exists for every relevant pair $\langle\gamma, \delta\rangle$.

We wish to make the factor $\chi(a\alpha+b\gamma)$ somewhat more explicit.
For a prime $p$, from now on we are using the following notations:
$$
	e_p = \ord_p(N), \quad s_p = \ord_p(q)\quad \text{ and }\quad f_p = \ord_p(c).
$$
Since $\frac{a}{c}\in C_{\Gamma, \chi}$, by Lemma~\ref{lem:open_cusp}, we have $s_p\leq \max\{e_p-f_p, f_p\}$.
Also $\gcd(\gamma, N)=c$ implies that $\ord_p(\gamma)\geq f_p$
with equality unless $f_p=e_p$.
Now the conditions on $\alpha=\alpha(\gamma, \delta)$ imply
$$
	a\alpha+b\gamma
	\equiv
	\begin{cases}
	a\delta^{-1} & \mod{p^{f_p}}, \\
	a(-d\frac{\gamma}{c}) + b\gamma \equiv -\frac{\gamma}{c} & \mod{p^{e_p-f_p}}.
	\end{cases}
$$
Hence
\begin{equation}\label{e:chip_ab}
	\chi_p(a\alpha+b\gamma)
	=
	\begin{cases}
	\chi_p(a) \overline{\chi_p(\delta)} & \text{if } e_p \leq 2f_p, \\
	\chi_p(-\gamma/c) & \text{if } e_p> 2f_p.
	\end{cases}
\end{equation}
Note that both relations are valid in the special case $e_p=2f_p$.

As in \cite{Hux84} we now introduce a family of sums similar to but simpler than \eqref{e:Eis_a/c_2}.
It will turn out that these sums can be expressed as linear combinations of the Eisenstein series $E_{a/c}(z, s, \chi)$ (cf.~Lemma~\ref{lem:E_chi1chi2} below),
and a key step in computing the scattering matrix will be to invert these linear relations.

Throughout this and the next section, we let $\chi_1$ and $\chi_2$ denote primitive Dirichlet characters satisfying $\chi_1(-1) = \chi_2(-1)$.
For $j\in \{1, 2\}$, we write $q_j = \cond(\chi_j)$ and $\chi_{j, p} = (\chi_j)_p$ for any prime $p$.
As in \cite[p.~143]{Hux84}, for $\Re(s)>1$ and $z=x+iy$, we define
$$
	B_{\chi_1}^{\chi_2}(z, s) = \sum_{(c, d)\in \Z^2-\{(0, 0)\}} \chi_1(c) \chi_2(d) \frac{y^s}{|cz+d|^{2s}}
$$
and
$$
	E_{\chi_1}^{\chi_2}(z, s) = \sum_{(c, d)\in \Z^2, \gcd(c, d)=1} \chi_1(c) \chi_2(d) \frac{y^{s}}{|cz+d|^{2s}}
	=
	L(2s, \chi_1\chi_2)^{-1} B_{\chi_1}^{\chi_2}(z, s).
$$
As in \cite{Hux84}, we have
\begin{equation}\label{e:B_chi1chi2_modular}
	B_{\chi_1}^{\chi_2}\!\left(\sm A& B\\ C& D\esm z, s\right)
	=
	\chi_1(D) \chi_2(A) B_{\chi_1}^{\chi_2}(z, s),
\end{equation}
for all $\sm A& B\\ C& D\esm\in \Gamma_0(1)$ with $q_1\mid C$ and $q_2\mid B$.
Similarly,
\begin{equation}\label{e:Bchi1chi2_V_fe}
	B_{\chi_1}^{\chi_2}(Vz, s) = \chi_1(-1) B_{\chi_1}^{\chi_2}(z, s).
\end{equation}
The same transformation formulas also hold for $E_{\chi_1}^{\chi_2}(z, s)$.

We let
\begin{equation}\label{e:F}
	F = \left\{(m, \chi_1, \chi_2)\;:\; m\in \Z_{\geq 1},\; mq_1 \mid N, \; q_2\mid m, \; \cond(\chi\chi_2\overline{\chi_1}) = 1\right\}.
\end{equation}
Note that $\cond(\chi\chi_2\overline{\chi_1})=1$ if and only if $\chi\chi_2(x) = \chi_1(x)$ for all $x\in (\Z/N\Z)^\times$;
however this does not imply $\chi\chi_2=\chi_1$ since the product $\chi\chi_2$ is not necessarily primitive.
It follows from \eqref{e:B_chi1chi2_modular} that for every $(m, \chi_1, \chi_2)\in F$, the function $E_{\chi_1}^{\chi_2}(mz, s)$ is $\langle\Gamma^+, \chi\rangle$-invariant,
that is $E_{\chi_1}^{\chi_2}(mTz, s) = \chi(T) E_{\chi_1}^{\chi_2}(mz, s)$ for all $T\in \Gamma^+$.
In fact, each such function $E_{\chi_1}^{\chi_2}(mz, s)$ may be expressed as a linear combination of the Eisenstein series for $\langle\Gamma^+, \chi\rangle$ as follows.
\begin{lemma}\label{lem:E_chi1chi2}
For any $(m, \chi_1, \chi_2)\in F$, we have
$$
	E_{\chi_1}^{\chi_2}(mz, s)
	=
	\frac{2N^s}{m^s}
	\sum_{\frac{a}{c}\in C_{\Gamma, \chi}} \frac{\gcd(m, c)^{2s}}{\gcd(c^2, N)^s}
	\,\,\chi_1\Bigl(-\frac{c}{\gcd(m, c)}\Bigr)\, \chi_2\Bigl(a\frac{m}{\gcd(m, c)}\Bigr)
	\,E_{a/c}(z, s, \chi).
$$
\end{lemma}
\begin{remark}
With some work, one can see that
Lemma \ref{lem:E_chi1chi2} is equivalent to
\cite[Thm.\ 7.1]{You19}.
However, our proof
is significantly different;
indeed note that 
the proof in \cite{You19}
is based on
inversion of a formula expressing $E_{a/c}(z,s,\chi)$ as a linear combination of functions of the form
$E_{\chi_1}^{\chi_2}(mz, s)$
\cite[Thm.\ 6.1]{You19}.
\end{remark}
\begin{proof}
For $(m, \chi_1, \chi_2)\in F$, by \cite[p.~146(top)]{Hux84},
$E_{\chi_1}^{\chi_2}(mz, s)$ can be expressed as
\begin{equation}\label{e:Echi1chi2_Hux}
	\sum_{m=gh}\,\,
	\sum_{\substack{e\pmod*{N/h}, \\ \gcd(e, g)=1}}
	\sum_{\substack{f\pmod*{N}, \\ \gcd(f, he, N)=1}}
	\chi_2(g) \chi_1(e) \chi_2(f) \Bigl(\frac{h}{g}\Bigr)^{\!s}
	\sum_{\gamma\in he+N\Z}
	\sum_{\substack{\delta\in f+N\Z, \\ \gcd(\gamma, \delta)=1}}
	\frac{y^s}{|\gamma z+\delta|^{2s}}.
\end{equation}
Given any $g, h, e, f$ as in this sum,
let $a$ and $c$ be the uniquely determined integers such that
$$
	c = \gcd(he, N), \quad
	a\equiv_{\gcd(c, N/c)} -f\frac{he}{c}
	\:\:\text{ and }\:\: \tfrac{a}{c}\in C_{\Gamma}.
$$
One easily checks that this transformation gives a bijection between the set of tuples
$$(g, h, e, f, \gamma, \delta)$$
appearing in the sum,
and the set
$$
	\left\{(a, c, \gamma, \delta)\;:\; \tfrac{a}{c}\in C_{\Gamma},
	\gamma, \delta\in \Z, \gcd(\gamma, \delta)=1, \gcd(\gamma, N)=c,
	\delta\equiv_{\gcd(c, N/c)} -a\left(\frac{\gamma}{c}\right)^{-1}
	\right\}.
$$
The inverse map is given by
$$
	h=\gcd(m, c), \quad
	g=\frac{m}{h}, \quad
	e\equiv_{N/h} \frac{\gamma}{h}
	\text{ and }
	f\equiv_N\delta.
$$

Now let $(a, c, \gamma, \delta)$ be any tuple in the above set,
with corresponding $(g, h, e, f, \gamma, \delta)$.
We then have, for each prime $p\mid N$,
\begin{equation}\label{e:chi2chi1chi2_p}
	\chi_{2, p}(g) \chi_{1, p}(e) \chi_{2, p}(f)
	=
	\chi_{2, p}(g) \chi_{2, p}(a) \chi_{1, p}\left(-\frac{c}{h}\right)
	\begin{cases}
	\chi_p(a) \overline{\chi_p(\delta)} & \text{if } \ord_p(N/c) \leq \ord_p(c), \\
	\chi_p(-\gamma/c) & \text{if } \ord_p(N/c) > \ord_p(c).
	\end{cases}
\end{equation}
To prove this, first note that if $\ord_p(q_2) > \ord_p(c)$ then
$g = \frac{m}{\gcd(c, m)}\equiv_p 0$,
and thus $\chi_{2, p}(g)=0$ so \eqref{e:chi2chi1chi2_p} holds.
Hence from now on we may assume $\ord_p(q_2) \leq \ord_p(c)$.

Let us first assume that $\ord_p(c) > \ord_p(m)$.
Then $p\mid e$ but $p\nmid a$, $p\nmid \delta$ (since $\gcd(a, c)=1$ and $\gcd(\gamma, \delta)=1$).
If $p\mid q_1$ then both sides of \eqref{e:chi2chi1chi2_p} are $0$.
On the other hand, if $p\nmid q_1$, viz. $\chi_{1, p}\equiv 1$ then $\cond(\chi_p\chi_{2, p})=1$
and so $\chi_{2, p}(f) = \overline{\chi_p(f)} = \overline{\chi_p(\delta)}$
and $\chi_{2, p}(a) = \overline{\chi_p(a)}$;
and if $\ord_p(N/c) > \ord_p(c)$ then also $\delta \equiv_{\gcd(c, p^\infty)} -a\left(\frac{\gamma}{c}\right)^{-1}$;
hence \eqref{e:chi2chi1chi2_p} holds.

It remains to consider the case $\ord_p(c) \leq \ord_p(m)$.
Then since $mq_1\mid N$ we have $\ord_p(q_1) \leq \ord_p(N/c)$
and hence $\chi_{1, p}(e) = \chi_{1, p}(-\gamma/c) \chi_{1, p}(-c/\gcd(m, c))$.
Also, using
$$
f\equiv \delta\equiv -a\left(\frac{\gamma}{c}\right)^{-1} \pmod*{\gcd(c, N/c)}
$$
we get
$$
	\begin{cases}
	\chi_{1, p}(-\gamma/c) = \chi_{1, p}(a) \overline{\chi_{1, p}(\delta)} & \text{if } \ord_p(N/c) \leq \ord_p(c); \\
	\chi_{2, p}(f) = \chi_{2, p}(a) \overline{\chi_{2, p}(-\gamma/c)} & \text{if } \ord_p(N/c) > \ord_p(c).
	\end{cases}
$$
The second identity holds since we are assuming $\ord_p(q_2) \leq \ord_p(c)$ from start.
From this \eqref{e:chi2chi1chi2_p} follows and so we have proved that \eqref{e:chi2chi1chi2_p} holds in all cases.

Next let us note that
\begin{equation}\label{e:chi2chi1=0}
	\chi_2(g)\chi_1(e)=0 \text{ whenever } \tfrac{a}{c}\notin C_{\Gamma, \chi}.
\end{equation}
To prove this, assume that $\frac{a}{c}\notin C_{\Gamma, \chi}$;
then by Lemma~\ref{lem:open_cusp}, there is a prime $p\mid N$ such that
$$\ord_p(q) > \ord_p\!\left(\tfrac{N}{\gcd(c, N/c)}\right)=\max\{\ord_p(N/c), \ord_p(c)\}.$$
As noted above, if $\ord_p(q_2) > \ord_p(c)$ then $\chi_{2, p}(g) = 0$ and \eqref{e:chi2chi1=0} holds.
On the other hand, if $\ord_p(q_2) \leq \ord_p(c)$ then using $\cond(\chi\chi_2\overline{\chi_1})=1$ we get $\ord_p(q_1) = \ord_p(q)$
and since $mq_1\mid N$ we have
$$\ord_p(m) < \ord_p(\gcd(c, N/c))\leq \ord_p(c), $$
hence $p\mid e$ so $\chi_{1, p}(e)=0$ and \eqref{e:chi2chi1=0} holds.

This lemma is now a direct consequence of \eqref{e:Echi1chi2_Hux},
the bijection between $\{(g, h, e, f, \gamma, \delta)\}$ and $\{(a, c, \gamma, \delta)\}$
and \eqref{e:chi2chi1chi2_p}, \eqref{e:chi2chi1=0}, \eqref{e:Eis_a/c_2} and \eqref{e:chip_ab}.
\end{proof}

Next, recalling the definition of $F$ in \eqref{e:F}, for any $m\mid N$ we now set
\begin{equation}\label{e:Fm}
	F_m = \{ (\chi_1, \chi_2) \;:\; (m, \chi_1, \chi_2)\in F\}.
\end{equation}
The following lemma gives an explicit enumeration of the set $F_m$.
For any $c\in \Z_{\geq 1}$, we write
$$
	X_c = \left\{\psi \text{ primitive Dirichlet character }:\; \cond(\psi)\mid c\right\}.
$$
This set is in one-to-one correspondence with the dual of the group $(\Z/c\Z)^\times$.
In particular $|X_c| = \varphi(c)$.
Recall that $q=\cond(\chi)$.
Let us write
$$
	G_{N, q} = \left\{m\mid N\;:\; q\mid\tfrac{N}{\gcd(m, N/m)}\right\}.
$$
\begin{lemma}\label{lem:map_Fm_X}
Given $m\mid N$, the set $F_m$ is nonempty if and only if $m\in G_{N, q}$.
When this holds, the map
\begin{equation}\label{e:map_Fm_X}
	(\chi_1, \chi_2) \mapsto \prod_{p\mid N}
	\begin{cases}
	\chi_{1, p} & \text{if } \ord_p(q) \leq \ord_p(m), \\
	\chi_{2, p} & \text{if } \ord_p(q) > \ord_p(m)
	\end{cases}
\end{equation}
is a bijection from $F_m$ onto $X_{\gcd(m, N/m)}$.
In particular $|F_m| = \varphi(\gcd(m, N/m))$.
\end{lemma}
\begin{proof}
By the definition of $F$, two primitive Dirichlet characters $\chi_1$ and $\chi_2$ satisfy $(\chi_1, \chi_2)\in F_m$ if and only if,
for every prime $p\mid N$,
\begin{equation}\label{e:Fm_p}
	s_2 \leq \ord_p(m), \quad s_1 \leq \ord_p(N/m) \quad \text{ and } \quad
	\cond(\chi_p \overline{\chi_{1, p}} \chi_{2, p})=1,
\end{equation}
Here $s_j = \ord_p(q_j)$ for $j\in \{1, 2\}$.
Let $s = \ord_p(q)$.
Note that $\cond(\chi_p \overline{\chi_{1, p}} \chi_{2, p})=1$ implies that $s\leq \max\{s_1, s_2\}$.
Hence by \eqref{e:Fm_p},
$$
	s \leq \max\{\ord_p(m), \ord_p(N/m)\} = \ord_p\!\left(\tfrac{N}{\gcd(m, N/m)}\right).
$$
Since this must hold for every $p\mid N$, 
it follows that $F_m$ can be nonempty only if $m\in G_{N,q}$.

From now on we assume $m\in G_{N, q}$.
We will prove that the map in \eqref{e:map_Fm_X} is a bijection from $F_m$ onto $X_{\gcd(m, N/m)}$.
In particular this will imply that $F_m$ is non-empty,
with $|F_m| = \varphi(\gcd(m, N/m))$.

First consider any prime $p\mid N$ satisfying $s\leq\ord_p(m)$ (with $s=\ord_p(q)$).
For such $p$, the first and third relations in \eqref{e:Fm_p}
imply $s_1\leq\ord_p(m)$,
which together with the second relation in \eqref{e:Fm_p}
implies $s_1 \leq \ord_p(\gcd(m, N/m))$.
Conversely if $\chi_{1,p}$ is \emph{any} primitive character with conductor $p^{s_1}$ subject to
$s_1\leq\ord_p(\gcd(m,N/m))$,
and if $\chi_{2,p}$ is the unique primitive character satisfying $\cond(\chi_p\overline\chi_{1,p}\chi_{2,p})=1$,
then all the conditions in \eqref{e:Fm_p} are fulfilled.

Next consider any $p\mid N$ for which $s>\ord_p(m)$.
Recall that we are assuming $m\in G_{N,q}$;
hence $s\leq\max\{\ord_p(m),\ord_p(N/m)\}$
and so $s\leq\ord_p(N/m)$.
Then the two last relations in \eqref{e:Fm_p} together imply
$s_2\leq\ord_p(N/{m})$,
and in combination with the first relation this gives
$s_2\leq\ord_p (\gcd({m},N/{m}))$.
Conversely if $\chi_{2,p}$ is \emph{any} primitive character with conductor $p^{s_2}$
subject to $s_2\leq\ord_p(\gcd(m,N/m))$,
and if $\chi_{1,p}$ is the unique primitive character satisfying $\cond(\chi_p\overline\chi_{1,p}\chi_{2,p})=1$,
then all the conditions in \eqref{e:Fm_p} are fulfilled.

The above observations imply that the 
map in \eqref{e:map_Fm_X} is indeed a bijection as stated.
\end{proof}

We now turn to the full group $\Gamma = \Gamma_0^\pm (N)$.
Let $\frac{a}{c}\in C_{\Gamma, \chi}$ be given.
By Lemma~\ref{lem:Ua/c} and \cite[(2.8)]{BS07}, we have
\begin{equation}\label{e:Eis_a/c_V}
	E_{a/c}(Vz, s, \chi)
	=
	\chi_{\alpha(N, c)}(-1) E_{k(a/c)}(z, s, \chi),
\end{equation}
where $\alpha(N, c)$ is as in \eqref{e:alpha} and $k(a/c)$ as in \eqref{e:ka/c}.
Hence \eqref{e:Eis_j_Gamma} gives
\begin{equation}\label{e:Eis_a/c_Gamma}
	E_{a/c}^{\Gamma}(z, s, \chi^{(\ve)})
	=
	\begin{cases}
	E_{a/c}(z, s, \chi) + (-1)^{\ve} E_{-a/c}(z, s, \chi) & \text{if } \gcd(c, N/c)>2, \\
	\big(\chi_{\alpha(N, c)}(-1)(-1)^{\ve} + 1\big) E_{1/c}(z, s, \chi)
	& \text{if } \gcd(c, N/c)\leq 2.
	\end{cases}
\end{equation}
It is also natural to introduce, for any primitive Dirichlet characters $\chi_1$ and $\chi_2$ satisfying $\chi_1(-1) = \chi_2(-1)$ as before,
$$
	\eB_{\chi_1}^{\chi_2}(z, s)
	=
	B_{\chi_1}^{\chi_2}(z, s) + \chi^{(\ve)}(V^{-1}) B_{\chi_1}^{\chi_2}(Vz, s)
	=
	\begin{cases}
	2B_{\chi_1}^{\chi_2}(z, s) & \text{if } \chi_1(-1) = (-1)^{\ve}, \\
	0 & \text{otherwise}
	\end{cases}
$$
(cf.~\eqref{e:Bchi1chi2_V_fe}).
This function is $\langle\Gamma, \chi\rangle$-invariant,
since one easily checks that it is $\langle\Gamma^+,\chi\rangle$-invariant and also transforms correctly under $V$.
We also set
\begin{equation}\label{e:eE_chi1chi2}
	\eE_{\chi_1}^{\chi_2}(z, s)
	=
	L(2s, \chi_1\chi_2)^{-1}\, \eB_{\chi_1}^{\chi_2}(z, s)
	=
	\begin{cases}
	2E_{\chi_1}^{\chi_2}(z, s) & \text{if } \chi_1(-1) = (-1)^{\ve}, \\
	0 & \text{otherwise.}
	\end{cases}
\end{equation}

Recall from \eqref{e:R_Gammachi} that we have fixed a subset $R_{\Gamma, \chi}\subset C_{\Gamma, \chi}$,
which in our present notation
(cf.\ also \eqref{e:ka/c} and Lemma~\ref{lem:Ua/c})
has the property that for every $\frac{a}{c}\in C_{\Gamma, \chi}$ with $\gcd(c, N/c)\leq 2$, $\frac{a}{c}$ lies in $R_{\Gamma, \chi}$ if and only if $\chi_{\alpha(N, c)}(-1) = (-1)^{\ve}$.
Set
\begin{multline*}
	F^{\ve}
	= \{(m, \chi_1, \chi_2)\in F\;:\; \chi_1(-1)=(-1)^{\ve}\}
	\\
	= \{(m, \chi_1, \chi_2)\;:\; m\in \Z_{\geq 1}, \; q_2\mid m,\; mq_1\mid N, \; \cond(\chi\chi_2\overline{\chi_1})=1,\; \chi_1(-1) = (-1)^{\ve}\}.
\end{multline*}
We now have the following analogue of Lemma~\ref{lem:E_chi1chi2}.
\begin{lemma}\label{lem:eE_chi1chi2}
For any $(m, \chi_1, \chi_2)\in F^{\ve}$,
$$
	\eE_{\chi_1}^{\chi_2}(mz, s)
	=
	\frac{2N^s}{m^s}
	\sum_{\frac{a}{c}\in R_{\Gamma, \chi}} \frac{\gcd(m, c)^{2s}}{\gcd(c^2, N)^{s}}
	\,\chi_1\Bigl(-\frac{c}{\gcd(m, c)}\Bigr)
	\,\chi_2\Bigl(\frac{am}{\gcd(m, c)}\Bigr)
	I_c
	\,E_{a/c}^\Gamma(z, s, \chi^{(\ve)}),
$$
where $I_c = 1+\delta_{\gcd(c, N/c) >2}$.
\end{lemma}
\begin{proof}
Fix $(m, \chi_1, \chi_2)\in F^{\ve}$.
Then $\eE_{\chi_1}^{\chi_2}(mz, s) = 2E_{\chi_1}^{\chi_2}(mz, s)\not\equiv0$
and this is expressed in Lemma~\ref{lem:E_chi1chi2} as a linear combination of
$E_{a/c}(z, s, \chi)$ over all $\frac{a}{c}\in C_{\Gamma, \chi}$.
For $\frac{a}{c}\in C_{\Gamma, \chi}$ with $\gcd(c, N/c)>2$,
we see immediately that the combined contribution from $\frac{a}{c}\in C_{\Gamma, \chi}$ and $\frac{-a}{c}\in C_{\Gamma, \chi}$
to the sum in Lemma~\ref{lem:E_chi1chi2} (doubled) is:
$$
	\frac{4N^s}{m^s}
	\frac{\gcd(m, c)^{2s}}{\gcd(c^2, N)^s}
	\chi_1\!\left(-\frac{c}{\gcd(m, c)}\right) \chi_2\!\left(a\frac{m}{\gcd(m, c)}\right)
	E_{a/c}^\Gamma(z, s, \chi^{(\ve)}),
$$
since $\chi_2(-1) = \chi_1(-1) = (-1)^{\ve}$.

Next take $\frac{a}{c}\in C_{\Gamma, \chi}$ with $\gcd(c, N/c)\leq 2$ (thus $a=1$).
If $\chi_{\alpha(N, c)}(-1) = -(-1)^{\ve}$, then as we will prove below
\begin{equation}\label{e:chi1chi2=0}
	\chi_1\!\left(-\frac{c}{\gcd(m, c)}\right) \chi_2\!\left(\frac{m}{\gcd(m, c)}\right)=0,
\end{equation}
so that the contribution from $E_{1/c}(z, s, \chi)$ in Lemma~\ref{lem:E_chi1chi2} vanishes.
If $\chi_{\alpha(N, c)}(-1) = (-1)^{\ve}$,
then the contribution to the doubled sum is
$$
	\frac{4N^s}{m^s} \frac{\gcd(c, m)^{2s}}{\gcd(c^2, N)^{s}}
	\chi_1\!\left(-\frac{c}{\gcd(m, c)}\right) \chi_2\!\left(\frac{m}{\gcd(m, c)}\right)
	\frac{1}{2} E_{1/c}^\Gamma(z, s, \chi^{(\ve)}),
$$
since $E_{1/c}^{\Gamma}(z, s, \chi^{(\ve)}) = 2E_{1/c}(z, s, \chi)$.
Hence, adding up the contributions, we obtain \eqref{e:eE_chi1chi2}.

It remains to prove the claim \eqref{e:chi1chi2=0} for every $\frac{1}{c}\in C_{\Gamma, \chi}$ satisfying $\gcd(c,N/c)\leq 2$
and $\chi_{\alpha(N, c)}(-1) = -(-1)^{\ve}$.
To do so, let us fix $s>1$.
Using \eqref{e:Bchi1chi2_V_fe} and $\chi_1(-1)=(-1)^{\ve}$, we see that
$E_{\chi_1}^{\chi_2}(mVz, s) \equiv (-1)^{\ve} E_{\chi_1}^{\chi_2}(mz, s)$
and applying Lemma~\ref{lem:E_chi1chi2} and \eqref{e:Eis_a/c_V},
this leads to a linear relation between the functions $z\mapsto E_{a/c}(z, s, \chi)$ for $\frac{a}{c}$ running through $C_{\Gamma, \chi}$.
However these functions are linearly independent,
since $E_{a/c}(z, s,\chi)$ grows like $(\Im(\N_{a/c}(z)))^s$ when $z$ approaches the cusp $\frac{a}{c}$
but is bounded by $O((\Im(\N_{a'/c'}(z)))^{1-s})$ as $z$ approaches any other cusp $\frac{a'}{c'}\in C_{\Gamma}$;
cf.\ \cite[p.~280 (Prop.~3.7)]{Hej83}.
It follows that all coefficients in the linear relation which we have obtained mush vanish.
In particular for $\frac{1}{c}\in C_{\Gamma, \chi}$ with $\gcd(c, N/c)\leq 2$
this implies
$$
	(\chi_{\alpha(N, c)}(-1)-(-1)^{\ve}) \chi_1\!\left(-\frac{c}{\gcd(m,
	c)}\right) \chi_2\!\left(\frac{m}{\gcd(m, c)}\right)
	=0,
$$
and \eqref{e:chi1chi2=0} follows.
\end{proof}
\begin{remark}
The claim \eqref{e:chi1chi2=0} for every $\frac{1}{c}\in C_{\Gamma, \chi}$ satisfying $\gcd(c, N/c)\leq 2$ and $\chi_{\alpha(N, c)}(-1) = -(-1)^{\ve}$
may alternatively be proved directly from the definition of $\alpha(N, c)$,
using no facts about Eisenstein series.
However we have not found any way of doing this without going through a somewhat lengthy case-by-case analysis.
\end{remark}

Mimicking now \cite[p.~147]{Hux84},
we let $\vecB(s)$ be the vector of the functions
$$
	\vecB(s) = (\eB_{\chi_1}^{\chi_2}(mz, s))_{(m, \chi_1, \chi_2)\in F^{\ve}}.
$$
So far our discussion has been for $\Re(s)>1$.
However the functions $B_{\chi_1}^{\chi_2}(z, s)$ can be meromorphically continued to all $s\in \C$;
hence so can $E_{\chi_1}^{\chi_2}(z, s)$, $\eB_{\chi_1}^{\chi_2}(z, s)$ and $\eE_{\chi_1}^{\chi_2}(z, s)$;
and we have the functional equation (cf.~\cite[p.~145]{Hux84})
$$
	\vecB(s) = \vecD(s) P\, \vecB(1-s),
$$
where $P$ is a permutation which in each row 
$(m, \chi_1, \chi_2)\in F^{\ve}$
has an entry one in column $(q_1m/q_2, \overline{\chi_2}, \overline{\chi_1}))\in F^{\ve}$
and zeros elsewhere,
and  $\vecD(s)$ is a diagonal matrix with entries
$$
	\frac{q_1^{1-s}}{q_2^{s}} \pi^{2s-1} \frac{\tau(\chi_2)}{\tau(\overline{\chi_1})}
	\frac{\Gamma(1-s)}{\Gamma(s)}.
$$
Using the fact that the map $(m, \chi_1, \chi_2)\mapsto (q_1m/q_2, \overline{\chi_2}, \overline{\chi_1})$,
which is the permutation given by $P$,
is an involution of the set $F^{\ve}$, we obtain
\begin{equation}\label{e:detP}
	\det(P) = (-1)^{\frac{|F^{\ve}|-|F_0^{\ve}|}{2}},
\end{equation}
where
\begin{equation}\label{e:Fve0}
	F_0^{\ve} = \{(m, \chi_1, \chi_2)\in F^{\ve}\;:\; \chi_1=\overline{\chi_2}\}
\end{equation}
is the set of fixpoints of the involution.
Making use of the same permutation we also see that
\begin{equation}\label{e:detD}
	\det(\vecD(s))
	=
	\left(\frac{\pi^{2s-1} \Gamma(1-s)}{\Gamma(s)}\right)^{|F^{\ve}|}
	\prod_{(m, \chi_1, \chi_2)\in F^{\ve}} q_1^{1-2s}.
\end{equation}

Let $\vecE^{\Gamma}(s)$ be the vector of the functions $E_{a/c}^{\Gamma}(z, s, \chi^{(\ve)})$
for all $\frac{a}{c}\in R_{\Gamma, \chi}$.
By \eqref{e:eE_chi1chi2} and Lemma~\ref{lem:eE_chi1chi2}, we have
\begin{equation}\label{e:vecB_vecEGamma}
	\vecB(s) = L(s) M(s) \vecE^{\Gamma}(s),
\end{equation}
where $L(s)$ is the $|F^{\ve}|\times |F^{\ve}|$ diagonal matrix with entries $L(2s, \chi_1\chi_2)$
and $M(s)$ is an $|F^{\ve}|\times |R_{\Gamma, \chi}|$ matrix, whose entry in row $(m, \chi_1, \chi_2)\in F^{\ve}$, column $\frac{a}{c}\in R_{\Gamma, \chi}$ is
\begin{equation}\label{e:Ms}
	\frac{2N^s}{m^s} \frac{\gcd(m, c)^{2s}}{\gcd(c^2, N)^{s}}
	\,\chi_1\!\left(-\frac{c}{\gcd(m, c)}\right) \chi_2\!\left(a\frac{m}{\gcd(m, c)}\right) I_c.
\end{equation}
At first \eqref{e:vecB_vecEGamma} holds for $\Re(s)>1$;
however $\vecE^{\Gamma}(s)$ can be meromorphically continued to all
$s\in \C$ (cf.~\eqref{e:Eis_a/c_Gamma} and \cite[Ch.~6.11 and Ch.~8.3]{Hej83})
and \eqref{e:vecB_vecEGamma} remains valid for $s\in \C$.

Let $Q(s)$ be the $|F^{\ve}|\times |F^{\ve}|$ matrix whose entry in row $(m, \chi_1, \chi_2)\in F^{\ve}$ and column $(m', \chi_1', \chi_2')\in F^{\ve}$ is
\begin{equation}\label{e:Qs}
	\begin{cases}
	\mu\!\left(\frac{m}{m'}\right) \chi_2\!\left(\frac{m}{m'}\right) \left(\frac{m}{m'}\right)^{-s}
	& \text{if } m'\mid m, \chi_1' = \chi_1 \text{ and }\chi_2' = \chi_2, \\
	0 & \text{otherwise.}
	\end{cases}
\end{equation}
Then the product matrix $Q(s) M(s)$ has the following entry in row $(m, \chi_1, \chi_2)\in F^{\ve}$ and column $\frac{a}{c}\in R_{\Gamma, \chi}$:
\begin{multline}\label{e:QsMs}
	I_c \sum_{m'\mid m} \mu\!\left(\frac{m}{m'}\right) \chi_2\!\left(\frac{m}{m'}\right) \left(\frac{m}{m'}\right)^{-s}
	\frac{2N^s}{{m'}^s} \frac{\gcd(c, m')^{2s}}{\gcd(c^2, N)^s}
	\,\chi_1\!\left(-\frac{c}{\gcd(c, m')}\right) \chi_2\!\left(a\frac{m'}{\gcd(c, m')}\right)
	\\
	=
	\frac{2N^s}{m^s \gcd(c^2, N)^{s}} I_c
	\sum_{d\mid m} \mu(d) \gcd(m/d, c)^{2s}
	\chi_1 \!\left(-\frac{c}{\gcd(m/d, c)}\right) \chi_2\!\left(a\frac{m/d}{\gcd(m/d, c)}\right)
	\\
	=
	\frac{2N^s}{m^s} \frac{\gcd(c, m)^{2s}}{\gcd(c^2, N)^{s}}
	\, I_c\, \chi_1\!\left(-\frac{c}{\gcd(c, m)}\right) \chi_2\!\left(a\frac{m}{\gcd(c, m)}\right)
	\\
	\times
	\prod_{p\mid m}
	\left(1-\begin{cases}
	1 & \text{if } \ord_p(c) < \ord_p(m), \\
	\chi_1(p) \chi_2(p) p^{-2s} & \text{if } \ord_p(c) \geq \ord_p(m)
	\end{cases}
	\right)
	\\
	=
	\begin{cases}
	0 & \text{if } m\nmid c, \\
	\frac{2N^s m^s}{\gcd(c^2, N)^s}
	I_c \chi_1\!\left(-\frac{c}{m}\right) \chi_2(a) \prod_{p\mid m} (1-\chi_1(p)\chi_2(p) p^{-2s})
	& \text{if } m\mid c.
	\end{cases}
\end{multline}

The following lemma is immediate from Lemma~\ref{lem:open_cusp} and the definition of $R_{\Gamma, \chi}$ in \eqref{e:R_Gammachi}.
\begin{lemma}\label{lem:|RGammachi|}
For every $c\in G_{N, q}$, we have
$$
	\left|\left\{a\;:\; \tfrac{a}{c}\in R_{\Gamma, \chi}\right\}\right|
	=
	\frac{1}{2}
	\begin{cases}
	\varphi\!\left(\gcd(c, N/c)\right)
	& \text{if } \gcd(c, N/c)>2, \\
	(-1)^{\ve} \chi_{\alpha(N, c)}(-1)+1 & \text{if } \gcd(c, N/c)\leq 2.
	\end{cases}
$$
Here $\chi^{(\ve)}(V) = (-1)^{\ve}$ and $\alpha(N, c)$ is given in \eqref{e:alpha}.
\end{lemma}

For any $m\mid N$, we set
$$
	F_m^{\ve} = \{(\chi_1, \chi_2)\;:\; (m, \chi_1, \chi_2)\in F^{\ve}\}
	=
	\{(\chi_1, \chi_2)\in F_m\;:\; \chi_1(-1) = (-1)^{\ve}\}.
$$
\begin{lemma}\label{lem:|Fm_ve|}
Assume $m\in G_{N, q}$.
Then for any $(\chi_1, \chi_2)\in F_m$, $(\chi_1, \chi_2)\in F_m^{\ve}$ holds
if and only if
\begin{equation}\label{e:varpi_chi_-1}
	\varpi(-1) \prod_{p\mid N, \ord_p(q) > \ord_p(m)} \chi_p(-1)  = (-1)^{\ve},
\end{equation}
where $\varpi\in X_{\gcd(m, N/m)}$ is the image of $(\chi_1, \chi_2)$ under the map \eqref{e:map_Fm_X}.
It follows that
\begin{equation}\label{e:|Fm_ve|}
	|F_m^{\ve}|
	=
	\frac{1}{2}
	\begin{cases}
	\varphi(\gcd(m, N/m)) & \text{if } \gcd(m, N/m) >2, \\
	(-1)^{\ve} \chi_{\alpha(N, m)}(-1)+1 & \text{if } \gcd(m, N/m)\leq 2,
	\end{cases}
\end{equation}
\end{lemma}
\begin{proof}
Using \eqref{e:map_Fm_X} and $\chi_{1, p}(-1) = \chi_{2, p}(-1) \chi_p(-1)$,
which holds for every prime $p$ since $q(\chi\chi_2\overline{\chi_1}) = 1$,
we see that for the left hand side of \eqref{e:varpi_chi_-1} equals to $\chi_1(-1)$.
This proves the first statement of the lemma.
It follows that $|F_m^{\ve}| = \frac{1}{2} \varphi(\gcd(m, N/m))$ if $\gcd(m, N/m)>2$ since then exactly half of the characters of the group $(\Z/\gcd(m, N/m)\Z)^\times$ are even and half are odd.

It remains to prove the statement \eqref{e:|Fm_ve|} in the case $\gcd(m, N/m)\leq 2$.
In this case $X_{\gcd(m, N/m)}$ consists of the trivial character only,
and we see that our task is to prove that
$$
	\prod_{p\mid N, \ord_p(q) > \ord_p(m)} \chi_p(-1) = \chi_{\alpha(N, m)}(-1).
$$
But $\chi_p(-1)=1$ whenever $p\nmid q$.
Hence we are done if we can show that for all primes $p$
\begin{align}\label{lem:|Fm_ve|PF1}
	\ord_p(\alpha(N, m)) = \begin{cases}
	\ord_p(N) & \text{if } \ord_p(q) > \ord_p(m), \\
	0 & \text{if } 1\leq \ord_p(q) \leq \ord_p(m).
	\end{cases}
\end{align}
When $p>2$ this claim is immediate from the definition of $\alpha(N, m)$ in \eqref{e:alpha},
since it follows from $\gcd(m, N/m) \leq 2$ that $\ord_p(m)\in \{0, \ord_p(N)\}$ for every such prime $p$.
Similarly the claim is immediate also for $p=2$ when $\gcd(m, N/m)=1$.

Finally assume $p=2$ and $\gcd(m, N/m)=2$.
Then $\ord_p(N) \geq 2$, $\ord_p(m)\in\{1, \ord_p(N)-1\}$ and \eqref{e:alpha} 
implies that $\ord_p(\alpha)=0$ if $\ord_p(m) = \ord_p(N)-1$, otherwise $\ord_p(\alpha)=\ord_p(N)$.
Also $\ord_p(q) < \ord_p(N)$ since $m\in G_{N,q}$;
and recall that $\ord_p(q)$ cannot equal to $1$ since there is no primitive character modulo $2$.
From these observations we see that \eqref{lem:|Fm_ve|PF1} holds also when $p=2$ and $\gcd(m, N/m)=2$.
\end{proof}

It follows from Lemma~\ref{lem:|RGammachi|} and Lemma~\ref{lem:|Fm_ve|} that
\begin{equation}\label{e:|Fve|_sumG}
	|F^{\ve}| = \sum_{m\in G_{N, q}} |F_m^{\ve}| = |R_{\Gamma, \chi}|.
\end{equation}
In particular $M(s)$ is a square matrix.
\begin{lemma}\label{lem:detM}
We have $\det(Q(s)) = 1$ and
\begin{equation}\label{e:detM}
	\det(M(s)) = c(N, \chi) \prod_{m\in G_{N, q}} \left(\frac{N}{\gcd(m, N/m)}\right)^{|F_v^{\ve}|s}
	\prod_{(m, \chi_1, \chi_2)\in F^{\ve}} \prod_{p\mid m}
	(1-\chi_1\chi_2(p) p^{-2s}),
\end{equation}
for some nonzero constant $c(N, \chi)\in \C$.
Here $Q(s)$ and $M(s)$ are given as in \eqref{e:Qs} and \eqref{e:Ms}
\end{lemma}
\begin{proof}
Let us fix an ordering $F^{\ve}$ such that $(m, \chi_1, \chi_2)$ comes before $(m', \chi_1', \chi_2')$ whenever $m< m'$.
Giving both the rows and columns of $Q(s)$ this ordering, then $Q(s)$ is lower triangular by definition,
and has all diagonal entries equal to one.
Therefore $\det(Q(s))=1$.

It also follows that $\det(M(s)) = \det(Q(s) M(s))$.
Now $Q(s) M(s)$ is a square matrix whose rows (respectively columns) are naturally indexed by $(m, \chi_1, \chi_2)\in F^{\ve}$ (respectively $\frac{a}{c}\in R_{\Gamma, \chi})$.
Let us view $Q(s)M(s)$ as a block matrix, with the blocks indexed by $m$ (row) and $c$ (column).
Note that Lemma~\ref{lem:|RGammachi|} and Lemma~\ref{lem:|Fm_ve|}
then show that each diagonal block is a square matrix;
furthermore, \eqref{e:QsMs} implies that $Q(s)M(s)$ is upper block triangular.
Hence, again using \eqref{e:QsMs},
we see that the determinant is given by the expression in the right hand side of \eqref{e:detM}, for some constant $C = c(N, \chi)\in \C$.

It also follows that in order to prove $C\neq 0$, it suffices to check that for any $c\in G_{N, q}$ for which
\begin{equation}\label{e:Rc}
	R_c = \left\{a\;:\; \tfrac{a}{c}\in R_{\Gamma, \chi}\right\}
\end{equation}
is nonempty, the $|R_c|\times |R_c|$-matrix $(\chi_2(a))$, with rows indexed by $(\chi_1, \chi_2)\in F_c$ and columns indexed by $a\in R_c$ (recall $|F_c| = |R_c|$)
has nonvanishing determinant.
This is trivial if $\gcd(c, N/c) \leq 2$ since then $R_c=\{1\}$ (if $R_c\neq \emptyset$).
Now assume that $\gcd(c, N/c) >2$,
and set  $t=|F_c|=|R_c|=\frac{1}{2}\varphi(\gcd(c, N/c))$.
Using Lemma~\ref{lem:map_Fm_X} and Lemma~\ref{lem:|Fm_ve|}
we find that by multiplying the columns with appropriate constants of absolute value one,
the determinant can be transformed into $\det(\psi_i(g_j))$ where $g_1,
\ldots, g_t$ are the elements and $\psi_1, \ldots, \psi_t$ are the
characters of the Abelian group $(\Z/\gcd(c, N/c)\Z)^{\times}/\{\pm1\}$.
Such determinant is nonvanishing, since multiplying the matrix with its
conjugate transpose gives $t$ times the $t\times t$ identity matrix
(cf.~\cite[p.~149]{Hux84}).
\end{proof}

Now by \eqref{e:eE_chi1chi2} and \eqref{e:vecB_vecEGamma} we have
$$
	\vecE^{\Gamma}(s) = M(s)^{-1} L(s)^{-1} D(s) P L(1-s) M(1-s) \vecE^{\Gamma}(1-s)
$$
and hence (by the same type of uniqueness argument as in \cite[below (2.32)]{BS07})
the ``$\Gamma$-scattering matrix" introduced in \cite[(2.33)]{BS07} is
$$
	\Phi^{\Gamma}(s) = M(s)^{-1} L(s)^{-1} D(s) P L(1-s) M(1-s).
$$
Hence by \eqref{e:detP}, \eqref{e:detD} and Lemma~\ref{lem:detM},
\begin{multline}\label{e:detPhiGamma}
	\det(\Phi^{\Gamma}(s))
	=
	(-1)^{\frac{|F^{\ve}|-|F_0^{\ve}|}{2}}
	\left(\prod_{m\in G_{N, q}} \left(\frac{N}{\gcd(m, N/m)}\right)^{|F_m^{\ve}|}\right)^{1-2s}
	\left(\frac{\pi^{2s-1} \Gamma(1-s)}{\Gamma(s)}\right)^{|F^{\ve}|}
	\\
	\times
	\prod_{(m, \chi_1, \chi_2)\in F^{\ve}} \left(q_1^{1-2s} \frac{L(2-2s, \chi_1\chi_2\omega_m)}{L(2s, \chi_1\chi_2\omega_m)}\right),
\end{multline}
where $\omega_m$ denotes the trivial Dirichlet character modulo $m$.
We note that for each $(m, \chi_1, \chi_2)\in F^{\ve}$, $\chi_1\chi_2\omega_m$ is a Dirichlet character modulo $mq_1$.
\begin{remark}
By an entirely similar discussion one computes the determinant of the
scattering matrix $\Phi(s)$ for $\langle\Gamma^+, \chi\rangle$.
We will not need this formula, but we state it here for possible future reference:
\begin{multline}\label{e:detPhi}
	\det(\Phi(s)) = (-1)^{\frac{|F|-|F_0|}{2}} \left(\prod_{m\in G_{N, q}}\left(\frac{N}{\gcd(m, N/m)}\right)^{\varphi(\gcd(m, N/m))}\right)^{1-2s}
	\left(\frac{\pi^{2s-1} \Gamma(1-s)}{\Gamma(s)}\right)^{|F|}
	\\
	\times
	\prod_{(m, \chi_1, \chi_2)\in F} \left(q_1^{1-2s} \frac{L(2-2s, \chi_1\chi_2\omega_m)}{L(2s, \chi_1\chi_2\omega_m)}\right).
\end{multline}
\end{remark}
From \eqref{e:detPhiGamma} we obtain:
\begin{multline}\label{e:varphiGamma_ratio}
	\frac{(\varphi^{\Gamma})'}{\varphi^\Gamma}(s)
	=
	-2\sum_{m\in G_{N, q}} |F_m^{\ve}| \log\!\left(\frac{N}{\gcd(m, N/m)}\right)
	+
	|F^{\ve}| \left(2\log \pi - \frac{\Gamma'}{\Gamma}(1-s) - \frac{\Gamma'}{\Gamma}(s)\right)
	\\
	-2\sum_{(m, \chi_1, \chi_2)\in F^{\ve}} \left(\log q_1 + \frac{L'}{L}(2-2s, \chi_1\chi_2\omega_m)
	+
	\frac{L'}{L}(2s, \chi_1\chi_2\omega_m)\right).
\end{multline}

We can now evaluate the integral appearing in \eqref{e:Eis}:
\begin{lemma}\label{lem:Eis_integral}
\begin{multline}\label{e:Eis_integral}
	\frac{1}{4\pi} \int_{\R} h(r)
	\frac{(\varphi^{\Gamma})'}{\varphi^{\Gamma}}\!\left(\frac{1}{2}+ir\right)\; dr
	\\
	=
	\left(|F^{\ve}| \log \pi - \sum_{m\in G_{N, q} } |F_m^{\ve}| \log\!\left(\frac{N}{\gcd(m, N/m)}\right)
	-
	\sum_{(m, \chi_1,\chi_2)\in F^{\ve}}  \log q_1\right)
	g(0)
	\\
	-
	\frac{|F^{\ve}|}{2\pi} \int_\R h(r) \frac{\Gamma'}{\Gamma}\!\left(\frac{1}{2}+ir\right) \; dr
	-
	\frac{|F_0^{\ve}|}{2} h(0)
	+
	2\sum_{r=1}^\infty \frac{\Lambda(r) \{\chi\}_{\ve}(r)}{r} g(2\log r),
\end{multline}
where
\begin{equation}\label{e:chi_ve}
	\{\chi\}_{\ve}(r) = \sum_{(m, \chi_1, \chi_2)\in F^{\ve}} \chi_1\chi_2\omega_m(r).
\end{equation}
\end{lemma}
\begin{proof}
This follows from \eqref{e:varphiGamma_ratio} using
\begin{multline*}
	\frac{1}{4\pi} \int_\R h(r) \left(\frac{L'}{L} (1-2ir, \chi_1\chi_2\omega_m)
	+ \frac{L'}{L} (1+2ir, \chi_1\chi_2\omega_m) \right) \; dr
	\\
	=
	-\sum_{r=1}^\infty \frac{\Lambda(r) \chi_1\chi_2\omega_m(r)}{r} g(2\log r)
	+
	\delta_{\chi_1=\overline{\chi_2}} \frac{1}{4} h(0),
\end{multline*}
which follows by imitating \cite[p.~509]{Hej83} and \eqref{e:Fve0}.
\end{proof}
\begin{remark}
We stress that our $\{\chi\}_{\ve}(r)$ (cf.~\eqref{e:chi_ve})
is not exactly a generalization of the number which was denoted
by that symbol in \cite[(2.41)]{BS07}.
Rather, in the special case when $N$ is squarefree,
``$\{\chi\}_{\ve}(r)$" in \cite{BS07} equals
$2^{1-\omega(N/\gcd(r,N))}\{\chi\}_{\ve}(r)$ in our present notation.
\end{remark}
In order to evaluate \eqref{e:Eis}, we also need to compute $\tr\Phi^{\Gamma}(\frac{1}{2})$.
\begin{lemma}\label{lem:TrPhiGamma_1/2}
We have
\begin{equation}\label{e:TrPhiGamma_1/2}
	\tr\Phi^{\Gamma}\!\left(\tfrac{1}{2}\right) = -|F_0^{\ve}|.
\end{equation}
\end{lemma}
\begin{proof}
We have
\begin{multline*}
	\tr \Phi^{\Gamma}\!\left(\frac{1}{2}\right)
	=
	\lim_{s\to \frac{1}{2}} \tr (M(s)^{-1} L(s)^{-1} D(s) P L(1-s) M(1-s))
	\\
	=
	\lim_{s\to \frac{1}{2}} \tr(L(s)^{-1} D(s) P L(1-s) M(1-s) M(s)^{-1}),
\end{multline*}
and using the definition of $L(s)$, $D(s)$ and $P$,
we find that $L(s)^{-1} D(s) P L(1-s)$ is the $|F^{\ve}|\times |F^{\ve}|$
matrix that in row $(m, \chi_1, \chi_2)\in F^{\ve}$ has all entries
zero except in column $(q_1 m/q_2, \overline{\chi_2}, \overline{\chi_1})$,
where the entry is
\begin{equation}\label{e:trPhiGamma_ratio_L}
	\frac{q_1^{1-s}}{q_2} \pi^{2s-1}
	\frac{\tau(\chi_2)}{\tau(\overline{\chi_1})}
	\frac{\Gamma(1-s)}{\Gamma(s)}
	\frac{L(2-2s, \overline{\chi_1\chi_2})}{L(2s, \chi_1\chi_2)}.	
\end{equation}
If $\chi_1\neq \overline{\chi_2}$ then $\chi_1\chi_2$ is a nontrivial
Dirichlet character and thus $L(s, \chi_1\chi_2)$ is holomorphic at
$s=1$ with $L(1, \chi_1\chi_2)\neq 0$ (cf., e.g., \cite[\S X.11]{Coh80}),
and similarly $L(1, \overline{\chi_1\chi_2})\neq 0$.
However if $\chi_1=\overline{\chi_2}$, then
$$
	L(2s, \chi_1\chi_2) = L(2s, \overline{\chi_1\chi_2})
	= \zeta(2s) \prod_{p\mid q_1}(1-p^{-2s})
	= (\prod_{p\mid q_1} (1-p^{-1}))(2s-1)^{-1}+O(1)
$$
as $s\to \frac{1}{2}$.
Hence when $s\to \frac{1}{2}$, the number in \eqref{e:trPhiGamma_ratio_L} tends to
$$
	\begin{cases}
	-1 & \text{if } \chi_1=\overline{\chi_2}, \\
	\sqrt{\frac{q_1}{q_2}} \frac{\tau(\chi_2)}{\tau(\overline{\chi_1})}
	\frac{L(1, \overline{\chi_1\chi_2})}{L(1, \chi_1\chi_2)}
	& \text{if } \chi_1\neq \overline{\chi_2}.
	\end{cases}
$$
Also note that $M(1-s) M(s)^{-1}$ tends to the $|F^{\ve}|\times |F^{\ve}|$ identity matrix as $s\to \frac{1}{2}$.
Indeed, $M(s)$ is by definition a holomorphic function of $s$ in the entire complex plane,
and Lemma~\ref{lem:detM} shows that $M(s)^{-1}$ is also holomorphic when $\Re(s)\neq 0$.
Hence we obtain the stated formula.
\end{proof}

\subsection{Continuous spectrum $\Eis(\Gamma, \chi)$}
Recall that we write $\Eis(N,\chi;n)$ for the contribution from the continuous spectrum
in the trace formula in Theorem \ref{thm:STF_N};
thus $\Eis(N,\chi;n)=\Eis(\Gamma, \chi^{(0)})+n\Eis(\Gamma, \chi^{(1)})$
where $\Eis(\Gamma,\chi^{(\ve)})$ is the corresponding contribution (cf.\ \eqref{e:Eis})
in the trace formula for $(\Gamma_0^\pm (N), \chi^{(\ve)})$.
Our aim in this section is to prove the following proposition,
giving an explicit formula for $\Eis(N,\chi;n)$.
Let
\begin{multline}\label{e:Psi5}
	\Psi_5(p^{e_p},p^{s_p})
	=
	\frac{\Psi_1(p^{e_p}, p^{s_p})-2}{p-1}
	\\
	+\max\left\{\lc\tfrac{e_p}{2}\rc, s_p\right\} p^{\min\left\{\lf\frac{e_p}{2}\rf, e_p-s_p\right\}}
	+
	\max\left\{\lc\tfrac{e_p+1}{2}\rc, s_p\right\} p^{\min\left\{\lf\frac{e_p-1}{2}\rf, e_p-s_p\right\}}.
\end{multline}
Also let
$$
	B_p(x) =
	\begin{cases}
	xp^x - 2\frac{p^x-1}{p-1} & \text{if } x\in \Z_{\geq 0}, \\
	0 & \text{if } x< 0.
	\end{cases}
$$
and
\begin{equation}\label{e:Psi6}
	\Psi_6(p^{e_p}, p^{s_p})
	=
	\begin{cases}
	B_p(e_p-s_p) + s_p p^{e_p-s_p} & \text{if } e_p < 2s_p, \\
	B_p\!\left(\lf\tfrac{e_p}{2}\rf\right) + B_p\!\left(\lf\tfrac{e_p-1}{2}\rf\right)
	- B_p(s_p-1) + s_pp^{s_p-1}
	& \text{if } e_p \geq 2s_p.
	\end{cases}
\end{equation}
Finally, let
\begin{equation}\label{e:Psi7}
	\Psi_7(p^{e_p}, p^{s_p})
	=
	\begin{cases}
	\begin{cases}
	-e_p+1 & \text{if } s_p=0, \\
	-e_p+2s_p & \text{if } s_p\geq 1, \\
	\end{cases}
	& \text{if } p>2, \\
	\begin{cases}
	-e_p+3 & \text{if } s_p=0\text{ and } e_p\geq 3, \\
	-e_p+2s_p+1 & \text{if } e_p > s_p\geq 2, \\
	e_p & \text{if } e_p=s_p\geq 2, \\
	0 & \text{otherwise,}
	\end{cases}
	& \text{if } p=2.
	\end{cases}
\end{equation}

\begin{proposition}\label{prop:Eis_N}
\begin{multline}\label{e:Eis_N_1}
	\Eis(\Gamma, \chi; 1)
	=
	\bigg\{
	\Psi_1(N, q) \log\pi
	-
	\Psi_1(N, q) \sum_{p\mid N} \frac{\Psi_5(p^{e_p}, p^{s_p})+\Psi_6(p^{e_p}, p^{s_p})}{\Psi_1(p^{e_p}, p^{s_p})} \log p
	\bigg\} g(0)
	\\
	-
	\Psi_1(N, q) \frac{1}{2\pi} \int_\R h(r) \frac{\Gamma'}{\Gamma}\!\left(\frac{1}{2}+ir\right) \; dr
	-
	\frac{1}{4} I_\chi \Psi_2(N, q)  h(0)
	\\
	+
	2\sum_{m=1}^\infty \frac{\Lambda(m)}{m}
	\prod_{p\mid N} \Big[(\Re\chi_p(m)) \Phi_+(p^{e_p}; m) \Big]\,
	g(2\log m)
\end{multline}
and
\begin{multline}\label{e:Eis_N_-1}
	\Eis(\Gamma, \chi; -1)
	\\
	=
	-I_\chi
	\bigg\{
	2^{\omega(N)} \big(\Omega_1(N, q) (\log N - \log (2\pi)) + \log 2\big)
	+
	\Omega_1(N, q) 2^{\omega(N)-1} \sum_{p\mid N} \Psi_7(p^{e_p}, p^{s_p}) \log p
	\bigg\} g(0)
	\\
	-
	I_\chi 2^{\omega(N)} \Omega_1(N, q)
	\frac{1}{2\pi} \int_\R h(r) \frac{\Gamma'}{\Gamma}\!\left(\frac{1}{2}+ir\right) \; dr
	\\
	-
	\frac{1}{4} I_\chi
	\bigg\{
	I_{q, 4}\, \chi(\sqrt{-1}) \,
	\min\{e_2+1, 4\}
	\prod_{p\mid N, p\equiv_41} \Psi_2(p^{e_p}, p^{s_p})
	\prod_{p\mid N, p\equiv_4-1} 2
	\bigg\} h(0)
	\\
	+
	2\sum_{m=1}^\infty \frac{\Lambda(m)}{m}
	\prod_{p\mid N}
	\Big[(\{\Re/i\Im\}\chi_p(m))\Phi_-(p^{e_p}; m) \Big]\,
	g(2\log m).
\end{multline}
Here $\Psi_1$, $\Psi_5$, $\Psi_6$, $\Psi_7$, $\Psi_2$ and $\Phi_\pm$ are given in \eqref{e:Psi1}, \eqref{e:Psi5}, \eqref{e:Psi6}, \eqref{e:Psi7}, \eqref{e:Psi2} and \eqref{e:Phi} respectively.
Also $\sqrt{-1}$ denotes a square root of $-1$ in $\Z/q\Z$
(as exists when $I_{q,4}=1$; and $\chi(\sqrt{-1})$ is independent of the
choice of the square root when $I_\chi=1$).
\end{proposition}

The proposition will be proved by
making the right hand sides in the formulas in Lemma~\ref{lem:Eis_integral}
and Lemma \ref{lem:TrPhiGamma_1/2} more explicit.
We will carry this out in a sequence of lemmas.

\begin{lemma}\label{lem:|Fm_ve|log}
Recall that $\ve\in \{0, 1\}$ is given by $\chi^{(\ve)}(V) = (-1)^{\ve}$.
We have
\begin{multline}\label{e:|Fm_ve|log}
	\sum_{m\in G_{N,q}} |F_m^{\ve}| \log\!\left(\frac{N}{\gcd(m, N/m)}\right)
	=
	\frac{1}{2}\Psi_1(N, q) \sum_{p\mid N} 
	\frac{\Psi_5(p^{e_p}, p^{s_p})}{\Psi_1(p^{e_p}, p^{s_p})}\log p
	\\
	+
	\frac{1}{2} (-1)^{\ve}
	I_\chi 2^{\omega(N)}
	\bigg\{
	\Omega_1(N, q)\log N - (\Omega_1(N, q) -1)\log 2
	\bigg\}.
\end{multline}
Here $\Psi_1$, $\Psi_5$ and $\Omega_1$ are given in \eqref{e:Psi1}, \eqref{e:Psi5} and \eqref{e:Omega1} respectively.
Furthermore,
\begin{equation}\label{e:|Fve|}
	|F^{\ve}|
	=
	\frac{1}{2}\Psi_1(N, q)
	+
	\frac{1}{2} (-1)^{\ve}
	I_\chi 2^{\omega(N)} \Omega_1(N, q).
\end{equation}
\end{lemma}
\begin{proof}
By \eqref{e:|Fm_ve|}, we have
\begin{multline*}
	\sum_{m\in G_{N,q}} |F_m^{\ve}| \log\left(\frac{N}{\gcd(m, N/m)}\right)
	=
	\frac{1}{2}\sum_{m\mid N, q\mid \frac{N}{\gcd(m, N/m)}}
	\varphi(\gcd(m, N/m)) \log\left(\frac{N}{\gcd(m, N/m)}\right)
	\\
	+
	\frac{1}{2} (-1)^{\ve}
	\sum_{\substack{m\mid N, \gcd(m, N/m)\leq 2, \\
	q\mid \frac{N}{\gcd(m, N/m)}}}
	\chi_{\alpha(N, m)}(-1) \log\!\left(\frac{N}{\gcd(m, N/m)}\right).
\end{multline*}

For the first term, recalling Lemma~\ref{lem:varphi_gcd},
let
$$
	f(x) = \sum_{m\mid N, q\mid \frac{N}{\gcd(m, N/m)}} \varphi(\gcd(m, N/m)) (\gcd(m,N/m))^{x}.
$$
Then we have
\begin{equation}\label{e:f0_Psi1}
	f(0) = \sum_{m\mid N, q\mid \frac{N}{\gcd(m, N/m)}} \varphi(\gcd(m, N/m))
	=
	\Psi_1(N, q)
\end{equation}
and
$$
	f'(0) = \sum_{m\mid N, q\mid \frac{N}{\gcd(m, N/m)}}
	\varphi(\gcd(m, N/m)) \log(\gcd(m,N/m)).
$$
By \eqref{e:varphi_gcd},
\begin{multline*}
	f'(0)
	=
	\sum_{p\mid N}
	\bigg(\prod_{p'\mid N, p'\neq p}\Psi_1({p'}^{e_{p'}}, {p'}^{s_{p'}})\bigg)
	\log p
	\bigg[\frac{-\Psi_1(p^{e_p}, p^{s_p})+2}{p-1}
	\\
	+
	\min\{\lf\tfrac{e_p}{2}\rf, e_p-s_p\} p^{\min\{\lf\frac{e_p}{2}\rf, e_p-s_p\}}
	+
	\min\{\lf\tfrac{e_p-1}{2}\rf, e_p-s_p\} p^{\min\{\lf\frac{e_p-1}{2}\rf, e_p-s_p\}}
	\bigg].
\end{multline*}
Note that by definition in \eqref{e:Psi1}, $\Psi_1(p^{e_p}, p^{s_p})\neq 0$.
We have
\begin{multline*}
	\sum_{m\mid N, q\mid \frac{N}{\gcd(m, N/m)}}
	\varphi(\gcd(m, N/m)) \log\left(\frac{N}{\gcd(m, N/m)}\right)
	=
	\log N f(0)-f'(0)
	\\
	=
	\Psi_1(N, q)
	\sum_{p\mid N}
	\frac{\log p}{\Psi_1(p^{e_p}, p^{s_p})}\Psi_5(p^{e_p}, p^{s_p}).
\end{multline*}
For the second term, we get
\begin{multline*}
	\sum_{\substack{m\mid N, \gcd(m, N/m) \leq 2, \\ q\mid \frac{N}{\gcd(m, N/m)}}}
	\chi_{\alpha(N, m)} (-1) \log\!\left(\frac{N}{\gcd(m, N/m)}\right)
	\\
	=
	\log N
	\sum_{\substack{m\mid N, \gcd(m, N/m) \leq 2, \\ q\mid \frac{N}{\gcd(m, N/m)}}}
	\chi_{\alpha(N, m)}(-1)
	-
	\log 2
	\sum_{\substack{m\mid N, \gcd(m, N/m) = 2, \\ q\mid \frac{N}{\gcd(m, N/m)}}}
	\chi_{\alpha(N, m)}(-1)
	.
\end{multline*}
As in Lemma \ref{lem:sum_chiT1/c0} and its proof, we have
\begin{equation}\label{e:sum_chialpha_2}
	\sum_{\substack{m\mid N, \gcd(m, N/m) \leq 2, \\ q\mid \frac{N}{\gcd(m, N/m)}}}
	\chi_{\alpha(N, m)}(-1)
	=
	I_\chi 2^{\omega(N)} \Omega_1(N, q).
\end{equation}
and
$$
	\sum_{\substack{m\mid N, \gcd(m, N/m) = 2, \\ q\mid \frac{N}{\gcd(m, N/m)}}}
	\chi_{\alpha(N, m)}(-1)
	=
	I_\chi 2^{\omega(N)} (\Omega_1(N, q)-1).
$$

Similarly, recalling \eqref{e:|Fve|_sumG} and \eqref{e:|Fm_ve|},
$$
	|F^{\ve}| = \sum_{m\in G_{N, q}} |F_m^{\ve}|
	=
	\frac{1}{2}\sum_{m\mid N, q\mid \frac{N}{\gcd(m, N/m)}}
	\varphi(\gcd(m, N/m))
	+
	\frac{1}{2} (-1)^{\ve} \sum_{\substack{m\mid N, \gcd(m, N/m)\leq 2, \\ q\mid \frac{N}{\gcd(m, N/m)}}}
	\chi_{\alpha(N, m)}(-1),
$$
so by \eqref{e:f0_Psi1} and \eqref{e:sum_chialpha_2}, we get \eqref{e:|Fve|}.
\end{proof}

For $p\mid N$, let
\begin{equation}\label{e:Sp}
	S_p = \left\{0\leq f \leq e_p\;:\; \min\{f, e_p-f\} \leq e_p-s_p\right\}.
\end{equation}
Then note that
$$
	S_p
	=
	\left\{0\leq f\leq e_p\;:\; f\leq e_p-s_p \text{ or } f\geq s_p\right\}.
$$
\begin{lemma}\label{lem:Fp}
Fix a prime $p\mid N$.
For $f\in S_p$, if $f\geq s_p$,
$$
	F_{p^{f}} =
	\left\{(\chi_1, \chi_2)\;:\; \chi_1\in X_{p^{\min\{f, e_p-f\}}},
	\chi_2\text{ primitive, determined by }
	\cond(\chi_p\chi_2\overline{\chi_1})=1\right\}
$$
and if $f<s_p$,
$$
	F_{p^{f}}
	=
	\left\{(\chi_1, \chi_2) \;:\; \chi_2\in X_{p^{\min\{f, e_p-f\}}}, \chi_1\text{ primitive, determined by }
	\cond(\chi_p\chi_2\overline{\chi_1})=1\right\}.
$$
Here $F_{p^f}$ is given in \eqref{e:Fm}.
\end{lemma}
\begin{proof}
Similar to the proof of Lemma~\ref{lem:map_Fm_X}.
\end{proof}

\begin{lemma}\label{lem:|Fve0|}
$$
	|F_0^{\ve}|
	=
	\frac{1}{2} I_\chi
	\bigg\{
	\Psi_2(N, q)
	+
	(-1)^{\ve} I_{q, 4}\, \chi(\sqrt{-1}) \min\{e_2+1, 4\}
	\prod_{p\mid N, p\equiv_4 1} \Psi_2(p^{e_p}, p^{s_p})
	\prod_{p\mid N, p\equiv_4{-1}} 2
	\bigg\}
$$
Here $\Psi_2$ is given in \eqref{e:Psi2},
and $\sqrt{-1}$ denotes a square root of $-1$ in $\Z/q\Z$
(as exists when $I_{q,4}=1$; and $\chi(\sqrt{-1})$ is independent of the
choice of the square root when $I_\chi=1$).
\end{lemma}
\begin{proof}
Recalling \eqref{e:Fve0}, we can write
$$
	F_0^{\ve} = \left\{(m, \chi_1) \;:\;
	m\in \Z_{\geq 1}, \; m\mid N, \; q_1\mid \gcd(m, N/m), \; \cond(\chi\overline{\chi_1}^2)=1, \; \chi_1(-1)=(-1)^{\ve}\right\}.
$$
Note that $F_0^{\ve}=\emptyset$ unless $\chi$ is square in the sense that there exists a primitive character $\chi_1$ such that $\chi \overline{\chi_1}^2(n) = 1$ for any $n$ with $\gcd(n, q)=1$.
There exists such a primitive character if and only if for every prime
$p\mid N$, $\chi_p(-1)=1$, i.e.\ $\chi$ is pure.
From now on, we assume that $\chi$ is pure.

Let
$$
	F_0(N)
	= \left\{(m, \chi_1)\;:\; m\in \Z_{\geq 1}, \; m\mid N,\; q_1 \mid \gcd(m, N/m), \; \cond(\chi\overline{\chi_1}^2)=1\right\}
$$
and
$$
	S_0^{\pm} = \sum_{(m, \chi_1)\in F_0(N)}\chi_1(\pm1).
$$
Then
$$
	|F_0^{\ve}| = \frac{1}{2} \big(S_0^+ + (-1)^{\ve} S_0^-\big).
$$
Note that $S_0^\pm$ is multiplicative.
One can write
$$
	S_0^\pm
	=
	\prod_{p\mid N}
	\bigg(\sum_{f\in S_p} \sum_{(p^f, \chi_1)\in F_0(p^{e_p})} \chi_1(\pm 1)\bigg).
$$
So for $p\mid N$, let
$$
	S_{0, p}^\pm =
	\sum_{f\in S_p} \sum_{(p^f, \chi_1)\in F_0(p^{e_p})} \chi_1(\pm 1).
$$

For every prime $p\mid N$,
fix a primitive Dirichlet character $\psi_p$ of minimal conductor subject to
$\cond(\chi_p\overline{\psi_p}^2)=1$.
We further assume that
$$
	\ord_p(\cond(\psi_p))
	=
	\begin{cases}
	0 & \text{if } p=2 \text{ and } s_p=0, \\
	s_p+1 & \text{if } p=2 \text{ and } s_p\geq 3, \\
	s_p & \text{if } p>2. \\
	\end{cases}
$$
Note that $s_2$ cannot be $1$ or $2$.

When $p$ is odd, let $\xi_p$ be the quadratic character modulo $p$.
When $p=2$ let $\xi_4$ be the quadratic character modulo $4$ and $\xi_8$
be the even quadratic character modulo $8$.
Then $\xi_4\xi_8$ is another quadratic character,
and $\xi_8$, $\xi_4\xi_8$ are the only primitive characters modulo $8$.

When $\chi_1$ is the primitive character such that $\cond(\chi_p\overline{\chi_1}^2)=1$,
then
$$
	\chi_1\in\{\psi_p \xi_p^{u_p}\;:\; u_p\in \{0, 1\}\},
$$
when $p$ is odd, and
$$
	\chi_1\in\{\psi_2 \xi_4^{u_2} \xi_8^{v_2}\;:\; u_2, v_2\in \{0, 1\}\},
$$
when $p=2$.
Here $\psi_p$ is as fixed above, so $q(\chi\overline{\chi_1}^2)=1$.
Note that
$$
	\ord_p(q_1) = \begin{cases}
	\max\{s_p, u_p\} & \text{if } p\text{ odd }, \\
	s_p+1 & \text{if } p=2\text{ and } s_p\geq 3, \\
	\max\{2u_2, 3v_2\} & \text{if } p=2 \text{ and } s_p=0.
	\end{cases}
$$

When $p$ is odd, we have
\begin{multline*}
	S_{0, p}^\pm
	=
	\sum_{u_p\in\{0, 1\}}
	\sum_{\max\{s_p, u_p\} \leq \min\{f, e_p-f\}}
	\psi_p\xi_p^{u_p}(\pm 1)
	\\
	=
	\psi_p(\pm 1)
	\begin{cases}
	(e_p-1)(1+\xi_p(\pm 1)) + 2 & \text{if } s_p=0, \\
	(e_p-2s_p+1)(1+\xi_p(\pm 1)) & \text{if } s_p\geq 1\text{ and }2s_p \leq e_p, \\
	0 & \text{if } 2s_p > e_p.
	\end{cases}
\end{multline*}
Note that $1+\xi_p(-1)=0$ if $s_p>0$ and $p\equiv_4-1$.
So
$$
	S_{0, p}^+ = \Psi_2(p^{e_p}, p^{s_p})
$$
and
$$
	S_{0, p}^- = \psi_p(-1) \begin{cases}
	2 & \text{if } s_p=0 \text{ and } p\equiv_{4}-1, \\
	\Psi_2(p^{e_p}, p^{s_p}) & \text{if } p\equiv_41, \\
	0 & \text{if } s_p>0 \text{ and } p\equiv_{4}-1.
	\end{cases}
$$

When $p=2$, we have
$$
	S_{0, p}^\pm
	=
	\sum_{u_2, v_2\in\{0, 1\}}
	\begin{cases}
	\sum_{\max\{2u_2, 3v_2\} \leq \min\{f, e_p-f\}} \xi_4^{u_2} \xi_8^{v_2}(\pm 1)
	& \text{if } s_p=0, \\
	\sum_{s_p+1 \leq \min\{f, e_p-f\}} \psi_p \xi_4^{u_2} \xi_8^{v_2}(\pm 1)
	& \text{if } s_p \geq 3.
	\end{cases}
$$
So we get
$$
	S^\pm_{0, p}
	=
	\psi_p(\pm 1)
	\begin{cases}
	e_p+1 + \max\{e_p-5, 0\} (1+\xi_4(\pm 1)) + \max\{e_p-3, 0\}\xi_4(\pm 1)
	& \text{if } s_p=0, \\
	2\max\{e_p-2s_p-1, 0\} \psi_p(\pm 1) (1+\psi_4(\pm 1))
	& \text{if } s_p\geq 3.
	\end{cases}
$$
Hence
$$
	S_{0, p}^+ = \Psi_2(p^{e_p}, p^{s_p})
$$
and
$$
	S_{0, p}^- =
	\begin{cases}
	4 & \text{if } e_p\geq 3 \text{ and } s_p=0, \\
	e_p+1 & \text{if } e_p\in\{1, 2\} \text{ and } s_p=0, \\
	0 & \text{otherwise.}
	\end{cases}
$$

\end{proof}

\begin{lemma}\label{lem:sum_logq1}
\begin{multline*}
	\sum_{(m, \chi_1, \chi_2)\in F^{\ve}}
	\log q_1
	=
	\frac{1}{2}
	\Psi_1(N, q) \sum_{p\mid N} \frac{\Psi_6(p^{e_p}, p^{s_p})}{\Psi_1(p^{e_p}, p^{s_p})} \log p
	\\
	+
	(-1)^{\ve} I_\chi 2^{\omega(N)-2} \Omega_1(N, q)
	\sum_{p\mid N} \Psi_7(p^{e_p}, p^{s_p}) \log p.
\end{multline*}
Here $\Psi_1$, $\Omega_1$, $\Psi_6$ and $\Psi_7$ are given in \eqref{e:Psi1},
\eqref{e:Omega1}, \eqref{e:Psi6} and \eqref{e:Psi7}.
\end{lemma}
\begin{proof}
Let
$$
	S^{\pm} = \sum_{(m, \chi_1, \chi_2)\in F} \chi_1(\pm 1) \log q_1.
$$
Then
$$
	\sum_{(m, \chi_1, \chi_2)\in F^{\ve}} \log q_1
	=
	\frac{1}{2}\big(S^++ (-1)^{\ve} S^-).
$$
For $n\in \{\pm1\}$, let
$$
	f_n (x)
	=
	\sum_{(m, \chi_1, \chi_2)\in F}
	\chi_1(n) q_1^x.
$$
Then
$$
	f_n'(0) = \begin{cases}
	S^+ & \text{if } n=1, \\
	S^- & \text{if } n=-1.
	\end{cases}
$$
Recalling \eqref{e:Sp},
$$
	f_n(x)
	=
	\prod_{p\mid N}
	\bigg(
	\sum_{f\in S_p}
	\sum_{(\chi_1, \chi_2)\in F_{p^{f}}} \chi_1(n) q_1^x
	\bigg).
$$
For each prime $p\mid N$,
$$
	f_{n, p}(x) = \sum_{f\in S_p}
	\sum_{(\chi_1, \chi_2)\in F_{p^{f}}} \chi_1(n) q_1^x.
$$
Then
$$
	f_n'(0)
	=
	\sum_{p\mid N} f_{n, p}'(0) \prod_{p'\mid N, p'\neq p} f_{n, p'}(0).
$$

For each prime $p\mid N$ and $f\in S_p$,
by Lemma~\ref{lem:Fp}, we have
$$
	\sum_{(\chi_1, \chi_2)\in F_{p^{f}}} \chi_1(n) q_1^x
	=
	\begin{cases}
	\sum_{\ell=0}^{\min\{f, e_p-f\}}
	\big(\sum_{\substack{\psi\text{ primitive}\\ \cond(\psi) = p^\ell}}
	\psi(n) \big)
	p^{\ell x}
	& \text{if } f\geq s_p, \\
	\sum_{\ell=0}^{\min\{f, e_p-f\}}
	\big(\sum_{\substack{\psi\text{ primitive}\\ \cond(\psi) = p^\ell}}
	\chi_p \psi(n) \big)
	p^{s_p x}
	& \text{if } f< s_p.
	\end{cases}
$$
Note that
$$
	\sum_{\substack{\psi\text{ primitive}\\ \cond(\psi) = p^\ell}}
	\psi(n)
	=
	\begin{cases}
	1 & \text{if } \ell=0, \\
	p^{\ell-2}(p-1)^2 & \text{if } n=1 \text{ and } \ell \geq 2, \\
	p-2 & \text{if } n=1 \text{ and } \ell=1, \\
	-1 & \text{if } p>2, n=-1 \text{ and } \ell=1, \\
	& \text{ or } p=2, n=-1 \text{ and } \ell=2, \\
	0 & \text{otherwise.}
	\end{cases}
$$
So when $f\geq s_p$,
\begin{multline*}
	\sum_{(\chi_1, \chi_2)\in F_{p^{f}}} \chi_1(n) q_1^{x}
	\\
	=
	\begin{cases}
	1 & \text{if } \min\{f, e_p-f\}=0, \\
	& \text{or } p=2, n=-1 \text{ and } \min\{f, e_p-f\}=1, \\
	\frac{(p-1)^2 p^{\min\{f, e_p-f\}(x+1)+x-1} - (p^x-1)^2}{p^{x+1}-1}
	& \text{if } n=1 \text{ and } \min\{f, e_p-f\}\geq 1, \\
	1-p^{x(1+\delta_{p=2})} & \text{if } n=-1 \text{ and } \min\{f, e_p-f\} \geq 1+\delta_{p=2}.
	\end{cases}
\end{multline*}
When $f< s_p$, we have
\begin{multline*}
	\sum_{(\chi_1, \chi_2)\in F_{p^{f}}} \chi_1(n) q_1^{x}
	\\
	=
	\chi_p(n)p^{s_px}
	\begin{cases}
	1 & \text{if } \min\{f, e_p-f\}=0, \\
	& \text{or } p=2, n=-1 \text{ and } \min\{f, e_p-f\}=1, \\
	(p-1) p^{\min\{f,e_p-f\}-1}
	& \text{if } n=1\text{ and } \min\{f, e_p-f\} \geq 1, \\
	0 & \text{if } n=-1 \text{ and } \min\{f, e_p-f\}\geq 1+\delta_{p=2}.
	\end{cases}
\end{multline*}

For $n=1$ and $s_p> \lf\frac{e_p}{2}\rf$, we have
\begin{multline*}
	f_{1, p}(x)
	=
	\sum_{f\in S_p}
	\sum_{(\chi_1, \chi_2)\in F_{p^{f}}} \chi_1(n) q_1^x
	\\
	=
	p^{s_p x + e_p-s_p}+1
	+
	\frac{1}{p^{x+1}-1}
	\bigg(
	(p-1)^2 p^{2x} \frac{p^{(e_p-s_p)(x+1)}-1}{p^{x+1}-1}
	-
	(p^x-1)^2(e_p-s_p+1)
	\bigg).
\end{multline*}
For $s_p \leq \lf\frac{e_p}{2}\rf$, we have
\begin{multline*}
	f_{1, p}(x)
	=
	p^{s_p(x+1)-1}+1
	+
	\frac{(p-1)^2 p^{x-1}}{(p^{x+1}-1)^2}
	\big(
	p^{(\lf\frac{e_p}{2}\rf +1)(x+1)} + p^{\lc\frac{e_p}{2}\rc(x+1)}
	- p^{s_p(x+1)} - p^{x+1}
	\big)
	\\
	-(e-s)\frac{(p^x-1)^2}{p^{x+1}-1},
\end{multline*}
So we have
$$
	f_{1, p}(0) =
	\begin{cases}
	2p^{e_p-s_p} & \text{if } e_p< 2s_p, \\
	p^{\lf\frac{e_p}{2}\rf} + p^{\lf\frac{e_p-1}{2}\rf} & \text{if } e_p \geq 2s_p
	\end{cases}
	=
	\Psi_1(p^{e_p}, p^{s_p})
$$
and
$$
	f_{1, p}'(0)
	=
	\log p
	\begin{cases}
	B_p(e_p-s_p) + s_p p^{e_p-s_p} & \text{if } e_p < 2s_p, \\
	B_p\!\left(\lf\tfrac{e_p}{2}\rf\right) + B_p\!\left(\tfrac{e_p-1}{2}\right)
	- B_p(s_p-1) + s_pp^{s_p-1}
	& \text{if } e_p \geq 2s_p.
	\end{cases}
$$
Hence
$$
	S^+
	=
	\Psi_1(N, q) \sum_{p\mid N} \frac{\Psi_6(p^{e_p}, p^{s_p})}{\Psi_1(p^{e_p}, p^{s_p})} \log p.
$$

For $n=-1$ and $p>2$, we have
$$
	f_{-1, p}(x)
	=
	\begin{cases}
	2+(e_p-1)(1-p^x) & \text{if } s_p=0, \\
	1+\chi_p(-1) p^{s_p x} + (e_p-s_p) (1-p^x) & \text{if } s_p > 0.
	\end{cases}
$$
So
$$
	f_{-1, p}(0) =
	1+\chi_p(-1)
$$
and
$$
	f_{-1, p}'(0) = \log p
	\begin{cases}
	-e_p+1 & \text{if } s_p=0, \\
	s_p (1+\chi_p(-1))-e_p & \text{if } s_p>0.
	\end{cases}
$$

For $n=-1$ and $p=2$, we have
$$
	f_{-1, p}(x)
	=
	\begin{cases}
	1+\chi_p(-1) p^{s_p x} & \text{if } e_p=s_p, \\
	2(1+\chi_p(-1) p^{s_px}) + (e_p-s_p-1)(1-p^{2x}) & \text{if } e_p>s_p \geq 2, \\
	2 & \text{if } s_p=0 \text{ and } e_p=1, \\
	3 & \text{if } s_p=0 \text{ and } e_p=2, \\
	4+(e_p-3)(1-p^{2x}) & \text{if } s_p=0 \text{ and } e_p\geq 3.
	\end{cases}
$$
So we get
$$
	f_{-1, p}(0)
	=
	(1+\chi_p(-1))
	\begin{cases}
	1 & \text{if } e_p=s_p \geq 2 \text{ or } e_p=1, \\
	2 & \text{if } e_p>s_p \text{ and }e_p\geq 3, \\
	\frac{3}{2} & \text{if } s_p=0 \text{ and } e_p=2, \\
	\end{cases}
$$
and
$$
	f_{-1, p}'(0)
	=
	2\log p
	\begin{cases}
	\frac{1}{2} s_p \chi_p(-1) & \text{if } e_p=s_p\geq 2, \\
	s_p \chi_p(-1) -e_p+s_p+1 & \text{if } e_p>s_p\geq 2, \\
	-e_p+3 & \text{if } s_p=0 \text{ and } e_p\geq 3, \\
	0 & \text{otherwise.}
	\end{cases}
$$
Since $1+\chi_p(-1)=0$ unless $\chi_p$ is even,
we can conclude that $S^-=0$ unless $\chi$ is pure.
Hence we get
$$
	S^-
	=
	\sum_{p\mid N} f_{-1, p}'(0) \prod_{p'\mid N, p'\neq p} f_{-1, p'}(0)
	=
	I_\chi 2^{\omega(N)-1} \Omega_1(N, q)
	\sum_{p\mid N} \Psi_7(p^{e_p}, p^{s_p}) \log p.
$$
\end{proof}

\begin{lemma}\label{lem:chi_ve}
Let $r\in\Z$. If $\gcd(r,q)>1$ then $\{\chi\}_{\ve}(r)=0$, while if $\gcd(r,q)=1$ then
$$
	\{\chi\}_{\ve}(r)
	=
	\frac{1}{2}
	\bigg\{
	\prod_{p\mid N,p\nmid r} (\Re\chi_p(r)) \Phi_+(p^{e_p}, r)
	+
	(-1)^{\ve} \prod_{p\mid N,p\nmid r} (\{\Re/i\Im\}\chi_p(r)) \Phi_-(p^{e_p}, r)
	\bigg\},
$$
where $\{\Re/i\Im\}\chi_p(r) = \frac{1}{2}(\overline{\chi_p(r)}+\chi_p(-r))$,
and $\Phi_+$ and $\Phi_-$ are given in \eqref{e:Phi}.
\end{lemma}
\begin{proof}
Let
$$
	f_\chi^{\pm}(r)
	=
	\sum_{(m, \chi_1, \chi_2)\in F} \chi_1(\pm 1) \chi_1\chi_2\omega_m(r).
$$
Then, recalling \eqref{e:chi_ve},
$$
	\{\chi\}_{\ve}(r) = \frac{1}{2}\big(f_\chi^+(r) + (-1)^{\ve} f_{\chi}^-(r)\big).
$$
Note that $f_\chi^{\pm}(r)$ is multiplicative.
For a prime $p\mid N$, let
$$
	f_{\chi, p}^{\pm}(r)
	=
	\sum_{f\in S_p} \omega_{p^f}(r)
	\sum_{(\chi_1, \chi_2)\in F_p} \chi_1(\pm 1) \chi_1\chi_2(r).
$$
Then $f_{\chi}^\pm = \prod_{p\mid N} f_{\chi, p}^{\pm}$.

When $p\mid r$, since $\omega_{p^f}(r) = 0$ unless $f=0$,
so $f_{\chi, p}^{\pm}(r)=0$ unless $s_p=0$.
From now on, assume that $p\nmid r$.

Recalling \eqref{e:Sp} and Lemma~\ref{lem:Fp},
$$
	f_{\chi, p}^\pm(r)
	=
	\sum_{f\in S_p}
	\begin{cases}
	\overline{\chi_p(r)} \sum_{\psi \in X_{p^{\min\{f, e_p-f\}}}} \psi (\pm r^2)
	& \text{if } f\geq s_p, \\
	\chi_p(\pm r) \sum_{\psi \in X_{p^{\min\{f, e_p-f\}}}} \psi (\pm r^2)
	& \text{if } f< s_p.
	\end{cases}
$$
Set $\ord_p(0)=\infty$.
Then
$$
	\sum_{\psi \in X_{p^{\min\{f, e_p-f\}}}} \psi (\pm r^2)
	=
	\begin{cases}
	\varphi(p^{\min\{f, e_p-f\}}) & \text{if } \min\{f, e_p-f\} \leq \ord_p(\pm r^2-1), \\
	0 & \text{otherwise.}
	\end{cases}
$$
Note that $\sum_{f=0}^k \varphi(p^f) = p^k$.

Let $u_\pm = \ord_p(\pm r^2-1)$.
When $e_p < 2s_p$, we have
$$
	f_{\chi, p}^{\pm}(r)
	=
	(\chi_p(\pm r) +\overline{\chi_p(r)} )
	\sum_{f=0}^{\min\{e_p-s_p, u_\pm \}} \varphi(p^f)
	=
	(\chi_p(\pm r) +\overline{\chi_p(r)} ) p^{e_p-w_{\pm}(r)}.
$$
When  $e_p \geq 2s_p$, we have
$$
	f_{\chi, p}^\pm(r)
	=
	\chi_p(\pm r) \sum_{f=0}^{\min\{s_p-1, u_\pm\}} \varphi(p^f)
	+
	\overline{\chi_p(r)}
	\bigg( \sum_{f=s_p}^{\min\{\lf\frac{e_p}{2}\rf, u_\pm\}} \varphi(p^{f})
	+
	\sum_{f=0}^{\min\{u_\pm, \lf\frac{e_p-1}{2}\rf\}}\varphi(p^{f})
	\bigg).
$$
For the second term, we have two cases:  $s_p> \min\{\lf\frac{e_p}{2}\rf, u_\pm\}$ or $s_p \leq \min\{\lf\frac{e_p}{2}\rf, u_\pm\}$.

When  $s_p> \min\{\lf\frac{e_p}{2}\rf, u_\pm\}$, since $e_p\geq 2s_p$, we get $u_\pm \leq s_p-1$.
So we have
$$
	f_{\chi, p}^{\pm} (r) = (\chi_p(\pm r) + \overline{\chi_p(r)}) \sum_{f=0}^{u_\pm} \varphi(p^f)
	=
	 (\chi_p(\pm r) + \overline{\chi_p(r)}) p^{u_\pm}.
$$
When $s_p\leq \min\{\lf\frac{e_p}{2}\rf, u_{\pm}\}$, i.e.\ $s_p\leq u_{\pm}$, we get
$$
	\chi_p(\pm r^2) = \chi_p(\pm 1) \chi_p(r)^2 = 1.
$$
We get
$$
	f_{\chi, p}^{\pm}(r)
	=
	(\chi_p(\pm r) + \overline{\chi_p(r)})
	\bigg(p^{\min\{\lf\frac{e_p-1}{2}\rf, u_\pm\}}+ \frac{1}{2}\delta_{2\mid e_p, \frac{e_p}{2} \leq u_\pm } \varphi(p^{\frac{e_p}{2}})\bigg)
$$
So we have
$$
	f_{\chi, p}^{\pm}(r)
	=
	(\chi_p(\pm r) + \overline{\chi_p(r)})
	\begin{cases}
	\frac{1}{2}\big(p^{\lf\frac{e_p-1}{2}\rf} + p^{\lf\frac{e_p}{2}\rf}\big)
	& \text{if } w_\pm(r) \leq\lf\frac{e_p}{2}\rf, \\
	p^{e_p-w_\pm(r)}
	& \text{if } w_\pm(r) > \lf\frac{e_p}{2}\rf.
	\end{cases}
$$
\end{proof}

\begin{proof}[Proof of Proposition \ref{prop:Eis_N}]
The proposition follows from 
\begin{align*}
\Eis(N,\chi;n)=\Eis(\Gamma, \chi^{(0)})+n\Eis(\Gamma, \chi^{(1)}),
\end{align*}
by evaluating $\Eis(\Gamma, \chi^{(0)})$ and $\Eis(\Gamma, \chi^{(1)})$ using
Lemmas \ref{lem:Eis_integral} and \ref{lem:TrPhiGamma_1/2} from the previous section
together with Lemmas \ref{lem:|Fm_ve|log}--\ref{lem:chi_ve} in the present section.
\end{proof}
\subsection{Cuspidal and Continuous contributions $(\Cu+\Eis)(\Gamma, \chi)$}

Finally we conclude the proof of Theorem \ref{thm:STF_N}
by computing the sum of $\Cu(N, \chi; n)$ from Proposition \ref{prop:Cu}
and $\Eis(\Gamma, \chi; n)$ from Proposition \ref{prop:Eis_N}.

For $n=1$, we combine \eqref{e:Cu_N_1} and \eqref{e:Eis_N_1},
and note that
\begin{multline*}
	(\Psi_4 + \Psi_5+ \Psi_6)(p^{e_p}, p^{s_p})
	\\
	=
	e_p\!\left(p^{\lf\frac{e_p}{2}\rf}+p^{\lf\frac{e_p-1}{2}\rf}\right)
	-
	\frac{p^{\lf\frac{e_p}{2}\rf} +p^{\lf\frac{e_p-1}{2}\rf}-2}{p-1}
	+
	p^{s_p-1}
	+
	2\frac{p^{s_p-1}-1}{p-1}
	=
	\Psi_3(p^{e_p}, p^{s_p}).
\end{multline*}
Here $\Psi_3$ is given in \eqref{e:Psi3}.
Hence we obtain \eqref{e:Cu+Eis_N_1}.

For $n=-1$, we combine \eqref{e:Cu_N_-1} and \eqref{e:Eis_N_-1}.
Here one notes that, if $I_\chi=I_{q,4}=1$,
$$
	\chi(\sqrt{-1})\min\{e_2+1, 4\} \prod_{p\mid N, p\equiv_41} \Psi_2 (p^{e_p}, p^{s_p}) \prod_{p\mid N, p\equiv_4-1} 2
	=
	\tPsi_2(N, \chi),
$$
with $\tPsi_2$ as in \eqref{e:tPsi2}.
Recall also that if $I_\chi=1$, i.e.\ if $\chi$ is pure, then $s_2=0$ or $s_2\geq 3$.
Therefore,
$$
	I_\chi 2^{\omega(N)} \frac{1}{2}
	\Big(
	\Omega_1(N, q) (-e_2+2) + \ve_N -2 -\Omega_1(N, q) \Psi_7(2^{e_2}, 2^{s_2})
	\Big)
	=
	I_\chi 2^{\omega(N)}\Omega_2(N, q),
$$
where $\Omega_2$ is given in \eqref{e:Omega2}.
Finally for $p\mid N$ odd, recalling \eqref{e:Psi7},
$$
	\frac{1}{2}\big(\Psi_7(p^{e_p}, p^{s_p}) + e_p\big)
	=
	\max\left\{\frac{1}{2}, s_p\right\}.
$$
Hence we obtain \eqref{e:Cu+Eis_N_-1}.

Recall also that the formulas for $\I(\Gamma, \chi; n)$ and $(\NEl+\El)(\Gamma, \chi; n)$
stated in Theorem \ref{thm:STF_N} were proved in Lemma \ref{lem:I_N}
and Proposition \ref{prop:NEll+Ell_N}, respectively.
Hence the proof of Theorem~\ref{thm:STF_N} is now complete.
\hfill$\square$


\section{Sieving}\label{sec:sieve}
Our goal in this section is to prove Theorem \ref{THMmintf},
by sieving out the contribution from the twist-minimal Hecke eigenforms in 
the trace formula in Theorem \ref{thm:STF_N}.

The sieving is carried out in two steps:
we first sieve for newforms,
and then
sieve for twist-minimal forms among the newforms.
The first step is as in \cite[\S2.3]{BS07}:
for any
Dirichlet character $\chi$ modulo $N$ of conductor $\cond(\chi)$,
any $\lambda>0$ and $n\in\{\pm1\}$,
we have
\begin{align}\label{newformsieve}
	\Tr T_n|_{\Anew_\lambda(\chi)}
	=
	\sum_{M \mid \frac N{\cond(\chi)}} \beta\left(\frac{N/\cond(\chi)}{M}\right)
	\Tr T_n|_{\A_\lambda(\chi|_{M\cond(\chi)})},
\end{align}
where $\beta(m)$ is the multiplicative function given by $\zeta(s)^{-2} = \sum_{m=1}^\infty \beta(m) m^{-s}$.

For the second step, sieving down to $\Amin_\lambda(\chi)$,
we assume that $\chi$ is minimal
as in Definition~\ref{def:minimal}.
Let $S_\chi$ be the set of pairs $\langle M,\psi\rangle$ where $M$ is a (positive) divisor of $N$ and
$\psi$ is a primitive Dirichlet character such that
$\lcm(M,\cond(\psi)\cond(\chi\psi))=N$,
$\cond(\chi\psi^2)\mid M$ and
$\chi\psi^2|_M$ is minimal.
Also let $\sim$ be the equivalence relation on $S_\chi$ defined by
$\langle M,\psi\rangle\sim\langle M',\psi'\rangle$ if and only if $M'=M$
and $\Amin_\lambda(\chi\psi^2|_M)\otimes\opsi
=\Amin_\lambda(\chi\psi'^2|_M')\otimes\opsi'$.
It then follows from Lemma~\ref{lem:twistconductor} and
Lemma~\ref{lem:twistminimal_isom} that for any $\lambda>0$,
we have the direct sum decomposition
\begin{align}\label{newmindecomp}
\Anew_\lambda(\chi)=\bigoplus_{\langle M,\psi\rangle\in S_\chi/\sim}\Bigl(\Amin_\lambda(\chi\psi^2|_M)\otimes\opsi\Bigr),
\end{align}
where $S_\chi/\!\!\sim$ denotes any set of representatives for $S_\chi$ modulo $\sim$.
Hence, also using the fact that $\overline{\psi(n)}=\psi(n)$ for $n\in\{\pm1\}$,
\begin{align}\label{newmindecomp2}
\Tr T_n|_{\Anew_\lambda(\chi)}=\sum_{\langle M,\psi\rangle\in S_\chi/\sim}\psi(n)\Tr T_n|_{\Amin_\lambda(\chi\psi^2|_M)}.
\end{align}
Note that the pairs $\langle M,\psi\rangle$ in $S_\chi$ with $M=N$ form a single equivalence class,
and this class contributes via the term $\Tr T_n|_{\Amin_\lambda(\chi)}$
to the sum in \eqref{newmindecomp2}.

In order to invert the formula \eqref{newmindecomp2},
we first note that the set $S_\chi$ and the relation $\sim$ can be fully described  by local conditions:
writing $\chi=\prod_{p\mid N} \chi_p$ as usual,
and setting $N_p=p^{e_p}$ ($e_p=\ord_p(N)$) and $M_p=p^{\ord_p(M)}$,
we have that a pair $\langle M,\psi\rangle$ lies in $S_\chi$ if and only if
$\langle M_p,\psi_p\rangle$ lies in $S_{\chi_p}$ for each prime $p$,
and furthermore 
$\langle M,\psi\rangle\sim\langle M',\psi'\rangle$ holds if and only if
$\langle M_p,\psi_p\rangle\sim\langle M'_p,\psi'_p\rangle$ for all primes $p\mid N$.
Also the equivalence classes in $S_{\chi_p}$ are easily classified:
if $p=2$ or $e_p\leq1$ or $2\nmid e_p$ or $s_p>1$ then 
\emph{all} elements in $S_{\chi_p}$ are equivalent with $\langle N_p,1\rangle$.
In the remaining case when $p>2$, $e_p\geq2$, $2\mid e_p$ and $s_p\in\{0,1\}$, 
then the set of elements in $S_{\chi_p}$ outside the equivalence class of
$\langle N_p,1\rangle$ equals
\begin{align}\label{PAIRS1}
\{\langle p^{e_p/2},\psi_p\rangle\::\:\cond(\psi_p)=p^{e_p/2},\:\psi_p\neq\ochi_p\}
\qquad\text{if $e_p\geq4$ or $s_p=1$,}
\end{align}
and
\begin{align}\label{PAIRS2}
\{\langle p,\psi_p\rangle\: :\:\cond(\psi_p)=p\}
\:\cup\:\{\langle 1,(\tfrac{\cdot}p)\rangle\}
\qquad\text{if $e_p=2$ and $s_p=0$.}
\end{align}
In the case of \eqref{PAIRS1},
each $\sim$ equivalence class in that set has exactly two elements.
In the case of \eqref{PAIRS2}, the elements $\langle p,(\frac{\cdot}p)\rangle$ and 
$\langle 1,(\frac{\cdot}p)\rangle$ form singleton equivalence classes, while the remaining $p-3$ elements 
group together into equivalence classes with exactly two elements each.

In particular it follows from the above description that 
\eqref{newmindecomp2} can be rewritten as
\begin{align*}
\Tr T_n|_{\Anew_\lambda(\chi)}=
\sum_{\langle M,\psi\rangle\in S_\chi'}
2^{-k(N,M,\psi)}\,\psi(n)\Tr T_n|_{\Amin_\lambda(\chi\psi^2|_M)},
\end{align*}
where $S_\chi'$ is the subset of all $\langle M,\psi\rangle\in S_\chi$
satisfying [$M_p<N_p$ or $\psi_p=1$] for each prime $p\mid N$,
and $k(N,M,\psi)$ is the number of primes $p\mid N$ 
for which $M_p<N_p$
and [$s_p=1$ or $\psi_p\neq(\frac{\cdot}p)$].
This formula can now be inverted as follows:
\begin{align}\label{minsieveformula}
\Tr T_n|_{\Amin_\lambda(\chi)}=
\sum_{\langle M,\psi\rangle\in S_\chi'}
(-1)^{k'(N,M)}
2^{-k(N,M,\psi)}\,
\psi(n)\Tr T_n|_{\Anew_\lambda(\chi\psi^2|_M)},
\end{align}
where $k'(N,M)$ is the number of primes $p\mid N$ for which $M_p<N_p$.

We will apply the sieving in \eqref{newformsieve} and then \eqref{minsieveformula}
to the trace formula in Theorem \ref{thm:STF_N}.
Note that the right hand side of that formula
is a sum of terms $f(N,\chi)$ each of which is multiplicative with respect to $\chi$,
in the sense that 
$f(N, \chi) = \prod_{p\mid N} f(p^{e_p}, \chi_p)$
for any Dirichlet character $\chi$ modulo $N$.
One verifies that this multiplicativity property is preserved by the sieving,
i.e., for any multiplicative function $f(N,\chi)$,
if we define $f^{\new}$ and $f^{\min}$ via
\begin{align}\label{newformsieve2}
f^{\new}(N,\chi)=
	\sum_{M \mid \frac N{\cond(\chi)}} \beta\left(\frac{N/\cond(\chi)}{M}\right)
	f(M\cond(\chi),\chi)
\end{align}
and
\begin{align}\label{minsieveformula2}
f^{\min}(N,\chi)=\sum_{\langle M,\psi\rangle\in S_\chi'}
(-1)^{k'(N,M)}
2^{-k(N,M,\psi)}\,
\psi(n) f^{\new}(M,\chi\psi^2),
\end{align}
then also $f^{\new}$ and $f^{\min}$ are multiplicative.
Hence our task is reduced to computing $f^{\min}(p^{e_p},\chi_p)$ 
for each term $f(N,\chi)$ appearing in the trace formula in Theorem \ref{thm:STF_N}.
Writing $e=e_p$
and $s=s_p$,
we note that \eqref{newformsieve2} implies
\begin{align}\label{fnewp}
	f^{\new}(p^{e}, \chi_p) = \sum_{j=0}^{e-s} \beta(p^{j}) f(p^{e-j}, \chi_p),
\end{align}
and we have, for each $j\geq0$,
$$
	\beta(p^j) = \begin{cases}
	1 & \text{if } j\in \{0, 2\}, \\
	-2 & \text{if } j=1, \\
	0 & \text{otherwise.}
	\end{cases}
$$
Next assume again that $\chi=\prod_p\chi_p$ is minimal as in Definition~\ref{def:minimal}.
It now follows from \eqref{minsieveformula2} and the description of the equivalence classes of
$S_{\chi_p}$ given around \eqref{PAIRS1} and \eqref{PAIRS2}
that, for each odd prime $p\mid N$:
\begin{multline}\label{e:fmin_odd}
f^{\min}(p^{e}, \chi_p)
	=
f^{\mathrm{new}}(p^{e}, \chi_p)-\delta_{e=2,s=0}\bigl(\tfrac{n}{p}\bigr)
\big(f^{\mathrm{new}}(1, 1)+\tfrac12f^{\mathrm{new}}(p, 1)\big)
\\
-\delta_{2\mid e,s\leq1}\cdot\frac12\sum_{\substack{\mathrm{cond}(\psi)=p^{\frac{e}{2}}\\\psi\neq\overline{\chi}_p}}
\psi(n)f^{\mathrm{new}}(p^{\frac{e}{2}},\chi_p\psi^2).
\end{multline}
On the other hand for $p=2$ we have simply
\begin{equation}\label{e:fmin_p=2}
	f^{\min}(2^{e}, \chi_2)
	=
	f^{\new} (2^{e}, \chi_2).
\end{equation}

\subsection{$\I(\Gamma, \chi)$}\label{IMINSIEVE}
Recalling Lemma~\ref{lem:I_N}, by \eqref{e:fmin_odd} and \eqref{e:fmin_p=2},
we have
$$
	\I^{\min}(\Gamma, \chi; n)
	=
	\frac{1+n}{2}
	\frac{M(\chi)}{12}
	\int_\R rh(r) \tanh(\pi r) \; dr.
$$
Here $M(\chi)$ is given in \eqref{eq:Mchi}.
By \cite[p.~141]{BS07},
$$
	\int_\R rh(r) \tanh(\pi r) \; dr = -\int_{-\infty}^\infty \frac{g'(u)}{\sinh(u/2)} \; du,
$$
so we get
$$
	\I^{\min}(\Gamma, \chi; n)
	=
	-
	\frac{1+n}{2}
	\frac{M(\chi)}{12}
	\int_{-\infty}^\infty \frac{g'(u)}{\sinh(u/2)} \; du.
$$

\subsection{$\boldS_p(p^{e}, \chi; t, n)$}
Using the Dirichlet class number formula
and the formula for $\boldS_p(1,1;t,n)$ in \eqref{e:boldSp_e=0_0},
the elliptic and hyperbolic terms in Theorem \ref{thm:STF_N} can be rewritten as
\begin{align*}
	(\NEl+\El)(\Gamma, \chi; n)
	=
	\sum_{\substack{t\in \Z\\ D=t^2-4n\\\sqrt D\notin\Q}}
\biggl(\prod_{p\mid N}\frac{\boldS_p(p^e,\chi_p;t,n)}{1+(p-(\frac dp))\frac{p^{\ord_p\ell}-1}{p-1}}\biggr)
\:L(1,\psi_{D})
\hspace{40pt}
\\
\cdot\begin{cases}
g\!\left(2\log\frac{|t|+\sqrt{D}}2\right)
&\text{if }D>0,\\
\frac{\sqrt{|D|}}{\pi}\int_\R
\frac{g(u)\cosh(u/2)}{4\sinh^2(u/2)+|D|}\,du
&\text{if }D<0,
\end{cases}
\end{align*}
where the character $\psi_D$ is as in Section \ref{ss:twist-minimal_TF}
(cf.\ also \cite[\S2]{BL17a}).
Hence to show that we obtain the corresponding sum in Theorem \ref{THMmintf},
it remains to prove that for each prime $p$ with $e=e_p>0$,
\begin{align}\label{HtnchiWTP}
\frac{\boldS_p^{\min}(p^e,\chi_p;t,n)}{1+(p-(\frac dp))\frac{p^{\ord_p\ell}-1}{p-1}}
=H_{t,n}(\chi_p).
\end{align}

\vspace{5pt}

From now on until further notice, let us assume that \emph{$p$ is odd.}
If $e=s>0$ then 
$\boldS_p^{\min}=\boldS_p^{\new}=\boldS_p$
by \eqref{fnewp} and \eqref{e:fmin_odd},
and one then verifies directly from
\eqref{e:Htn_s=e}
and
\eqref{e:boldSp_odd} (with $h=\max\{2e-1, e\} = 2e-1$)
that \eqref{HtnchiWTP} holds.

Next assume $s\in\{0, 1\}$ and $e>2$.
Then by \eqref{e:boldSp_odd}
(where now $h=\max\{2s-1, e\} =e$)
and \eqref{fnewp}, we get
\begin{multline}\label{e:Spnew}
	\boldS_p^{\new}(p^e, \chi; t, n)
	=
	\boldS_p(p^e, \chi; t, n) - 2\boldS_p(p^{e-1}, \chi; t, n) + \boldS_p(p^{e-2}, \chi; t, n)
	\\
	=
	\chi\!\left(\frac{t+\delta_{2\nmid t} p^s}{2}\right)
	p^{e-3}
	\begin{cases}
	(p-1)
	\bigg( \delta_{2\mid e} \left(p-\left(\frac{d}{p}\right)\right) p^{f-\frac{e}{2}+1}
	+
	\left(\left(\frac{d}{p}\right)-1\right)(p+1)
	\bigg)
	& \text{if } g\geq e, \\
	\left(\left(\frac{d}{p}\right)-1\right) (p-1) (1+\delta_{2\nmid e} p)
	& \text{if } g=e-1, \\
	1-\left(\frac{d}{p}\right)p
	& \text{if } g=e-2, \\
	0 & \text{otherwise.}
	\end{cases}
\end{multline}

Note that for $\alpha\geq 2$,
$$
	\sum_{\substack{\psi\text{ primitive}\\ \cond(\psi)=p^\alpha}}
	\psi(x)
	=
	\varphi(p^{\alpha-1})
	\begin{cases}
	p-1 & \text{if } \alpha \leq \ord_p(x-1), \\
	-1 & \text{if } \alpha = \ord_p(x-1)+1, \\
	0 & \text{otherwise.}
	\end{cases}
$$
To evaluate the last sum in \eqref{e:fmin_odd},
we need the following lemma.
\begin{lemma}
When $\left(\frac{d}{p}\right)=1$,
$n\left(\frac{t}{2}\pm \frac{\sqrt{d}\ell}{2}\right)^2 \equiv_{p^{\alpha}} 1$
if and only if $\ell\equiv_{p^{\alpha}} 0$.
\end{lemma}
\begin{proof}
Set $x \equiv_{p^{\alpha}} \frac{t}{2}\pm \frac{\sqrt{d}\ell}{2}$ and
$y\equiv_{p^{\alpha}} \frac{t}{2}\mp \frac{\sqrt{d}\ell}{2}$.
Then since $t^2-4n=d\ell^2$, we get $xy \equiv_{p^{\alpha}} n$ and $x^2\equiv_{p^{\alpha}} tx-n$.

Assume that $n(tx-n)=nx^2 \equiv_{p^{\alpha}}1$.
Then $ntx \equiv_{p^{\alpha}} 2$.
Multiplying by $y$ on both sides, we get
$$
	t\equiv 2y \equiv t-\sqrt{d}\ell \pmod*{p^{\alpha}}.
$$
So $p^{\alpha}\mid \ell$.

Conversely, assume that $\ell\equiv_{p^{\alpha}}0$.
Then
$$
	nx^2 \equiv ntx-1 \equiv n\frac{t^2}{2} -1  = n\frac{t^2-4n}{2}+1
	\equiv 1\pmod*{p^{\alpha}}.
$$
\end{proof}
By the above lemma, for $e>2$, $2\mid e$, we have
\begin{multline*}
	\sum_{\cond(\psi)=p^{\frac{e}{2}}}
	\psi(n)\,\boldS_p^{\new}(p^{\frac{e}{2}}, \chi\psi^2; t, n) 
	=
	\chi\!\left(\tfrac{t+\delta_{2\nmid t} p^s}{2}\right)
	(p-1) p^{e-3}
	\\
	\times
	\begin{cases}
	2(p-1)p^{f-\frac{e}{2}+1} + \left(\left(\frac{d}{p}\right)-1\right)\big(-2p^{f-\frac{e}{2}+1} + p +1\big)
	& \text{if } g\geq e-1, \\
	-2 & \text{if } g=e-2 \text{ and } \left(\frac{d}{p}\right)=1, \\
	0 & \text{otherwise.}
	\end{cases}
\end{multline*}
By \eqref{e:fmin_odd} and \eqref{e:Spnew} we now obtain, when $s\in\{0,1\}$ and $e>2$, 
$$
	\boldS_p^{\min}(p^e, \chi; t, n)
	=
	\left(\left(\tfrac{d}{p}\right)-1\right)
	\frac{1+\delta_{2\nmid e}}{2}
	\chi\!\left(\tfrac{t+\delta_{2\nmid t} p^s}{2}\right)
	p^{e-3}
	\begin{cases}
	(p-1)(p+1) & \text{if } g\geq e-1, \\
	-\delta_{2\mid e} p - 1 & \text{if } g=e-2, \\
	0 & \text{otherwise.}
	\end{cases}
$$
Hence \eqref{HtnchiWTP} again holds; cf.\ \eqref{e:Htn_s<e}.

We next turn to the case $e\in\{1, 2\}$, $s<e$.
By \eqref{e:fmin_odd}, when $e=1$ and $s=0$, we have
$$
	\boldS_p^{\min}(p, 1; t, n)= \boldS_p^{\new}(p, 1; t, n) = \boldS_p(p, 1; t, n) - 2\boldS_p(1, 1; t, n) = \left(\tfrac{d}{p}\right)-1.
$$
For $e=2$,
\begin{multline*}
	\boldS_p^{\new}(p^2, \chi; t, n)
	\\
	=
	\chi\!\left(\tfrac{t+\delta_{2\nmid t}p^s}{2}\right)
	\begin{cases}
	p^f (p-2-s) + \left(\left(\frac{d}{p}\right)-1\right) \frac{-p^f(p-2-s)+p^2-p-1-s}{p-1}
	& \text{if } g\geq 2, \\
	-1-s & \text{if } g=1, \\
	0 & \text{if } g=0
	\end{cases}
	\\
	+
	\delta_{g=0}
	\bigg\{
	(1-s)\frac{1-\left(\frac{d}{p}\right)}{2}
	-
	\frac{1}{2-s}
	\frac{1+\left(\frac{d}{p}\right)}{2}
	\left(\chi\!\left(\tfrac{t+\sqrt{d}\ell}{2}\right) + \chi\!\left(\tfrac{t-\sqrt{d}\ell}{2}\right)\right)
	\bigg\}.
\end{multline*}
Assuming first $s=1$, $e=2$, we compute:
\begin{multline*}
	\sum_{\substack{\psi\pmod*{p}\\\psi\notin\{\ochi_p,1\}}}
	\boldS_p^{\new}(p, \chi\psi^2; t, n) \psi(n)
	\\
	=
	\begin{cases}
	\chi\!\left(\frac{t+\delta_{2\nmid t} p}{2}\right)
	(p-3)
	\Bigl(2p^f + \bigl(\bigl(\frac{d}{p}\bigr)-1\bigr) \frac{-2p^f + p+1}{p-1}\Bigr)
	& \text{if } g\geq 1, \\
	-2\Bigl(\chi\!\bigl(\frac{t+\sqrt{d}\ell}{2}\bigr) + \chi\!\bigl(\frac{t-\sqrt{d}\ell}{2}\bigr)\Bigr)
	& \text{if } g=0 \text{ and } \bigl(\frac{d}{p}\bigr)=1, \\
	0 & \text{otherwise.}
	\end{cases}
\end{multline*}
Hence by \eqref{e:fmin_odd},
\begin{multline*}
	\boldS_p^{\min}(p^2, \chi; t, n)
	=
	\boldS_p^{\new}(p^2, \chi; t, n) -
	\frac{1}{2}\sum_{\substack{\psi\pmod*{p}\\\psi\notin\{\ochi_p,1\}}}
	\boldS_p^{\new}(p^2, \chi\psi^2; t, n) \psi(n)
	\\
	=
	\begin{cases}
	\chi\!\left(\frac{t+\delta_{2\nmid t} p^s}{2}\right) \left(\left(\frac{d}{p}\right)-1\right)
	\frac{p+1}{2}
	& \text{if } g\geq 1, \\
	0
	& \text{if } g=0.
	\end{cases}
\end{multline*}
Finally for $s=0$, $e=2$, we have
\begin{multline*}
	\left(\tfrac{n}{p}\right)
	\big( \boldS_p(1, 1; t, n) + \tfrac12\boldS_p^{\new}(p, 1; t, n)\big)
	+
	\frac{1}{2} \sum_{\substack{\psi\pmod*{p}\\\psi\neq1}}
	\boldS_p^{\new}(p, \psi^2; t, n) \psi(n)
	\\
	=
	\begin{cases}
	(p-2)p^f + \left(\left(\frac{d}{p}\right)-1\right) \frac{\frac{1}{2}(p^2-3) - p^f(p-2)}{p-1}
	& \text{if } g\geq 1, \\
	-1+\left(\left(\frac{d}{p}\right)-1\right) \frac{\left(\frac{n}{p}\right)-1}{2}
	& \text{if } g=0, \\
	0 & \text{otherwise,}
	\end{cases}
\end{multline*}
and so, by \eqref{e:fmin_odd},
$$
	\boldS_p^{\min}(p^2, 1; t, n)
	=
	\frac{\left(\tfrac{d}{p}\right)-1}{2}
	\begin{cases}
	p-1 & \text{if } g\geq 1, \\
	-\left(\frac{n}{p}\right)-1 & \text{if } g=0, \\
	0 & \text{otherwise.}
	\end{cases}
$$
In all these cases we again see that \eqref{HtnchiWTP} holds; cf.\ \eqref{e:Htn_s<e}.


\medskip

Finally we turn to the case $p=2$.
Since $\chi$ is minimal, we have to consider the following subcases (cf.\ Definition~\ref{def:minimal}
and recall that $s=1$ is impossible when $p=2$):
\begin{align}\label{def:minimalp2expl}
	s=
	\begin{cases}
	e & \text{if } e\geq 2, \\
	\lf\frac{e}{2}\rf & \text{if } e\geq 4, \\
	2 & \text{if } e\geq 5\text{ and } 2\nmid e, \\
	0 & \text{if } 2\nmid e \text{ or } e=2.
	\end{cases}
\end{align}
Recalling \eqref{e:boldSp_even},
for $s\geq 2$ and $s\leq e\leq 2s$, we have
\begin{multline*}
	\boldS_2(2^e, \chi; t, n)
	\\
	=
	\chi\!\left(\tfrac{t}{2}\right) 2^{e-1}
	\begin{cases}
	\delta_{2\nmid d} 2^{\frac{g}{2}-s} (4-\delta_{e=2s})
	+ \left(\left(\frac{d}{2}\right)-1\right) (3-2^{\lf\frac{g}{2}\rf-s}(4-\delta_{e=2s}))
	& \text{if } g\geq 2s+1, \\
	-\delta_{2\nmid d} (4-\delta_{e=2s})
	+ \delta_{e\leq 2s-1}\left(\left(\frac{d}{2}\right)-1\right)
	& \text{if } g=2s, \\
	0 & \text{otherwise}
	\end{cases}
	\\
	+
	\delta_{\left(\frac{d}{2}\right)=1, g\leq 2s-2}
	\left(\chi\!\left(\tfrac{t+\sqrt{d}\ell}{2}\right) + \chi\!\left(\tfrac{t-\sqrt{d}\ell}{2}\right)\right)
	2^{f+\min\{e-s, f\}},
\end{multline*}
For $e=2s+1$, we have
\begin{multline*}
	\boldS_2(2^{e}, \chi; t, n)
	=
	\chi\!\left(\tfrac{t}{2}\right) 2^{e-1}
	\begin{cases}
	\delta_{2\nmid d} 2^{\frac{g}{2}-s+1} + \left(\left(\frac{d}{2}\right) -1\right) (3-2^{\lf\frac{g}{2}\rf-s+1})
	& \text{if } g\geq 2s+2, \\
	0 & \text{otherwise}
	\end{cases}
	\\
	+
	\delta_{\left(\frac{d}{2}\right)=1, g\leq 2s}
	\left(\chi\!\left(\tfrac{t+\sqrt{d}\ell}{2}\right) + \chi\!\left(\tfrac{t-\sqrt{d}\ell}{2}\right)\right)
	2^{2f}.
\end{multline*}
Here $f=\ord_2(\ell)$.
Then by \eqref{e:fmin_p=2}, for $s=\lf\frac{e}{2}\rf$, $e\geq 4$, we
have
\begin{align*}
	\boldS_2^{\min}(2^e, \chi; t, n)
	=
	\boldS_2^{\new}(2^{e}, \chi; t, n)
	=
	\chi\!\left(\tfrac{t}{2}\right)&\left(\left(\tfrac{d}{2}\right)-1\right)
	2^{e-3}
	\begin{cases}
	3 & \text{if } g\geq e+1, \\
	-1-\delta_{2\mid e} 2 & \text{if } g=e, \\
	1-\delta_{2\nmid d}4 & \text{if } g=e-1, 2\nmid e, \\
	0 & \text{otherwise}
	\end{cases}
	\\
	&-
	\delta_{\left(\frac{d}{2}\right)=1, g=e-2}
	\left(\chi\!\left(\tfrac{t+\sqrt{d}\ell}{2}\right)+ \chi\!\left(\tfrac{t-\sqrt{d}\ell}{2}\right)\right)
	2^{g-1}.
\end{align*}
For $e=2s$ and $g=2s-2$, we claim that
$\chi\!\left(\frac{t+\sqrt{d}\ell}{2}\right)+ \chi\!\left(\frac{t-\sqrt{d}\ell}{2}\right)=0$.
Note that $g=4$ cannot occur, so $s\geq 4$.
Moreover, if $a$ is odd, then
$$
	\chi\!\left(\tfrac{t}{2}+a2^{s-2}\right)^2
	=
	\chi\!\left(\tfrac{t^2}{4} + a 2^{s-1}\right)
	=
	-\chi\!\left(\tfrac{t^2}{4}\right) = -\chi(1) =-1,
$$
since $t^2-4n=d\ell^2$ and $g=\ord_2(d\ell^2) = 2s-2>4$, which forces $n=1$.
This implies that
$$
	\chi\!\left(\tfrac{t}{2}+a 2^{s-2}\right) = \pm i.
$$
Since $\left(\frac{t+\sqrt{d}\ell}{2}\right)\left(\frac{t-\sqrt{d}\ell}{2}\right) = 1$,
we get $\chi\!\left(\frac{t+\sqrt{d}\ell}{2}\right)\chi\!\left(\frac{t-\sqrt{d}\ell}{2}\right)=1$,
so $\chi\!\left(\frac{t+\sqrt{d}\ell}{2}\right)=\overline{\chi\!\left(\frac{t-\sqrt{d}\ell}{2}\right)}$.
Then $\chi\!\left(\frac{t+\sqrt{d}\ell}{2}\right)+\chi\!\left(\frac{t-\sqrt{d}\ell}{2}\right)=0$ as claimed.

Now consider the cases for $s\in\{0, 2\}$.
For $e\geq \max\{1, 2s-1\}$, we have
\begin{multline*}
	\boldS_2(2^e, \chi; t, n)
	\\
	=
	\chi\!\left(\tfrac{t}{2}\right)
	\begin{cases}
	\delta_{2\nmid d}
	2^{\frac{g}{2}+\lf\frac{e}{2}\rf}(1+2^{\lc\frac{e}{2}\rc-\lf\frac{e}{2}\rf-1})
	\\
	+
	\left(\left(\frac{d}{2}\right)-1\right)
	(3\cdot 2^{e-1} - 2^{\lf\frac{g}{2}\rf+\lf\frac{e}{2}\rf}(1+2^{\lc\frac{e}{2}\rc-\lf\frac{e}{2}\rf-1}))
	& \text{if } g\geq e+1 \text{ and } 2\nmid g, \\
	& \text{or } g\geq \max\{e, 2s+2\} \text{ and } 2\mid g, \\
	0 & \text{otherwise}
	\end{cases}
	\\
	+
	\delta_{\left(\frac{d}{2}\right)=1, g\leq e-1}
	\left(\chi\!\left(\tfrac{t+\sqrt{d}\ell}{2}\right) + \chi\!\left(\tfrac{t-\sqrt{d}\ell}{2}\right)\right) 2^g.
\end{multline*}
For $e=s=0$, we have
$$
	\boldS_2(1, 1; t, n)
	=
	\delta_{2\nmid d} 2^{\frac{g}{2}} + \left(\left(\tfrac{d}{2}\right)-1\right)(1-2^{\lf\frac{g}{2}\rf}).
$$
For $e=4$ and $s=2$, we have
$$
	\boldS_2^{\min}(2^4, \chi_{-4}; t, n)
	=
	\chi_{-4}\!\left(\tfrac{t}{2}\right) \left(\left(\tfrac{d}{2}\right)-1\right)
	\begin{cases}
	6 & \text{if } g\geq 5, \\
	0 & \text{otherwise.}
	\end{cases}
$$
For $s=0$ and $e\in\{1, 2\}$, we have
$$
	\boldS_2^{\min}(2^e, 1; t,n)
	=
	\begin{cases}
	\left(\frac{d}{2}\right)-1
	& \text{if } g\geq 2, \\
	-2 & \text{if } e=1 \text{ and } g=0, \\
	1 & \text{if } e=2 \text{ and } g=0.
	\end{cases}
$$
For $s\in\{0, 2\}$ and $e \geq\max\{3, 2s+1\}$ and $2\nmid e$,
we have
$$
	\boldS_2^{\min}(2^e, \chi; t, n)
	=
	\chi\!\left(\tfrac{t}{2}\right)
	\left(\left(\tfrac{d}{2}\right)-1\right) 2^{e-3}
	\begin{cases}
	3 & \text{if } g\geq e+1, \\
	\delta_{2\nmid d} 4 -1 & \text{if } g\in\{e, e-1\}, \\
	0 & \text{otherwise.}
	\end{cases}
$$
In all these cases we again see that \eqref{HtnchiWTP} holds; cf.\ \eqref{e:boldS2_min_e>s}.

\subsection{$\Psi_1$, $\Psi_2$, $\tPsi_2$ and $\Psi_3$}
Note that $\Psi_1$ appears only when $n=1$.
\begin{lemma}\label{lem:Psi1_min}
We have
\begin{equation}\label{e:Psi1_min}
	\Psi_1^{\min}(p^e, p^s)
	=
	\delta_{e=s} 2.
\end{equation}
So
\begin{equation}\label{e:Psi1_min_N}
	\Psi_1^{\min}(N, q) = \delta_{N=q} 2^{\omega(N)}.
\end{equation}
\end{lemma}
\begin{proof}
Recalling \eqref{e:Psi1}, when $p$ is odd and $e=s\geq 1$ or
$e\geq\max\{1, 2s\}$ for $s\in\{0, 1\}$, we have
$$
	\Psi_1^{\new}(p^e, p^s)
	=
	\begin{cases}
	p^{\frac{e}{2}-2} (p-1)^2 & \text{if } e\geq \max\{3, 2s+2\} \text{ and } 2\mid e, \\
	p-2-s & \text{if } e=2\text{ and } s\in\{0, 1\}, \\
	2 & \text{if } e=s\geq 1, \\
	0 & \text{otherwise.}
	\end{cases}
$$
Then by \eqref{e:fmin_odd}, we get \eqref{e:Psi1_min}.
When $p=2$ and $s$ is as in \eqref{def:minimalp2expl}, we have by \eqref{e:fmin_p=2}
$$
	\Psi_1^{\min}(2^e, 2^s) = \Psi_1^{\new}(2^e, 2^s)
	= \begin{cases}
	2 & \text{if } e=s\geq 2, \\
	0 & \text{otherwise.}
	\end{cases}
$$
\end{proof}

Recalling \eqref{e:Psi2}, let
$$
	\Psi_2(p^e, \chi) = \frac{\chi(-1)+1}{2} \Psi_2(p^e, p^s).
$$
Note that $\Psi_2$ appears only when $n=1$. But it also occurs in the definition of $\tPsi_2$ when $n=-1$.
\begin{lemma}\label{lem:Psi2_min}
For $p\equiv_41$, we have
\begin{equation}\label{e:Psi2_min}
	\Psi_2^{\min}(p^e, \chi)
	=
	0,
\end{equation}
so that $\Psi_2^{\min}(N,\chi)=\delta_{N=1}$.
\end{lemma}
\begin{proof}
Recall \eqref{e:Psi2}.
When $p$ is odd, $e\geq \max\{1, 2s\}$, $s\in\{0, 1\}$ or $e=s\geq 1$,
we have
$$
	\Psi_2^{\new}(p^e, \chi)
	=
	\delta_{e=2, s=0}.
$$
Then by \eqref{e:fmin_odd} and \eqref{e:fmin_p=2}, we get \eqref{e:Psi2_min}.
\end{proof}

Note that $\tPsi_2$ appears only when $n=-1$.
\begin{lemma}\label{lem:tPsi2_min}
We have
$$
	\tPsi_2^{\min}(p^e, \chi) = 0,
$$
so that $\tPsi_2^{\min}(N,\chi)=\delta_{N=1}$.
\end{lemma}
\begin{proof}
Recalling \eqref{e:tPsi2}, for $p\not\equiv_41$ and $s=0$, we have
$$
	\tPsi_2^{\new}(p^e, 1)
	=
	\begin{cases}
	-1 & \text{if } p\equiv_4-1 \text{ and } e=2, \\
	0 & \text{otherwise.}
	\end{cases}
$$
For $p\equiv_41$, by \eqref{e:fmin_odd}, \eqref{e:fmin_p=2} and \eqref{e:Psi2_min},
we have $\tPsi_2^{\min}(p^e, \chi)=0$.
\end{proof}

Note that $\Psi_3$ appears only for $n=1$.
\begin{lemma}\label{lem:Psi3_min}
When $p$ is odd, we have
$$
	\Psi_3^{\min}(p^e, p^s)
	=
	\begin{cases}
	4e-1 & \text{if } e=s>0, \\
	2 & \text{if } e=1 \text{ and } s=0, \\
	\frac{1}{2}(p-1+2s) & \text{if } e=2 \text{ and } s\in\{0, 1\}, \\
	p^{\lf\frac{e-3}{2}\rf} (p-1) \frac{\delta_{2\mid e} p+ \delta_{2\nmid e} 3+1 }{2}
	& \text{if } e\geq 3 \text{ and } s\in \{0, 1\}.
	\end{cases}
$$
For $p=2$, we have
$$
	\Psi_3^{\min}(2^e, 2^s)
	=
	\begin{cases}
	4e-1 & \text{if } e=s\geq 2, \\
	2 & \text{if } e\in\{1, 3\} \text{ and } s=0, \\
	1 & \text{if } e=2 \text{ and } s=0, \\
	3\cdot 2^{\frac{e}{2}-2} & \text{if } e=2s\geq 4, \\
	2^{\frac{e-1}{2}} & \text{if } e\geq 5, 2\nmid e \text{ and } s\in\{0, 2, \frac{e-1}{2}\}.
	\end{cases}
$$
Thus, combining with \eqref{e:Psi1_min},
\begin{multline*}
	\sum_{p\mid N} \Psi_3^{\min}(p^{e_p}, p^{s_p})
	\prod_{p'\mid N, p'\neq p} \Psi_1^{\min}({p'}^{e_{p'}}, {p'}^{s_{p'}}) \log p
	\\
	=
	\begin{cases}
	2^{\omega(N)} \big(\log N^2 - \frac{1}{2} \sum_{p\mid N} \log p\big)
	& \text{if } N=q,\\
	2^{\omega(N)-1} \Psi_3^{\min}(p^{e_p}, p^{s_p}) \log p
	& \text{if } \frac{N}{q} = p^{e_p-s_p}>1, \\
	0 & \text{otherwise.}
	\end{cases}
\end{multline*}
\end{lemma}
\begin{proof}
Recalling \eqref{e:Psi3}, for $e=s>0$ or $e\geq \max\{1, 2s\}$ for $s\in\{0, 1, 2\}$,
we have
$$
	\Psi_3^{\new}(p^e, p^s)
	=
	\begin{cases}
	4e-1 & \text{if } e=s>0, \\
	p^{\frac{e}{2}-2} (p-1) (e(p-1)+1)
	& \text{if } e\geq \max\{3, 2s+2\} \text{ and } 2\mid e, \\
	2p^{\frac{e-3}{2}} (p-1)
	& \text{if } e\geq \max\{3, 2s+2\} \text{ and } 2\nmid e, \\
	2p-3 & \text{if } e=2 \text{ and } s=0, \\
	2 & \text{if } e=1 \text{ and } s=0, \\
	2(p-1) & \text{if } e=3 \text{ and } s=1, \\
	2(p-2) & \text{if } e=2 \text{ and } s=1, \\
	e-1 & \text{if } p=2, e\in\{4, 5\} \text{ and } s=2, \\
	\end{cases}
$$
For $n=1$, by applying \eqref{e:fmin_odd} and \eqref{e:fmin_p=2},
we prove the lemma.
\end{proof}

\subsection{$\Phi_\pm$}

Let us define
$$
	\Phi_\pm (p^{e}, \chi; m)
	=\begin{cases}
        1&\text{if }s=0,\: p\mid m
        \\[3pt]
        \frac12\bigl(\overline{\chi(m)} + \chi(\pm m)\bigr)\Phi_\pm(p^e, m)&\text{otherwise.}
        \end{cases}
$$
Then in the trace formula in Theorem \ref{thm:STF_N},
the sums over $m$ appearing in the last lines of \eqref{e:Cu+Eis_N_1} (for $n=1$) and \eqref{e:Cu+Eis_N_-1} (for $n=-1$), 
can be expressed as
\begin{align*}
2\sum_{m=2}^\infty \frac{\Lambda(m)}m\biggl(\prod_{p\mid N}\Phi_n(p^e,\chi_p;m)\biggr)\,g(2\log m).
\end{align*}
Hence to show that we obtain
the corresponding sum in 
Theorem \ref{THMmintf},
we have to prove that for each prime $p\mid N$,
writing now $\chi$ in place of $\chi_p$,
\begin{align}\label{e:Phi_min}
\Phi_n^{\min}(p^e,\chi;m)=\Phi_{m,n}(\chi)
=\begin{cases}
\overline{\chi(m)}+\chi(nm)&\text{if }s=e,\\
-1&\text{if }p\mid m,\: e=1\text{ and }s=0,\\
0&\text{otherwise}
\end{cases}
\end{align}
(cf.\ \eqref{Phimndef}).

By \eqref{e:Phi} we have, for $e>0$,
$$
\Phi_\pm(p^e, \chi;m)
	=
	\begin{cases}
	\delta_{s=0} & \text{if } p\mid m, \\[3pt]
	\bigl(\overline{\chi(m)} + \chi(\pm m)\bigr)\,p^{e-w_\pm} & \text{if } p\nmid m,\: e<2w_\pm, \\[3pt]
	\bigl(\overline{\chi(m)} + \chi(\pm m)\bigr)\,\frac12(p^{\lfloor e/2\rfloor}+p^{\lfloor (e-1)/2\rfloor}) 
& \text{if } p\nmid m,\: e\geq2w_\pm,
	\end{cases}
$$
where $w_\pm = w_\pm (m) = \max\{s, e-\ord_p(\pm m^2-1)\}$.
Let $f_\pm = \ord_p(\pm m^2-1)$.
For $s=0$, we compute
$$
	\Phi_\pm^{\new}(p^e, \chi; m)
	=
	\begin{cases}
	-1 & \text{if } p\mid m \text{ and } e=1, \\
	\begin{cases}
	-1 & \text{if } e=2 \text{ and } f_\pm = 0, \\
	-p^{f_\pm}+p^{f_\pm -1} & \text{if } e=2f_\pm +2 \text{ and } f_\pm \geq 1, \\
	p^{\frac{e}{2}-2}(p-1)^2 & \text{if } 3\leq e \leq 2f_\pm \text{ and } 2\mid e, \\
	p-2 & \text{if } e=2 \text{ and } f_\pm \geq 1,
	\end{cases}
	& \text{if } p\nmid m, \\
	0 & \text{otherwise.}
	\end{cases}
$$
For $s=1$, $p$ odd and $e\geq 2$, and assuming $p\nmid m$, we compute
\begin{multline*}
	\Phi_\pm^{\new}(p^e, \chi; m)
	\\
	=
	\frac{1}{2}\big(\overline{\chi(m)}+\chi(\pm m)\big)
	\begin{cases}
	-2 & \text{if } e=2 \text{ and } f_\pm =0, \\
	\begin{cases}
	p-3 & \text{if } e=2, \\
	p^{\frac{e}{2}-2} (p-1)^2 & \text{if } 4\leq e \leq 2f_\pm \text{ and } 2\mid e, \\
	-p^{f_\pm}+p^{f_\pm -1} & \text{if } e=2f_\pm +2,
	\end{cases}
	& \text{if } f_\pm \geq 1, \\
	0 & \text{otherwise. }
	\end{cases}
\end{multline*}
For $p$ odd, using the above formulas together with \eqref{e:fmin_odd}
and the fact that $\Phi_\pm^{\new}=\Phi_\pm$
when $s=e>0$,
the desired result \eqref{e:Phi_min} now follows 
by a direct computation in each case.

Next assume $p=2$.
Recalling \eqref{e:fmin_p=2} and \eqref{def:minimalp2expl}, we consider the following cases.
For $e\geq 5$, $2\nmid e$ and $s\in \{0, 2\}$, we get
$$
	\Phi_\pm^{\min}(2^e, \chi ;m) = 0,
$$
and for $s=e\geq2$ we get
$$
	\Phi_\pm^{\min}(2^e, \chi ;m) = \overline{\chi(m)}+\chi(nm).
$$
For $s\geq 2$ and $e\in\{2s, 2s+1\}$, we obtain
$$
	\Phi_{\pm}^{\min}(2^e, \chi; m)
	=
	\frac{1}{2}\big(\overline{\chi(m)}+\chi(\pm m)\big)
	\begin{cases}
	-2^{f_\pm} & \text{if } 2\nmid m, e=2s \text{ and } f_\pm = s-1, \\
	0 & \text{otherwise.}
	\end{cases}
$$
For $f_\pm = s-1$, we have $m^2 = \pm 1+a \cdot 2^{s-1}$ for an odd integer $a$.
Then
$$
	\chi(m)^2 = \chi(\pm 1+a\cdot 2^{s-1}) =-\chi(\pm 1).
$$
For $\Phi_+$, we get $\chi(m)=\pm i$ and $\overline{\chi(m)}+\chi(m)=0$.
For $\Phi_-$, if $\chi$ is even then $\chi(m)=\pm i$.
If $\chi$ is odd then $\chi(m) = \pm 1$.
For either case we get $\overline{\chi(m)}+\chi(-m)=0$.
So finally we conclude that $\Phi_\pm^{\min}(2^e, \chi; m)=0$ for $e\in\{2s, 2s+1\}$.

For $s=0$ and $e\in\{1, 2, 3\}$, we get
$$
	\Phi_\pm^{\min}(2^e, \chi; m) =
	\begin{cases}
	-1 & \text{if } 2\mid m \text{ and } e=1, \\
	& \text{or } 2\nmid m, e=2 \text{ and } f_\pm=0, \\
	0 & \text{otherwise.}
	\end{cases}
$$
This formula agrees with \eqref{e:Phi_min},
since in fact $f_{\pm}=\ord_p(\pm m^2-1)\geq1$ must hold when $p=2$ and $2\nmid m$.

Now \eqref{e:Phi_min} has been proved in all cases.

\subsection{$\Omega_1$ and $\Omega_2$}\label{Omega12MINSIEVE}
For $j\in \{1, 2\}$, define
$$
	\tOmega_j(p^e, \chi)
	=
	\begin{cases}
	2\Omega_j(N, q) & \text{if } p=2, e\geq 1 \text{ and } \chi \text{ even}, \\
	\Omega_j(N, q) & \text{if } p=2 \text{ and } e=0, \\
	2 & \text{if } p \text{ odd}, e\geq 1 \text{ and } \chi \text{ even}, \\
	1 & \text{if } p \text{ odd } \text{ and } e=0, \\
	0 & \text{otherwise.}
	\end{cases}
$$
Here $\Omega_j$ is given in \eqref{e:Omega1} and \eqref{e:Omega2}.
Then
$$
	\tOmega(N, \chi) = \prod_{p\mid N} \tOmega_j(p^{e_p}, \chi_p) = I_\chi 2^{\omega(N)} \Omega_j(N, q).
$$
\begin{lemma}\label{lem:tOmega_min}
When $p$ is odd and $n=-1$, we have
\begin{equation}\label{e:tOmega_min_odd}
	\tOmega_j^{\min}(p^e, \chi) =
	\delta_{e=s\geq 1} (1+\chi(-1)).
\end{equation}
When $p=2$ and $n=-1$, we have
\begin{equation}\label{e:tOmega1_min_p=2}
	\tOmega_1^{\min}(2^e, \chi) = \delta_{e=s\geq 2} (1+\chi(-1))
\end{equation}
and
\begin{equation}\label{e:tOmega2_min_p=2}
	\tOmega_2^{\min}(2^e, \chi)
	=
	\begin{cases}
	\frac{1}{2} & \text{if } e=0, \\
	-2e & \text{if } e=s\geq 3 \text{ and } \chi\text{ even}, \\
	-\frac{3}{2} & \text{if } e=1 \text{ and }s=0, \\
	-\frac{1}{2} & \text{if } e\in \{2, 3\} \text{ and } s=0, \\
	0 & \text{otherwise.}
	\end{cases}
\end{equation}
\end{lemma}
\begin{proof}
When $p$ is odd, for $e=s\geq 1$ or $e\geq \max\{1, 2s\}$, we have
$$
	\tOmega_j^{\new}(p^e, \chi) =
	\begin{cases}
	2 & \text{if } e=s>0 \text{ and } \chi \text{ even}, \\
	-2 & \text{if } e=2, s=1 \text{ and } \chi \text{ even}, \\
	-1 & \text{if } e=2 \text{ and } s=0, \\
	0 & \text{otherwise.}
	\end{cases}
$$
By \eqref{e:fmin_odd}, for $j\in \{1, 2\}$ and $n=-1$, we get \eqref{e:tOmega_min_odd}.

When $p=2$, for $j\in\{1, 2\}$, by \eqref{e:fmin_p=2}, since $\tOmega_j^{\min}(2^e, \chi) = \tOmega_j^{\new}(2^e, \chi)$, we get \eqref{e:tOmega1_min_p=2} and \eqref{e:tOmega2_min_p=2}.
\end{proof}

\subsection{Concluding the proof of Theorem \ref*{THMmintf}}
It follows from \eqref{newformsieve} and \eqref{minsieveformula}
that for any minimal character $\chi\pmod*{N}$, any $n\in\{\pm1\}$ and any pair of test functions of trace class $(g,h)$,
we have
\begin{align}\label{THMmintfpf2}
\sum_{\lambda>0}&\tr T_n|_{\Amin_\lambda(\chi)}
h\Bigl(\sqrt{\lambda-\tfrac14}\Bigr)=
\sum_{\langle M,\psi\rangle\in S_\chi'}
(-1)^{k'(N,M)}
2^{-k(N,M,\psi)}\,
\psi(n)\Tr^{\new}\bigl(M,\chi\psi^2\bigr),
\end{align}
where
\begin{align}\label{THMmintfpf1}
\Tr^{\new}(N',\chi')
=
\sum_{M \mid \frac {N'}{\cond(\chi')}} \beta\left(\frac{N'/\cond(\chi')}{M}\right)
\sum_{\lambda>0}\Tr T_n|_{\A_\lambda(\chi'|_{M\cond(\chi')})} h\Bigl(\sqrt{\lambda-\tfrac{1}{4}}\Bigr).
\end{align}
Now the inner sum in \eqref{THMmintfpf1} can be evaluated by using Theorem \ref{thm:STF_N}
and compensating for the possible contribution from the Laplace eigenvalue $\lambda=0$.
Specifically, we have $\A_0(\chi)=\C$ if $\chi$ is the trivial character, otherwise $\A_0(\chi)=\{0\}$;
therefore $\lambda=0$ contributes with a term
$h(\frac i2)$ in the left hand side of \eqref{e:STF_N}
if and only if $n=1$ and $\chi$ is trivial.
This means that for $n=1$ and $\chi'$ trivial,
when using the right hand side of \eqref{e:STF_N} 
to evaluate \eqref{THMmintfpf1},
we need to compensate by subtracting a term
$(\sum_{M\mid N'}\beta(M))h(\frac i2)=\mu(N')h(\frac i2)$
from the resulting formula.
This compensation has an effect in the right hand side of \eqref{THMmintfpf2}
only if $n=1$, $\chi$ is trivial and $\ord_p(N)\leq2$ for all $p$;
and if $\ord_p(N)=2$ for at least one $p$ then the compensations cancel each other out;
hence it is only for $N$ squarefree that the net effect is nonzero.
Using $h(\frac i2)=\int_{\R}g(u)\cosh(u/2)\,du$ we now see that the total contribution from 
these compensations is exactly the term
$-\delta_{n=1}\mu(\chi)\int_{\R}g(u)\cosh(u/2)\,du$ appearing in the formula of
Theorem \ref{THMmintf}.

Taking the results of Sections \ref{IMINSIEVE}--\ref{Omega12MINSIEVE} 
into account, we now see that the formula in Theorem~\ref{THMmintf} 
follows by using \eqref{THMmintfpf1} and \eqref{e:STF_N} to evaluate \eqref{THMmintfpf2}.
Hence the proof of Theorem~\ref{THMmintf} is now complete.
\hfill$\square$

We are grateful to the referee for suggesting a remark along the following lines.
\begin{remark}
As is discussed in \cite[Sec.\ 8]{You19},
the newform part of the continuous part of the spectrum of
the Laplace operator on
$L^2(\Gamma_0(N)\bsl \HH, \chi)$
is spanned by the Eisenstein series
$E_{\chi_1}^{\chi_2}(q_2z,\frac12+it)$
(cf.\ Sec.\ \ref{EISsec})
with $\langle\chi_1,\chi_2\rangle$ running through all pairs of primitive Dirichlet
characters satisfying $q_1q_2=N$ (with $q_j=\cond(\chi_j)$)
and $\cond(\chi\chi_2\ochi_1)=1$.

It is natural to define the twist of
the Eisenstein series $E_{\chi_1}^{\chi_2}$ with an arbitrary primitive Dirichlet character
$\psi$ to be $E_{\chi_1'}^{\chi_2'}$, where $\chi_1'$ and $\chi_2'$ are the primitive Dirichlet characters corresponding to
the products $\chi_1\psi$ and $\chi_2\psi$, respectively.
(This is most naturally seen in the language of automorphic representations of the adele group of $\GL_2$,
where twisting by $\psi$ simply corresponds to multiplying any automorphic representation by
the 1-dimensional representation $\psi\circ\det$.)
It then follows that the 
\textit{twist-minimal part} of the 
continuous spectrum 
is spanned by
$E_{\chi_1}^{\chi_2}(q_2z,\frac12+it)$
with $\langle\chi_1,\chi_2\rangle$ running through the set described above but further restricted
by the condition $\gcd(q_1,q_2)=1$.
This nice and simple description
of the twist-minimal part of the continuous spectrum 
is likely to play an important role in
a potential future more direct proof of the twist-minimal trace formula 
avoiding the intermediate steps of Theorem \ref{thm:STF_N} and  sieving.
\end{remark}

\section{Artin representations and $\Gamma(N)$}\label{sec:galois}
In this section we prove two lemmas that restrict the computations of
twist-minimal spaces needed to prove Theorem~\ref{thm:SEC} for
$\Gamma(N)$ and Theorem~\ref{thm:Artin}.
\begin{lemma}\label{lem:Gammaadmissible}
For any $N\in\Z_{>0}$, the Selberg eigenvalue conjecture holds
for $\Gamma(N)$ if and only if $\Amin_\lambda(\chi)=\{0\}$
for all $\lambda\in(0,\tfrac14)$ and
$\chi\pmod*{M}$ with $\lcm(M,\cond(\chi)^2)\mid N^2$.
\end{lemma}
\begin{proof}
Given $\lambda>0$, let $\A_\lambda(N)$ denote the space of
Maass cusp forms of eigenvalue $\lambda$ that are invariant under the
action of $\Gamma(N)$. Then, as shown in \cite[\S3.5]{Hum18},
we have the isomorphism
$$
\A_\lambda(N)\cong\bigoplus_{\chi\pmod*{N}}\A_\lambda(\chi|_{N^2}),
$$
where to a given $f\in\A_\lambda(\chi|_{N^2})$ we associate
the function $z\mapsto f(N^{-1}z)$, which is an element of $\A_\lambda(N)$.
In turn, for a given pair $\lambda>0$ and $\chi\pmod*{N}$, we have
$$
\A_\lambda(\chi|_{N^2})=\sum_{\substack{M,d\in\Z_{>0}\\
\cond(\chi)\mid M\mid N^2\\d\mid\frac{N^2}{M}}}
\ell_{M,d}\,\Anew_\lambda(\chi|_M),
$$
in the notation of \S\ref{sec:prelim}.

By Lemma~\ref{lem:twistconductor}, the space $\Anew_\lambda(\chi|_M)$
is spanned by the Hecke eigenforms $f\otimes\psi$, where $f$ is twist
minimal and $M=\lcm(\cond(f),\cond(\psi)\cond(\chi\overline{\psi}))$.
Put $M'=\cond(f)$ and $\chi'=\chi\overline{\psi}^2|_{M'}$, so that
$f\in\Amin_\lambda(\chi')$.
Since $M\mid N^2$ and $\cond(\chi)\mid N$,
we have $M'\mid N^2$ and $\cond(\chi')\mid N$.

Thus, if $\Amin_\lambda(\chi')=\{0\}$ whenever $\lambda<\frac14$
and $\chi'\pmod*{M'}$ satisfies $M'\mid N^2$ and $\cond(\chi')\mid
N$ then the Selberg eigenvalue conjecture holds for $\Gamma(N)$.
Conversely, we have $\Amin_\lambda(\chi')\subseteq\Anew_\lambda(\chi')$
so the Selberg conjecture for $\Gamma(N)$ implies that $\Amin_\lambda(\chi')=\{0\}$
for $\lambda<\frac14$.
\end{proof}

\begin{lemma}\label{lem:Galoisadmissible}
Let $\rho:\Gal(\overline{\Q}/\Q)\to\GL_2(\C)$ be a nondihedral,
irreducible Artin representation of conductor $N$, and let
$\chi=\prod_{p\mid N}\chi_p$ be the Dirichlet character associated to
$\det\rho$ via class field theory.
If $p\mid N$ is a prime such
that $\ord_p{N}\in\{1,\ord_p\cond(\chi)\}$, then
$\chi_p$ has order $2$, $3$, $4$ or $5$.
Further,
if $p$ and $q$ are two such primes then $\chi_p\chi_q$ cannot have order
$20$.
\end{lemma}
\begin{proof}
Let $\rho_p$ denote the restriction of $\rho$ to
$\Gal\bigl(\overline{\Q}_p/\Q_p\bigr)$.  Then $\rho_p$ factors through
$G=\Gal(L/\Q_p)$ for some finite extension $L/\Q_p$. Let
$G_i$, $i=0,1,2,\ldots$, denote the ramification subgroups of $G$, with
$G_0$ the inertia group. Then $\rho_p$ and $\det\rho_p$ have conductor
exponents
$$
e=\frac1{\#G_0}\sum_{i\ge0}\sum_{g\in G_i}\bigl(2-\tr\rho_p(g)\bigr)
\quad\text{and}\quad
s=\frac1{\#G_0}\sum_{i\ge0}\sum_{g\in G_i}\bigl(1-\det\rho_p(g)\bigr),
$$
respectively.
Note that the average of $\tr\rho_p$ over $G_i$ is the number of copies
of the trivial representation in $\rho_p|G_i$. If this is nonzero then
$\rho_p|G_i\cong\det\rho_p|G_i\oplus1$, from which it follows that the $i$th
terms of the two sums above are the same. If $\rho_p|G_i$ does not contain
the trivial representation then
$$
\frac1{\#G_i}\sum_{g\in G_i}\bigl(2-\tr\rho_p(g)\bigr)=2
>1\ge\frac1{\#G_i}\sum_{g\in G_i}\bigl(1-\det\rho_p(g)\bigr).
$$
Thus, the $i$th term of the formula for $e$ is always $\ge$ the $i$th
term of the formula for $s$, with equality if and only if $\rho_p|G_i$
contains the trivial representation.
If $e\in\{1,s\}$ then equality must hold for every term; in particular,
$\rho_p|G_0$ contains the trivial representation, so
that $\rho_p|G_0\cong\det\rho_p|G_0\oplus1$.
When $e=1$, this in turn implies that $s=1$, so $\det\rho_p|G_0$ is nontrivial.

Let $\bar\rho_p$ denote the composition of $\rho_p$ with the
canonical projection $\GL_2(\C)\to\PGL_2(\C)$. Then when $e\in\{1,s\}$, the
natural maps $\rho_p(G_0)\to\det\rho_p(G_0)$ and $\rho_p(G_0)\to\bar\rho_p(G_0)$
are isomorphisms. Since $\det\rho_p(G_0)$ is a nontrivial cyclic subgroup of
$\C^\times$ and $\bar\rho_p(G_0)$ is a subgroup of $A_4$, $S_4$ or
$A_5$, it follows that $\det\rho_p(G_0)\cong\bar\rho_p(G_0)$ is cyclic of
order $2$, $3$, $4$ or $5$. Since the Dirichlet character $\chi_p$
associated to $\det\rho_p$ is determined by $\det\rho_p|G_0$, they
have the same order, which implies the first claim.

Finally, order $4$ (resp.\ $5$) can only occur when $\rho$ is octahedral
(resp.\ icosahedral). These possibilities are mutually exclusive, whence
the second claim.
\end{proof}
In the $A_5$ case, we may also take advantage of the fact that
icosahedral representations occur in Galois-conjugate pairs that
are not twist equivalent. Thus, assuming Artin's conjecture, we can
still rule out the existence of an icosahedral representation of a
given conductor when our computation accommodates one representation
(in total over all characters, modulo twist equivalence) but not two.
We used this trick to rule out icosahedral representations with conductor
$N\in\{2221,2341,2381,2529,2799\}$.

\section{Numerical remarks}\label{sec:numerics}
To prove Theorems \ref{thm:SEC} and \ref{thm:Artin}, we applied the
numerical method described in \cite[\S4]{BS07}. Briefly, we consider
test functions of the form
$$
h(r)=\left(\sinc^2\!\left(\frac{\delta{r}}2\right)
\sum_{j=0}^{M-1}x_j\cos(j\delta{r})\right)^2,
$$
where
$$
\sinc{r}:=\begin{cases}
\sin(r)/r&\text{if }r\ne0,\\
1&\text{if }r=0,
\end{cases}
$$
$\delta=X/2M$ for $X,M\in\Z_{>0}$ and
$x_0,\ldots,x_{M-1}\in\R$ are arbitrary.
For each $\chi\pmod*{N}$ and $\epsilon\in\{0,1\}$, set
$$
m_\chi=\#\{\psi\pmod*{N}:\psi(-1)=1,\;\psi^2=1,\;
\cond(\psi)\cond(\chi\psi)\mid N\}
$$
and let $n_{\chi,\epsilon}$ be a lower bound for the number of
twist-minimal forms of character $\chi$ and parity $\epsilon$ arising
from Artin representations; we use \texttt{PARI/GP} \cite{PARI} to
compute the contribution from dihedral representations, as described in
\cite[\S3.2]{BS07} (see also the source code at \cite{code}), and the
data from Table~\ref{tab:Artin} for the rest. Then the quantity
$$
Q_{\chi,\epsilon}(x_0,\ldots,x_{M-1})
:=\frac1{m_\chi}\left[
\sum_{\lambda>0}\tr\tfrac12(T_1+(-1)^\epsilon T_{-1})|_{\Amin_\lambda(\chi)}
h\Bigl(\sqrt{\lambda-\tfrac14}\Bigr)-n_{\chi,\epsilon}h(0)\right],
$$
is a positive-definite quadratic form in the $x_j$. By standard
trigonometric identities, the matrix of $Q_{\chi,\epsilon}$ is
determined from the traces of $\sinc^4(\delta{r}/2)\cos(j\delta{r})$
for $0\le j\le2M-2$. We apply the trace formula to compute these,
and then minimize $Q_{\chi,\epsilon}$ with respect to the constraint
$\sum_{j=0}^{M-1}x_j=1$.

In every case, it turned out that the criterion from \cite[\S4.3]{BS07}
applied, so that the optimal test function $h$ satisfied $h(r)\ge1$ for
$r\in i\R$. As explained in \cite[\S3.4]{BS07}, every non-CM form occurs
with multiplicity $m_\chi$. Thus, since the CM forms satisfy Selberg's
conjecture,\footnote{For squarefree $N$, we actually remove the
contribution from CM forms, following \cite[\S3.3]{BS07}. With more work
one could generalize that approach to arbitrary $N$, but it turns out not
to be necessary for our applications.}
whenever the resulting minimal value of $Q_{\chi,\epsilon}$ is
less than $1$, we deduce both the Selberg conjecture and the completeness
of the list of nondihedral Artin representations for twist-minimal forms
of character $\chi$.

To ensure the accuracy of our numerical computations, we used the
interval arithmetic library \href{http://arblib.org/}{\texttt{Arb}}
\cite{Joh17} throughout.
To handle the integral terms of the trace formula, for each basis function
we first computed $\int_0^\infty g'(u)\log{u}\,du$ symbolically, which
allowed us to replace $\log(\sinh(u/2))$ and $\log(\tanh(u/4))$ by the
real-analytic functions
$\log(\sinh(u/2)/u)$ and $\log(\tanh(u/4)/u)$, respectively. Thus, in
every integral term, the integrand agrees with an analytic function
on each interval $[j\delta,(j+1)\delta]$. After applying a suitable affine
transformation to replace the interval by $[-1,1]$, we use the
following rigorous numerical quadrature estimate of Molin \cite{Mol10}:
\begin{lemma}[Molin]
Let $f$ be an analytic function on an open neighborhood of
$D=\{z\in\C:|z|\le 2\}$. Then for any $n\ge1$ we have
$$
\left|\int_{-1}^1 f(x)\,dx-\sum_{k=-n}^na_kf(x_k)\right|
\le\exp(4-5/h)\sup_{z\in\partial{D}}|f(z)|,
$$
where $h=\log(5n)/n$, $a_k=\frac{h\cosh(kh)}{\cosh^2(\sinh(kh))}$
and $x_k=\tanh(\sinh(kh))$.
\end{lemma}
Note that the error term decays exponentially in $n/\log{n}$.
To obtain a bound for $|f|$ on $\partial D$, we write
$\partial D=\bigcup_{j=0}^{n-1}\{2e(\theta):\theta\in[j/n,(j+1)/n)\}$
and use interval arithmetic to bound $|f(2e(\theta))|$ on each segment.

Using the algorithm from \cite{BBJ}, we computed the class numbers of
$\Q(\sqrt{t^2\pm4})$ for all $t\le e^{20}$, which enables us to take $X$
as large as $40$ in the above. Taking $M=200$, various $X\le 40$ and $\chi$ as
indicated by Lemmas~\ref{lem:Gammaadmissible} and
\ref{lem:Galoisadmissible} sufficed to prove Theorems~\ref{thm:SEC} and
\ref{thm:Artin}.

\thispagestyle{empty}
{\footnotesize
\nocite{*}
\bibliographystyle{amsalpha}
\providecommand{\bysame}{\leavevmode\hbox to3em{\hrulefill}\thinspace}
\providecommand{\MR}{\relax\ifhmode\unskip\space\fi MR }
\providecommand{\MRhref}[2]{%
  \href{http://www.ams.org/mathscinet-getitem?mr=#1}{#2}
}
\providecommand{\href}[2]{#2}

}
\end{document}